\documentclass[a4paper,10pt,british]{article}

\usepackage[T1]{fontenc}
\usepackage[utf8]{inputenc}

\usepackage{amsmath}
\usepackage{amssymb}
\usepackage{mathtools}

\usepackage{lmodern}

\usepackage{babel}
\usepackage{csquotes}
\usepackage[babel]{microtype}
\usepackage[shortcuts]{extdash}

\usepackage[backend=biber,backref=true,doi=true,maxnames=100,sorting=nyt]{biblatex}
\renewbibmacro{in:}{}

\numberwithin{equation}{subsection}

\usepackage{amsthm}
\theoremstyle{plain}
\newtheorem{assertion}{Assertion}[subsection]

\newtheorem{corollary}[assertion]{Corollary}

\newtheorem{lemma}[assertion]{Lemma}
\newtheorem{proposition}[assertion]{Proposition}
\newtheorem{theorem}[assertion]{Theorem}

\theoremstyle{definition}

\newtheorem{example}[assertion]{Example}
\newtheorem{examples}[assertion]{Examples}

\theoremstyle{remark}

\newtheorem{notation}[assertion]{Notation}

\newtheorem{remark}[assertion]{Remark}

\usepackage{float}

\usepackage{tikz-cd}

\usepackage[hidelinks]{hyperref}

\newcommand{\N}{\mathbf{N}}
\newcommand{\Z}{\mathbf{Z}}
\newcommand{\Q}{\mathbf{Q}}

\newcommand{\C}{\mathbf{C}}

\newcommand{\B}{\mathbf{B}}
\newcommand{\Qbar}{\bar{\Q}}

\newcommand{\Zp}{\Z_p}
\newcommand{\Qp}{\Q_p}
\newcommand{\Cp}{\C_p}

\newcommand{\Qpbar}{\Qbar_p}

\newcommand{\Oc}{\mathcal{O}}
\newcommand{\mf}{\mathfrak{m}}
\newcommand{\dprime}{{\prime\prime}}
\newcommand{\torsion}{\textup{tor}}
\newcommand{\divisible}{\textup{div}}
\newcommand{\cont}{\textup{cont}}
\newcommand{\mult}{\textup{m}}
\newcommand{\cyc}{\textup{cyc}}
\newcommand{\ab}{\textup{ab}}
\newcommand{\ur}{\textup{ur}}
\newcommand{\Iw}{\textup{Iw}}
\newcommand{\et}{\textup{ét}}
\newcommand{\dR}{\textup{dR}}
\newcommand{\st}{\textup{st}}
\newcommand{\cris}{\textup{cris}}
\newcommand{\e}{\textup{e}}
\newcommand{\BdR}{\B_\dR}
\newcommand{\Bst}{\B_\st}
\newcommand{\Bcris}{\B_\cris}
\newcommand{\Be}{\B_\e}

\newcommand{\Ac}{\mathcal{A}}
\newcommand{\Fc}{\mathcal{F}}
\newcommand{\Lc}{\mathcal{L}}

\newcommand{\Uc}{\mathcal{U}}
\newcommand{\Xc}{\mathcal{X}}

\DeclareMathOperator{\Hom}{Hom}

\DeclareMathOperator{\Ker}{Ker}
\DeclareMathOperator{\Img}{Im}
\DeclareMathOperator{\Coker}{Coker}

\DeclareMathOperator{\HH}{H}

\DeclareMathOperator{\DDb}{\mathbf{D}}
\DeclareMathOperator{\DdR}{\DDb_\dR}
\DeclareMathOperator{\Dst}{\DDb_\st}
\DeclareMathOperator{\Dcris}{\DDb_\cris}
\DeclareMathOperator{\De}{\DDb_\e}
\DeclareMathOperator{\Fil}{Fil}
\DeclareMathOperator{\Gr}{Gr}
\DeclareMathOperator{\Gal}{Gal}

\DeclareMathOperator{\Mod}{Mod}
\DeclareMathOperator{\Ab}{Ab}

\DeclareMathOperator{\Br}{B}
\DeclareMathOperator{\Cr}{C}
\DeclareMathOperator{\Sr}{S}
\DeclareMathOperator{\Zr}{Z}

\DeclareMathOperator{\Tr}{Tr}
\DeclareMathOperator{\height}{ht}
\DeclareMathOperator{\id}{id}
\DeclareMathOperator{\length}{length}

\DeclareMathOperator{\res}{res}

\newcommand{\similarrightarrow}{\xrightarrow{\sim}}

\DeclareFontFamily{U}{wncy}{}
\DeclareFontShape{U}{wncy}{m}{n}{<->wncyr10}{}
\DeclareSymbolFont{mcy}{U}{wncy}{m}{n}
\DeclareMathSymbol{\Sha}{\mathord}{mcy}{"58}

\newcommand{\keywords}[1]{\noindent{\small\textbf{Keywords:}~#1.}\par}
\newcommand{\msc}[2]{\noindent{\small\textbf{Mathematics Subject Classification (2020):}~Primary:~#1. Secondary:~#2.}\par}
\newcommand{\email}[1]{\href{mailto:#1}{\nolinkurl{#1}}}
\newcommand{\contact}[3]{\begin{samepage}\noindent{\small\textsc{#1}\\*~#2\\*\textit{E-mail address:}~\email{#3}\par}\end{samepage}}

\title{\textbf{On a characterisation of perfectoid fields by Iwasawa theory}}
\author{Gautier \textsc{Ponsinet}}
\date{\today}

\AtEndDocument{
	\bigskip
	\contact{G.~\textsc{Ponsinet}}{Universität Duisburg--Essen, Fakultät für Mathematik, Thea-Leymann-Str. 9, 45127 Essen, Germany}{gautier.ponsinet@uni-due.de}
}

\begin{document}
\maketitle
\begin{abstract}
	We prove that the vanishing of the module of universal norms associated with a de Rham Galois representation whose Hodge\--Tate weights are not all non\-/positive characterises the algebraic extensions of the field of \(p\)\=/adic numbers whose completion is a perfectoid field.
	We thereby generalise results by Coates and Greenberg for abelian varieties, and by Bondarko for \(p\)\=/divisible groups.
\end{abstract}
\bigskip
\keywords{Iwasawa theory, perfectoid fields}
\msc{\texttt{11R23}}{\texttt{14G45}}
\tableofcontents

\section{Introduction}
\label{sec:introduction}

Let \(p\) be a prime number.
The purpose of the present article is to characterise the algebraic extensions of the field of \(p\)\=/adic numbers whose completion is a perfectoid field in terms of the Iwasawa theory of de Rham Galois representations.

We fix the following notation for the entirety of the introduction.

Let \(\Qpbar\) be an algebraic closure of the field \(\Qp\) of \(p\)\=/adic numbers.
Let \(K\) be a finite extension of \(\Qp\) contained in \(\Qpbar\).
We denote by \(G_K = \Gal(\Qpbar/K)\) the absolute Galois group of \(K\).
Let \(L\) be an algebraic extension of \(K\) contained in \(\Qpbar\).

Let \(V\) be a \(p\)\=/adic representation of \(G_K\), that is, a finite dimensional \(\Qp\)\=/vector space equipped with a continuous and \(\Qp\)\=/linear action of \(G_K\).
Let \(T\) be a \(\Zp\)\=/lattice in \(V\) stable under the action of \(G_K\), that is, a finitely generated \(\Zp\)\=/submodule of \(V\), spanning \(V\) over \(\Qp\), and stable under the action of \(G_K\).

We adopt the convention that the Hodge\--Tate weight of the \(1\)\=/dimensional \(p\)\=/adic representation \(\Qp(1)\), on which \(G_K\) acts via the cyclotomic character, is~\(1\).

\subsection{Motivation}
\label{subsec:motivation}

Bloch and Kato~\cite{BlochKato1990} have defined via \(p\)\=/adic Hodge theory, and for each finite extension \(K^\prime\) of \(K\), subgroups in Galois cohomology
\[
	\HH^1_e(K^\prime,T) \subseteq \HH^1_f(K^\prime,T) \subseteq \HH^1_g(K^\prime,T) \subseteq \HH^1(K^\prime,T),
\]
which are involved in their conjecture on the special values of \(L\)\=/functions of motives~\cite{Fontaine1992}.

We consider the \(1\)\=/st Iwasawa cohomology group
\[
	\HH^1_\Iw(K,L,T) = \varprojlim \HH^1(K^\prime,T),
\]
and, for each \(\ast \in \{e,f,g\}\), the module of \emph{\(\ast\)\=/universal norms}
\[
	\HH^1_{\Iw,\ast}(K,L,T) = \varprojlim \HH^1_\ast(K^\prime,T),
\]
where \(K^\prime\) runs through the finite extensions of \(K\) contained in \(L\), and the transition morphisms are the corestriction maps.
The groups thus defined
\[
		\HH^1_{\Iw,e}(K,L,T) \subseteq \HH^1_{\Iw,f}(K,L,T) \subseteq \HH^1_{\Iw,g}(K,L,T) \subseteq \HH^1_\Iw(K,L,T)
\]
are involved in Iwasawa theory~\cite{FukayaKato2006}, which motivates their study.

For instance, if \(T=T_p(A)\) is the \(p\)\=/adic Tate module of an abelian variety \(A\) defined over \(K\), then the modules of universal norms \(\HH^1_{\Iw,\ast}(K,L,T_p(A))\) all coincide and are isomorphic via the Kummer map to the module studied in Mazur's foundational article~\cite{Mazur1972}
\[
	\varprojlim A^{(p)}(K^\prime),
\]
where  \(K^\prime\) runs through the finite extensions of \(K\) contained in \(L\), the group \(A^{(p)}(K^\prime)\) is the \(p\)\=/adic completion of the group \(A(K^\prime)\) of \(K^\prime\)\=/rational points of \(A\), and the transition morphisms are the norm maps.

If the completion \(\hat{L}\) of \(L\) for the \(p\)\=/adic valuation topology is a perfectoid field, a class of fields introduced by Scholze~\cite{Scholze2012}, then the modules of universal norms have been computed in several cases (see \S~\ref{subsec:bloch-kato-groups-over-perfectoid-fields}, where results by Coates and Greenberg~\cite{CoatesGreenberg1996}, Perrin-Riou~\cite{Perrin-Riou2000,Perrin-Riou2001}, Berger~\cite{Berger2005}, and the author~\cite{Ponsinet2025,Ponsinet2024} are reviewed).
In particular, these results provide examples of non\-/trivial de Rham representation \(V\) whose Hodge\--Tate weights are not all \(\leq 0\), which implies that, for each finite extension \(K^\prime\) of \(K\), the \(\Zp\)\=/module \(\HH^1_g(K^\prime,T)\) has rank \(> 0\), and such that if \(\hat{L}\) is perfectoid, then the module \(\HH^1_{\Iw,g}(K,L,T)\) is trivial.
The purpose of the present article is to establish the converse statement.

\subsection{Main result}
\label{subsec:result}

The main result of the present article is the following.

\begin{theorem} \label{intro:theorem:main}
	Assume that \(V\) is a non\-/trivial de Rham representation of \(G_K\) whose Hodge\--Tate weights are not all \(\leq 0\).
	If the module \(\HH^1_{\Iw,g}(K,L,T)\) is trivial, then the field \(\hat{L}\) is perfectoid.
\end{theorem}

\begin{remark}
	From Theorem~\ref{intro:theorem:main}, we recover, via a different method, results by Coates and Greenberg~\cite{CoatesGreenberg1996} and by Bondarko~\cite{Bondarko2003} characterising perfectoid fields in terms of the Galois cohomology of abelian varieties and \(p\)\=/divisible groups respectively (see \S~\ref{subsec:p-divisible-groups}).
\end{remark}

\begin{remark}
	As we have already mentioned, there exist non\-/trivial de Rham representations \(V\) whose Hodge\--Tate weights are not all \(\leq 0\), and such that if \(\hat{L}\) is perfectoid, then \(\HH^1_{\Iw,g}(K,L,T)\) is trivial.
	Therefore, for these representations, the following equivalence holds: the field \(\hat{L}\) is perfectoid if and only if the module \(\HH^1_{\Iw,g}(K,L,T)\) is trivial (see \S~\ref{subsec:bloch-kato-groups-over-perfectoid-fields}).
\end{remark}

\begin{remark}
	The hypothesis on the Hodge\--Tate weights of \(V\) in Theorem~\ref{intro:theorem:main} is necessary (see \S~\ref{subsec:hodge-tate-weights}).
\end{remark}

\subsection{Outline of the proof}
\label{subsec:outline}

Our proof of Theorem~\ref{intro:theorem:main} relies on the following.
\begin{itemize}
	\item The definition of the Bloch\--Kato groups involves the \(p\)\=/adic period rings (see \S~\ref{subsec:bloch-kato-groups} and \S~\ref{subsec:universal-norms}).
	\item The Galois cohomology of the \(p\)\=/adic period rings characterises perfectoid fields (see \S~\ref{subsec:galois-cohomology-of-p-adic-period-rings}).
	Recall that using Tate's method~\cite{Tate1967}, Iovita and Zaharescu~\cite{IovitaZaharescu1999:1} have computed the Galois cohomology groups of the ring \(\BdR^+\) of \(p\)\=/adic period introduced by Fontaine~\cite{Fontaine1994:II}, and they have for instance proved
	\[
		\HH^1(L,\BdR^+) =
		\begin{cases}
			0 & \text{if \(\hat{L}\) is perfectoid}, \\
			\text{\(1\)\=/dimensional \(\hat{L}\)\=/vector space} & \text{otherwise}.
		\end{cases}
	\]
	\item Galois cohomology groups may be equipped with a well\-/behaved structure of topological abelian groups induced by the compact\--open topology on continuous cochains, which is closely related to condensed group cohomology (see \S~\ref{subsec:the-compact-open-topology}).
\end{itemize}

The fraction field \(\BdR\) of \(\BdR^+\) contains the period ring \(\Be = \Bcris^{\varphi = 1}\), and, if \(V\) is de Rham, then there exists a commutative diagram whose rows and columns are exact
\begin{equation} \label{eq:intro:diagram}
	\begin{tikzcd}
		& & & 0 \ar{d} \\
		& & & \HH^1(L,\BdR^+ \otimes_{\Qp} V) \ar{d} \\
		0 \ar{r} & \HH^1_g(L,V)/\HH^1_e(L,V) \ar{r} \ar{d} & \HH^1(L,\Be \otimes_{\Qp} V) \ar{r}{f_L} \ar[equal]{d} & \HH^1(L,\BdR \otimes_{\Qp} V) \ar{d} \\
		0 \ar{r} & \HH^1(L,V)/\HH^1_e(L,V) \ar{r} & \HH^1(L,\Be \otimes_{\Qp} V) \ar{r} & \HH^1(L,(\BdR/\BdR^+) \otimes_{\Qp} V) \ar{d} \\
		& & & \phantom{,}0,
	\end{tikzcd}
\end{equation}
where
\[
	\begin{split}
		\HH^1_e(L,V) & = \Ker\left(\HH^1(L,V) \rightarrow \HH^1(L,\Be \otimes_{\Qp} V)\right), \\
		\HH^1_g(L,V) & = \Ker\left(\HH^1(L,V) \rightarrow \HH^1(L,\BdR \otimes_{\Qp} V)\right)
	\end{split}
\]
are respectively the exponential and the geometric Bloch\--Kato groups.

The aformentionned topology defines a structure of \(p\)\=/adic Banach space on \(\HH^1(L,V)\), and we will prove that, if \(V\) is de Rham, then the geometric Bloch\--Kato group \(\HH^1_g(L,V)\) is closed in \(\HH^1(L,V)\).

We will prove Theorem~\ref{intro:theorem:main} by contradiction as follow.
Let \(V^\ast(1) = \Hom(V,\Qp(1))\) be the Tate dual representation of \(V\).
Assume that \(V^\ast(1)\) satisfies the hypotheses of Theorem~\ref{intro:theorem:main} and that \(\hat{L}\) is not perfectoid.
We will then prove the following two contradictory statements.
\begin{itemize}
	\item The extension \(L/K\) is infinite, and, by arguing on the dimensions of Galois cohomology groups along the finite extensions \(K^\prime\) of \(K\) contained in \(L\), it holds
	\[
		\Img(f_L) \cap \HH^1(L,\BdR^+ \otimes_{\Qp}V) \neq 0.
	\]
	\item The exponential Bloch\--Kato group \(\HH^1_e(L,V)\) is dense in \(\HH^1(L,V)\), and thus, since \(\HH^1_g(L,V)\) is closed in \(\HH^1(L,V)\), it holds \(\HH^1_g(L,V) = \HH^1(L,V)\).
	The commutativity of the diagram~\eqref{eq:intro:diagram} then implies that
	\[
		\Img(f_L) \cap \HH^1(L,\BdR^+ \otimes_{\Qp}V) = 0.
	\]
\end{itemize}

\subsection{Organisation of the article}
\label{subsec:organisation}

In section~\ref{sec:continuous-group-cohomology}, we review and we further study the structure of topological abelian groups on continuous group cohomology induced by the compact\-/open topology on continuous cochains.

In section~\ref{sec:p-adic-period-rings-and-perfectoid-fields}, we review the relation between the Galois cohomology of \(p\)\=/adic period rings and perfectoid fields.
Moreover, we study the structure of topological groups of the Galois cohomology groups of \(p\)\=/adic period rings.

In section~\ref{sec:bloch-kato-groups-and-perfectoid-fields}, we review the definition and properties of the Bloch\--Kato groups and we prove Theorem~\ref{intro:theorem:main}.

In section~\ref{sec:examples}, we recover from Theorem~\ref{intro:theorem:main} results by Coates and Greenberg for abelian varieties, and by Bondarko for \(p\)\=/divisible groups.
We also provide examples of representations which imply that the statement Theorem~\ref{intro:theorem:main} is not empty, and that the hypothesis in Theorem~\ref{intro:theorem:main} on the Hodge\--Tate weights is necessary.

\subsection{Acknowledgements}
\label{subsec:acknowledgements}

This work has been supported by:
\begin{itemize}
	\item the Deutsche Forschungsgemeinschaft (DFG, German Research Foundation) through the Collaborative Research Centre TRR 326 \enquote{Geometry and Arithmetic of Uniformized Structures}, project number 444845124, which has allowed me to conduct research at the Universität Heidelberg,
	\item the Institut des Hautes Études Scientifiques,
	\item the Universität Duisburg--Essen.
\end{itemize}
I thank Otmar Venjakob at the Universität Heidelberg and Massimo Bertolini at the Universität Duisburg--Essen.
I thank Laurent Berger, Pierre Colmez, Tongmu He, and Joaquín Rodrigues Jacinto for their comments on a first version of the article.
I also thank Pierre Colmez for providing me with Hyodo's article~\cite{Hyodo1991}.

\subsection{Notation}
\label{subsec:notation}

We adopt the convention that the set of natural numbers \(\N\) contains \(0\).

We fix a prime number \(p\), and an algebraic closure \(\Qpbar\) of the field \(\Qp\) of \(p\)\=/adic numbers.
We denote by \(\Cp\) the completion of \(\Qpbar\) for the \(p\)\=/adic valuation topology.
Every algebraic extension of \(\Qp\) considered is contained in \(\Qpbar\).
We denote by \(\Qp^\ur\) the maximal unramified extension of \(\Qp\), and by \(\sigma \in \Gal(\Qp^\ur/\Qp)\) the arithmetic Frobenius.
If \(L\) is an algebraic extension of \(\Qp\), then we denote by \(G_L = \Gal(\Qpbar/L)\) the absolute Galois group of \(L\), and by \(L_0 = L \cap \Qp^\ur\) the maximal unramified extension of \(\Qp\) contained in \(L\).
We fix a finite extension \(K\) of \(\Qp\).

If \(k\) is a valued field, then we denote by \(\Oc_k\) the associated valuation ring, and by \(\mf_k\) the maximal ideal of \(\Oc_k\), and we denote by \(\hat{k}\) the completion of \(k\).

If \(A\) is an abelian group, then we denote by \(A_\divisible\) the maximal divisible subgroup of \(A\), by \(A_\torsion\) the torsion subgroup of \(A\), by \(A[p^n]\) the \(p^n\)\=/torsion subgroup of \(A\) for each \(n \in \N\), and by \(A[p^\infty] = \bigcup_{n \in \N} A[p^n]\) the \(p\)\=/primary torsion subgroup of \(A\).

\section{Continuous group cohomology and the compact\texorpdfstring{\-/}{-}open topology}
\label{sec:continuous-group-cohomology}

\subsection{Continuous group cohomology}
\label{subsec:continuous-group-cohomology}

We briefly recall the definition of continuous group cohomology~\cite[\S~2]{Tate1976}.

Let \(G\) be a topological group.
A \emph{topological \(G\)\=/module} is a topological abelian group \(M\) equipped with a left action by \(G\)
\begin{equation} \label{eq:group-action}
	G \times M \rightarrow M,
\end{equation}
which is compatible with the group structure of \(M\) and continuous, that is, such that the map~\eqref{eq:group-action} is continuous.

An \emph{exact sequence} of topological abelian groups is a sequence of topological abelian groups whose underlying sequence of abelian groups is exact.
A \emph{strict exact sequence} of topological abelian groups is an exact sequence of topological abelian groups in which all maps are strict.
An \emph{exact sequence} (respectively a \emph{strict exact sequence}) of topological \(G\)\=/modules is a sequence of topological \(G\)\=/modules whose underlying sequence of topological abelian groups is exact (respectively strict exact).

If \(M\) is a topological \(G\)\=/module, then, for each \(n \in \N\), we denote by \(\Cr^n(G,M)\) the group of continuous \(n\)\=/cochains, that is, the group of continuous functions from \(G^n\) to \(M\).
For each \(n \in \N\), there exits a coboundary group homomorphism
\[
	d_n : \Cr^n(G,M) \rightarrow \Cr^{n+1}(G,M),
\]
the group of \(n\)\=/cocycles is defined by \(\Zr^n(G,M) = \Ker d_n\), and the group of \(n\)\=/coboundaries by
\[
	\Br^n(G,M) =
	\begin{cases}
		\Img(d_{n-1}) & \text{if \(n>0\),} \\
		0 & \text{if \(n=0\)}.
	\end{cases}
\]
The coboundary maps satisfy the equation \(d_{n+1} \circ d_n = 0\), and thus, there are inclusions
\[
	\Br^n(G,M) \subseteq \Zr^n(G,M) \subseteq \Cr^n(G,M).
\]
The \(n\)\=/th group of \emph{continuous group cohomology} \(\HH^n(G,M)\) is defined by
\[
	\HH^n(G,M) = \Zr^n(G,M)/\Br^n(G,M).
\]
It is convenient to also set \(\HH^{-1}(G,M) = 0\).

\begin{proposition} \label{proposition:lim-cohomology}
	Let \((M_i,f_i : M_{i+1} \rightarrow M_i)_{i \in \N}\) be a projective system of topological \(G\)\=/modules such that each map \(f_i\) admits a continuous section.
	Let \(\varprojlim_{f_i} M_i\) be the topological \(G\)\=/module equipped with the inverse limit topology.
	Then, for each \(n \in \N\), there exists an exact sequence of groups
	\[
		0 \rightarrow {\varprojlim_{f_i}}^1 \HH^{n-1}(G,M_i) \rightarrow \HH^n(G,\varprojlim_{f_i} M_i) \rightarrow \varprojlim_{f_i} \HH^n(G,M_i) \rightarrow 0.
	\]
\end{proposition}
\begin{proof}
	Since the map \(f_i\) admits a continuous section \(s_i : M_i \rightarrow M_{i+1}\), the map \(h_i : \Cr^n(G,M_{i+1}) \rightarrow \Cr^n(G,M_i)\) induced by \(f_i\) is surjective for each integers \(i\) and \(n\).
	Indeed, if \(c_i\) is an element of \(\Cr^n(G,M_i)\), then \(s_i \circ c_i \in \Cr^n(G,M_{i+1})\) and \(h_i (s_i \circ c_i) = f_i \circ s_i \circ c_i = c_i\).
	It follows that, for each \(n \in \N\), there is an exact sequence of groups
	\begin{equation} \label{eq:complex-limit}
		0 \rightarrow \Cr^n(G,\varprojlim_{f_i} M_i) \rightarrow \prod_{i \in \N} \Cr^n(G,M_i) \xrightarrow{\id - h} \prod_{i \in \N} \Cr^n(G,M_i) \rightarrow 0,
	\end{equation}
	where \((\id - h)((c_i)_i) = (c_i - h_{i+1}(c_{i+1}))_i\) for each \((c_i)_i \in \prod_{i \in \N} \Cr^n(G,M_i)\).
	The exactness in the middle of the sequence~\eqref{eq:complex-limit} follows from the definition of \(\varprojlim_{f_i} M_i\), and the surjectivity of \((\id - h)\) can be proved by constructing an element in the preimage of an element \((c_i)_i\) by induction on \(i\).
	The cohomology of the short exact sequence of complexes~\eqref{eq:complex-limit} yields the statement.
\end{proof}

\subsection{The compact\texorpdfstring{\-/}{-}open topology}
\label{subsec:the-compact-open-topology}

Let \(G\) be a topological group.
If \(M\) is a topological \(G\)\=/module, then, for each \(n \in \N\), the \emph{compact\-/open topology} on the set \(\Cr^n(G,M)\) is the topology generated by the subsets
\[
	\Sr(J,U) = \{f \in \Cr^n(G,M), f(J) \subseteq U\},
\]
where \(J\) runs through the compact subsets of \(G^n\), and \(U\) runs through the open subsets of \(M\).

Recall~\cite[Appendix~A]{Ponsinet2025} that if \(G\) is locally compact and separated, then, for each \(n \in \N\), the compact\-/open topology defines a structure of topological abelian group on \(\Cr^n(G,M)\), and the coboundaries maps \(d_n\) are morphisms of topological abelian groups.
The subspace topology from \(\Cr^n(G,M)\) makes \(\Zr^n(G,M)\) and \(\Br^n(G,M)\) into topological abelian groups, and the quotient topology from \(\Zr^n(G,M)\) makes \(\HH^n(G,M)\) into a topological abelian group.
Moreover, the structure of topological abelian group on \(\HH^n(G,M)\) satisfies the following properties.
\begin{itemize}
	\item Let \(M^G\) be the submodule of \(G\)\=/invariant elements of \(M\) equipped with the subspace topology from \(M\).
	The canonical isomorphism
	\[
		\HH^0(G,M) \similarrightarrow M^G
	\]
	is an isomorphism of topological abelian groups.
	\item If \(M \rightarrow N\) is a morphism of topological \(G\)\=/modules, then the induced map \(\HH^n(G,M) \rightarrow \HH^n(G,N)\) is a morphism of topological abelian groups.
	\item If
	\begin{equation} \label{eq:ses}
		0 \rightarrow M^\prime \rightarrow M \rightarrow M^\dprime \rightarrow 0
	\end{equation}
	is a strict exact sequence of topological \(G\)\=/modules, then the sequence~\eqref{eq:ses} induces an exact sequence of topological abelian groups
	\begin{equation} \label{eq:esh0h1}
		\begin{tikzcd}
			0 \ar{r} & \HH^0(G,M^\prime) \ar{r} & \HH^0(G,M) \ar{r} \ar[phantom, ""{coordinate, name=Z}]{d} & \HH^0(G,M^\dprime) \ar[rounded corners, to path={ -- ([xshift=2em]\tikztostart.east) |- (Z) [near end]\tikztonodes -| ([xshift=-2em]\tikztotarget.west) -- (\tikztotarget)}]{dll} \\
			& \HH^1(G,M^\prime) \ar{r} & \HH^1(G,M) \ar{r} & \HH^1(G,M^\dprime).
		\end{tikzcd}
	\end{equation}
	Moreover, if there exists a continuous section of the projection of \(M\) on \(M^\dprime\) as topological spaces, then the sequence~\eqref{eq:esh0h1} extends into a long exact sequence of topological abelian groups
	\[
		\cdots \rightarrow \HH^n(G,M^\prime) \rightarrow \HH^n(G,M) \rightarrow  \HH^n(G,M^\dprime) \rightarrow \HH^{n+1}(G,M^\prime) \rightarrow \cdots.
	\]
\end{itemize}

\begin{remark}
	The continuous group cohomology groups, thus equipped with a structure of topological abelian groups induced by the compact\-/open topology, are closely related to condensed group cohomology developped by Clausen and Scholze (see for instance~\cite[\S~2.2.1]{Artusa2024} and~\cite[Lemma~3.3.2 2)]{BarthelSchlankStapletonWeinstein2024}).
\end{remark}

\begin{notation}
	We recall the standard notation for \emph{Galois cohomology}.
	\begin{itemize}
		\item If \(G=\Gal(k^\prime/k)\) is the Galois group of a Galois field extension \(k^\prime/k\), then we write \(\HH^n(k^\prime/k,M)\) instead of \(\HH^n(\Gal(k^\prime/k),M)\).
		\item If \(G=G_k\) is the absolute Galois group of a field \(k\), then we write \(\HH^n(k,M)\) instead of \(\HH^n(G_k,M)\).
	\end{itemize}
\end{notation}

\begin{remark}
	Galois groups are profinite, and thus, compact and separated.
\end{remark}

\subsection{Linear topology}
\label{subsec:linear-topology}

Let \(G\) be a compact and separated topological group.
Let \(M\) be a topological \(G\)\=/module whose topology is linear, generated by a set of \(G\)\=/submodules \(\{M_i\}_{i \in \N}\) which forms a fundamental system of open neighbourhoods of the identity element of \(M\).

\begin{proposition} \label{proposition:linear}
	For each \(n \in \N\), the topology of \(\HH^n(G,M)\) is linear, generated by the set of subgroups
	\[
		\left\{\Img\left(\HH^n(G,M_i) \rightarrow \HH^n(G,M)\right)\right\}_{i \in \N},
	\]
	which forms a fundamental system of open neighbourhoods of the identity element of \(\HH^n(G,M)\).
\end{proposition}
\begin{proof}
	We first note that the topology of the group \(\Cr^n(G,M)\) is linear, generated by the subgroups \(\Sr(J,M_i)\), where \(J\) runs through the compact subsets of \(G^n\) and \(i \in \N\).
	Indeed, the set \(\{m + M_i\}_{m \in M,i \in \N}\) is a base of the topology of the space \(M\), which implies that the set
	\[
		\{\Sr(J,m+M_i); J \text{ compact}, m \in M, i \in \N\}
	\]
	is a base of the compact-open topology on \(\Cr^n(G,M)\) (see~\cite[X \S~3 \textnumero~4 Remarques~2)]{Bourbaki1974:TG:5-10}).
	Moreover, if the identity element  of \(\Cr^n(G,M)\) belongs to \(\Sr(J,m+M_i)\), then \(\Sr(J,m+M_i) = \Sr(J,M_i)\).
	Therefore, the set of subgroups \(\{\Sr(J,M_i); J \text{ compact}, i \in \N\}\) forms a fundamental system of open neighbourhoods of the identity element of \(\Cr^n(G,M)\).

	Since \(G\) is compact, the space \(G^n\) is compact, and thus, the subgroup \(\Sr(G^n,M_i)\) of \(\Cr^n(G,M)\) is open.
	Moreover, it holds
	\[
		\Sr(G^n,M_i) \subseteq \Sr(J,M_i),
	\]
	for each compact subsets \(J\) of \(G^n\).
	Therefore, in the linear topology on \(\Cr^n(G,M)\) generated by the subgroups \(\Sr(G^n,M_i)\), the subgroup \(\Sr(J,M_i)\) is open~\cite[III \S~2 \textnumero~1 Corollaire]{Bourbaki1971:TG:1-4}, and thus, the compact\-/open topology on the group \(\Cr^n(G,M)\) coincides with the linear topology generated by the subgroups \(\{\Sr(G^n,M_i)\}_{i \in \N}\).

	Note that the subgroup \(\Sr(G^n,M_i)\) is the image of the group \(\Cr^n(G,M_i)\) in \(\Cr^n(G,M)\).
	Since the morphism of groups \(\Cr^j(G,M_i) \rightarrow \Cr^j(G,M)\) is injective for each \(i,j \in \N\), the induced topology on \(\Zr^n(G,M)\) is then generated by the subgroups
	\[
		\begin{split}
			& \Zr^n(G,M) \cap \Sr(G^n,M_i) \\
			= & \Zr^n(G,M) \cap \Img\left(\Cr^n(G,M_i) \rightarrow \Cr^n(G,M)\right) \\
			= & \Img\left(\Zr^n(G,M_i) \rightarrow \Zr^n(G,M)\right).
		\end{split}
	\]
	Finally, the quotient topology on \(\HH^n(G,M)\) is generated by the images of the groups \(\Zr^n(G,M_i)\) in \(\HH^n(G,M)\), and these images coincide with the images of the groups \(\HH^n(G,M_i)\) in \(\HH^n(G,M)\).
\end{proof}

\begin{corollary} \label{corollary:discrete}
	If \(M\) is discrete, then, for each \(n \in \N\), the topological group \(\HH^n(G,M)\) is discrete.
\end{corollary}
\begin{proof}
	If \(M\) is discrete, then the topology of \(M\) is linear with one of the subgroups \(M_i\) being the trivial subgroup of \(M\), and the statement follows from Proposition~\ref{proposition:linear}.
\end{proof}

\begin{lemma} \label{lemma:lim1-top}
	Let \((A_i,f_i)_{i \in \N}\) be a projective system of topological abelian groups.
	The quotient topology on \(\varprojlim^1_{f_i} A_i\), from \(\prod_{i \in \N} A_i\) equipped with the product topology, is the trivial topology.
\end{lemma}
\begin{proof}
	On the one hand, the group \(\varprojlim^1_{f_i} A_i\) is the cokernel of the map
	\[
		\prod_{i \in \N} A_i \xrightarrow{\id - f} \prod_{i \in \N} A_i,
	\]
	where \((\id - f)((a_i)_i) = (a_i - f_{i+1}(a_{i+1}))_i\) for each \((a_i)_i \in \prod_{i \in \N} A_i\).
	On the other hand, a basis for the product topology on \(\prod_{i \in \N} A_i\) is formed by subsets \(\prod_i U_i\), where \(U_i\) is an open subset in \(A_i\) and \(U_i = A_i\) for all but finitely many \(i \in \N\).
	We easily verify that for any such subset \(\prod_i U_i\), it holds \((\prod_i U_i) + \Img(\id-f) = \prod_i A_i\).
\end{proof}

If \(N\) is an abelian group equipped with a linear topology generated by a set of subgroups \(\{N_i\}_{i \in \N}\) which forms a fundamental system of open neighbourhoods of the identity element of \(N\), then the topological group \(N\) is separated (respectively complete) if the natural map
\[
	N \rightarrow \varprojlim_{i \in \N} N/N_i
\]
is injective (respectively surjective).

\begin{corollary} \label{corollary:lim-cohomology-linear}
	If \(M\) is separated and complete, then, for each \(n \in \N\), the exact sequence of groups from Proposition~\ref{proposition:lim-cohomology}
	\[
		0 \rightarrow {\varprojlim}^1 \HH^{n-1}(G,M/M_i) \rightarrow \HH^n(G,M) \rightarrow \varprojlim \HH^n(G,M/M_i) \rightarrow 0
	\]
	defines a strict exact sequence of topological groups, where
	\begin{itemize}
		\item the group \(\varprojlim \HH^n(G,M/M_i)\) is equipped with the inverse limit topology, and
		\item the group \(\varprojlim^1 \HH^{n-1}(G,M/M_i)\) is equipped with the quotient topology from \(\prod_{i \in \N} \HH^{n-1}(G,M/M_i)\), which is the trivial topology by Lemma~\ref{lemma:lim1-top}.
	\end{itemize}
	Moreover, the topological group \(\HH^n(G,M)\) is complete.
\end{corollary}
\begin{proof}
	For each \(n \in \N\) and \(i \in \N\), the topological group \(\HH^n(G,M/M_i)\) is discrete by Corollary~\ref{corollary:discrete}.
	Moreover, for each \(n \in \N\) and \(i \in \N\), the strict exact sequence of topological groups
	\[
		0 \rightarrow M_i \rightarrow M \rightarrow M/M_i \rightarrow 0
	\]
	induces an exact sequence of topological groups
	\[
		\HH^n(G,M_i) \xrightarrow{\iota} \HH^n(G,M) \rightarrow \HH^n(G,M/M_i).
	\]
	and thus, the continuous morphism \(\HH^n(G,M) \rightarrow \HH^n(G,M/M_i)\) factorises through the discrete group \(\HH^n(G,M)/\iota(\HH^n(G,M_i))\).

	The exact sequence of groups from Proposition~\ref{proposition:lim-cohomology}
	\[
		0 \rightarrow {\varprojlim}^1 \HH^{n-1}(G,M/M_i) \rightarrow \HH^n(G,M) \rightarrow \varprojlim \HH^n(G,M/M_i) \rightarrow 0
	\]
	combined with the factorisation maps
	\begin{equation} \label{eq:factorisations}
		0 \rightarrow \HH^n(G,M)/\iota(\HH^n(G,M_i)) \rightarrow \HH^n(G,M/M_i)
	\end{equation}
	yields a short exact sequence of groups
	\begin{equation} \label{eq:lim-cohomology-factorisations}
		0 \rightarrow {\varprojlim}^1 \HH^{n-1}(G,M/M_i) \rightarrow \HH^n(G,M) \rightarrow \varprojlim \HH^n(G,M)/\iota(\HH^n(G,M_i)) \rightarrow 0,
	\end{equation}
	and an isomorphism of groups
	\begin{equation} \label{eq:lim-factorisations}
		\varprojlim \HH^n(G,M)/\iota(\HH^n(G,M_i)) \similarrightarrow \varprojlim \HH^n(G,M/M_i).
	\end{equation}
	By Proposition~\ref{proposition:linear} the topology of \(\HH^n(G,M)\) is linear generated by \(\iota(\HH^n(G,M_i))\), by Lemma~\ref{lemma:lim1-top} the topology of \(\varprojlim^1 \HH^{n-1}(G,M/M_i)\) is trivial, and we have already established that the factorisation maps~\eqref{eq:factorisations} are continuous morphisms of discrete groups, and hence, we conclude that the exact sequence~\eqref{eq:lim-cohomology-factorisations} defines a strict exact sequence of topological groups and the isomorphism~\eqref{eq:lim-factorisations} defines an isomorphism of topological groups.
\end{proof}

\begin{proposition} \label{proposition:spectral-sequence}
	Assume that \(G\) is profinite.
	Let \(H\) be a closed normal subgroup of \(G\).
	If \(M\) is separated and complete, and such that, for each \(n \in \N\), the projective system \((\HH^n(H,M/M_i))_{i \in \N}\) of discrete \(G/H\)\=/modules satisfies the Mittag-Leffler condition.
	Then there exists a convergent spectral sequence of groups
	\[
		E_2^{n,m} = \HH^n(G/H,\HH^m(H,M)) \Rightarrow \HH^{n+m}(G,M).
	\]
\end{proposition}
\begin{proof}
	Recall~\cite[Lemma~2.6.5]{NeukirchSchmidtWingberg2008} that if \(\Pi\) is a profinite group, then the category \(\Mod_\delta(\Pi)\) of discrete \(\Pi\)\=/modules is abelian with enough injectives, which implies that the category \(\Mod_\delta(\Pi)^\N\) of projective systems indexed by \(\N\) is abelian with enough injectives~\cite[Proposition~1.1]{Jannsen1988}.
	Let \(\Ab\) be the category of abelian groups.
	Following Jannsen~\cite[\S~2]{Jannsen1988}, for each \(n \in \N\), the \(n\)\=/th derived functor associated with the left exact functor
	\[
		\begin{split}
		\Mod_\delta(\Pi)^{\N} & \rightarrow \Ab \\
		(A_i) \mapsto & \varprojlim \HH^0(\Pi,A_i)
		\end{split}
	\]
	is denote by \(\HH^n(\Pi,(A_i))\).
	Recall~\cite[Theorem~2.2]{Jannsen1988} that if \((A_i)\) satisfies the Mittag-Leffler condition and \(A = \varprojlim A_i\) denotes the topological \(\Pi\)\=/module equipped with the inverse limit topology, then there exists a canonical isomorpshism of groups
	\begin{equation} \label{eq:jannsen-continuous}
		\HH^n(\Pi,A) \simeq \HH^n(\Pi,(A_i)).
	\end{equation}

	We consider the composition of functors
	\begin{equation} \label{eq:composition-functors}
		\begin{tikzcd}[row sep=0]
			\Mod_\delta(G)^{\N} \ar{r} &  \Mod_\delta(G/H)^{\N} \ar{r} & \Ab \\
			(A_i) \ar{r} & (\HH^0(H,A_i)) & \\
			& (B_i) \ar{r} & \varprojlim \HH^0(G/H,B_i).
		\end{tikzcd}
	\end{equation}
	Since the first functor in~\eqref{eq:composition-functors} is right adjoint to the exact forgetful functor from \(\Mod_\delta(G/H)^{\N}\) to \(\Mod_\delta(G)^{\N}\), the first functor in~\eqref{eq:composition-functors} maps injective objects to injective objects.
	Thus, we may consider the Grothendieck spectral sequence associated with the composition of functors~\eqref{eq:composition-functors}, which, applied to the system \((M/M_i)\), reads
	\begin{equation} \label{eq:spectral-sequence-jannsen}
		\HH^n(G/H,(\HH^m(H,M/M_i))) \Rightarrow \HH^{n+m}(G,(M/M_i)).
	\end{equation}
	Since the system \((M/M_i)\) and the systems \((\HH^m(H,M/M_i))\) for all \(m \in \N\) satisfy the Mittag-Leffler condition, the combination of the canonical isomorphisms~\eqref{eq:jannsen-continuous} and the spectral sequence~\eqref{eq:spectral-sequence-jannsen} yields the spectral sequence
	\begin{equation} \label{eq:spectral-lim}
		\HH^n(G/H,\varprojlim \HH^m(H,M/M_i)) \Rightarrow \HH^{n+m}(G,M),
	\end{equation}
	and, by Corollary~\ref{corollary:lim-cohomology-linear}, there exists, for each \(m \in \N\), an isomorphism of topological groups
	\begin{equation} \label{eq:iso-lim}
		\HH^m(H,M) \similarrightarrow \varprojlim \HH^m(H,M/M_i).
	\end{equation}
	We conclude by combining the spectral sequence~\eqref{eq:spectral-lim} with the isomorphisms~\eqref{eq:iso-lim}.
\end{proof}

\begin{remark}
	If \(M\) is discrete, then Proposition~\ref{proposition:spectral-sequence} is the Hochschild–Serre spectral sequence~\cite[I \S~2.6 b)]{Serre1994}.
\end{remark}

\subsection{\texorpdfstring{\(p\)\=/}{p-}adic topology}
\label{sec:p-adic-topology}

A \(\Zp\)\=/module \(M\) is \emph{\(p\)\=/adically separated} (respectively \emph{\(p\)\=/adically complete}) if it is separated (respectively complete) for the \(p\)\=/adic topology, that is, if the canonical map
\begin{equation} \label{eq:completion}
	M \rightarrow \varprojlim_{i \in \N} M/p^i M
\end{equation}
is injective (respectively surjective).

\begin{remark} \label{remark:completion}
	The kernel and cokernel of the map~\eqref{eq:completion} are respectively \(\varprojlim_{i} p^i M = \bigcap_{i} p^i M\) and \({\varprojlim}^1_{i} p^i M\), where the transition maps are the inclusion maps \(p^{i+1} M \subseteq p^i M\).
	Therefore, the module \(M\) is \(p\)\=/adically separated (respectively complete) if and only if the module \(\varprojlim_{i} p^i M\) (respectively \({\varprojlim}^1_{i} p^i M\)) is trivial.
\end{remark}

We denote by \((M,p)\) (respectively by \((M[p^i])\)) the projective system associated with a \(\Zp\)\=/module \(M\), indexed by \(i \in \N\), where the module at \(i\) is \(M\) (respectively \(M[p^i]\)), and whose transition maps are multiplication by \(p\).
Recall~\cite[\S~4]{Jannsen1988} that there exists a natural long exact sequence
\begin{equation} \label{eq:various-projective-system-sequence}
	\begin{tikzcd}
		0 \ar{r} & \varprojlim M[p^i] \ar{r} & \varprojlim (M,p) \ar{r} \ar[phantom, ""{coordinate, name=Z}]{d} & \varprojlim p^i M \ar[rounded corners, to path={ -- ([xshift=2em]\tikztostart.east) |- (Z) [near end]\tikztonodes -| ([xshift=-2em]\tikztotarget.west) -- (\tikztotarget)}]{dll} & \\
		& {\varprojlim}^1 M[p^i] \ar{r} & {\varprojlim}^1 (M,p) \ar{r} & {\varprojlim}^1 p^i M \ar{r} & 0,
	\end{tikzcd}
\end{equation}
and
\begin{equation} \label{eq:divisible}
	\Img\left(\varprojlim (M,p) \rightarrow \varprojlim p^i M\right) = M_\divisible.
\end{equation}
Note that the system \((M[p^i])\) depends only on \(M_\torsion\), and the module \(\varprojlim M[p^i]\) is the \emph{\(p\)\=/adic Tate module} \(T_p(M)\) of \(M\).

Recall the following~\cite[Proposition~4.4]{Jannsen1988}.

\begin{proposition} \label{proposition:completness-projective-system}
	Let \(M\) be a \(\Zp\)\=/module.
	The following are equivalent.
	\begin{enumerate}
		\item The module \(M\) is \(p\)\=/adically complete and separated.
		\item The groups in the long exact sequence~\eqref{eq:various-projective-system-sequence} associated with \(M\) are all trivial.
		\item There exists a projective system \((M_i,f_i)_{i \in \N}\) of \(\Zp\)\=/modules in which each \(M_i\) is torsion with finite exponent, and an isomorphism \(M \simeq \varprojlim_{f_i} M_i\).
	\end{enumerate}
\end{proposition}

\begin{remark} \label{remark:p-adic-finite-exponent}
	If \(M\) is a torsion \(\Zp\)\=/module with finite exponent, then \(M\) is \(p\)\=/adically complete and separated.
	Indeed, for \(i\) sufficiently large, the multiplication by \(p^i\) on \(M\) is the trivial map, and thus, the canonical map~\eqref{eq:completion} is an isomorphism, and the groups in the long exact sequence~\eqref{eq:various-projective-system-sequence} associated with \(M\) are obviously all trivial.
\end{remark}

\begin{remark} \label{remark:projective-system-sequence}
	A short exact sequence of \(\Zp\)\=/modules
	\[
		0 \rightarrow M^\prime \rightarrow M \rightarrow M^\dprime \rightarrow 0
	\]
	induces a short exact sequence of projective systems
	\[
		0 \rightarrow (M^\prime,p) \rightarrow (M,p) \rightarrow (M^\dprime,p) \rightarrow 0,
	\]
	and thus, an exact sequence of \(\Zp\)\=/modules
	\[
		\begin{tikzcd}
			0 \ar{r} & \varprojlim (M^\prime,p) \ar{r} & \varprojlim (M,p) \ar{r} \ar[phantom, ""{coordinate, name=Z}]{d} & \varprojlim (M^\dprime,p) \ar[rounded corners, to path={ -- ([xshift=2em]\tikztostart.east) |- (Z) [near end]\tikztonodes -| ([xshift=-2em]\tikztotarget.west) -- (\tikztotarget)}]{dll} & \\
			& {\varprojlim}^1 (M^\prime,p) \ar{r} & {\varprojlim}^1 (M,p) \ar{r} & {\varprojlim}^1 (M^\dprime,p) \ar{r} & 0.
		\end{tikzcd}
	\]
\end{remark}

Recall~\cite[Definition~4.6]{Jannsen1988} that a \(\Zp\)\=/module \(M\) is \emph{weakly \(p\)\=/complete} (or \emph{derived \(p\)\=/complete}) if
\[
	\varprojlim (M,p) = {\varprojlim}^1 (M,p) = 0.
\]

\begin{remark} \label{remark:complete-separated-weakly-complete}
	By Proposition~\ref{proposition:completness-projective-system}, if \(M\) is a \(p\)\=/adically complete and separated \(\Zp\)\=/module, then \(M\) is weakly \(p\)\=/complete.
\end{remark}

\begin{lemma} \label{lemma:weakly-p-complete}
	Let \(M\) be a weakly \(p\)\=/complete \(\Zp\)\=/module.
	\begin{enumerate}
		\item The module \(M\) is \(p\)\=/adically complete.
		\item The module \(M\) is \(p\)\=/adically separated if and only if \({\varprojlim}^1 M[p^i]\) is trivial.
		In particular, if \(M_\torsion\) has finite exponent, then \(M\) is \(p\)\=/adically separated.
		\item The subgroup \(M_\divisible\) is trivial.
	\end{enumerate}
\end{lemma}
\begin{proof}
	The first two statements follow from the exact sequence~\eqref{eq:various-projective-system-sequence}, Remark~\ref{remark:p-adic-finite-exponent} and Remark~\ref{remark:completion}.
	The last statement follows from the identity~\eqref{eq:divisible}.
\end{proof}

Recall the following~\cite[Proposition~4.8]{Jannsen1988}.

\begin{proposition} \label{proposition:weakly-p-complete}
	The following holds.
	\begin{enumerate}
		\item If \(0 \rightarrow M^\prime \rightarrow M \rightarrow M^\dprime \rightarrow 0\) is an exact sequence of \(\Zp\)\=/modules in which two of the modules are weakly \(p\)\=/complete, then the third module is weakly \(p\)\=/complete.
		\item If \((M_i,f_i)_{i \in \N}\) is a projective system of \(\Zp\)\=/modules such that each \(M_i\) is torsion with finite exponent, then \(\varprojlim_{f_i}^1 M_i\) is weakly \(p\)\=/complete.
	\end{enumerate}
\end{proposition}
\begin{proof}
	The first statement follows from Remark~\ref{remark:projective-system-sequence}.
	The second statement follows from the first one applied to the exact sequence
	\[
		0 \rightarrow \varprojlim_{f_i} M_i \rightarrow \prod_{i \in \N} M_i \rightarrow \prod_{i \in \N} M_i \rightarrow {\varprojlim_{f_i}}^1 M_i \rightarrow 0
	\]
	since \(\varprojlim_{f_i} M_i\) and \(\prod_{i \in \N} M_i\) are \(p\)\=/adically separated and complete by Proposition~\ref{proposition:completness-projective-system}.
\end{proof}

\begin{proposition} \label{proposition:p-adic-topology-torsion}
	Let \(f: M \rightarrow N\) be a morphism of \(\Zp\)\=/modules such that \(\Ker(f)\) and \(\Coker(f)\) are torsion with finite exponent.
	Then, the following holds.
	\begin{enumerate}
		\item \label{proposition:p-adic-topology-torsion:1} The torsion subgroup \(M_\torsion\) has finite exponent if and only if the torsion subgroup \(N_\torsion\) has finite exponent.
		\item The map \(f\) induces isomorphisms
		\[
			\begin{split}
				\varprojlim (M,p) & \similarrightarrow  \varprojlim (N,p), \\
				{\varprojlim}^1 (M,p) & \similarrightarrow  {\varprojlim}^1 (N,p).
			\end{split}
		\]
		\item \label{proposition:p-adic-topology-torsion:3} The module \(M\) is weakly \(p\)\=/complete if and only if the module \(N\) is weakly \(p\)\=/complete.
		\item If either \(M\) or \(N\) is weakly \(p\)\=/complete, and either \(M_\torsion\) or \(N_\torsion\) has finite exponent, then \(M\) and \(N\) are both \(p\)\=/adically complete and separated.
		\item If \(\Ker(f)\) is trivial, then the map \(f\) induces an isomorphism of the long exact sequences~\eqref{eq:various-projective-system-sequence} associated with \(M\) and \(N\) respectively.
		\item If \(\Ker(f)\) is trivial, then the module \(M\) is \(p\)\=/adically complete and separated if and only if the module \(N\) is \(p\)\=/adically complete and separated.
		\item If \(\Ker(f)\) is trivial, then the map \(f\) induces an isomorphism \(M_\divisible \similarrightarrow N_\divisible\).
	\end{enumerate}
\end{proposition}
\begin{proof}
	Since \(\Ker(f)\) is torsion, we deduce from the snake lemma applied to the commutative diagram
	\[
		\begin{tikzcd}
			0 \ar{r} & M_\torsion \ar{r} \ar{d} & M \ar{r} \ar{d}{f} & M/M_\torsion \ar{r} \ar{d} & 0 \\
			0 \ar{r} & N_\torsion \ar{r} & N \ar{r} & N/N_\torsion \ar{r} & 0
		\end{tikzcd}
	\]
	that the map \(f\) induces an exact sequence
	\begin{equation} \label{eq:es-torsion}
		0 \rightarrow \Ker(f) \rightarrow M_\torsion \rightarrow N_\torsion \rightarrow C \rightarrow 0,
	\end{equation}
	where \(C\) is a subgroup of \(\Coker(f)\), and in particular, the module \(C\) has finite exponent.
	Therefore, the exact sequence~\eqref{eq:es-torsion} implies that \(M_\torsion\) has finite exponent if and only if \(N_\torsion\) has finite exponent.

	By Proposition~\ref{proposition:completness-projective-system} and Remark~\ref{remark:p-adic-finite-exponent} applied to \(\Ker(f)\) and \(\Coker(f)\), and by Remark~\ref{remark:projective-system-sequence}, it holds
	\[
		\varprojlim (\Ker(f),p) = {\varprojlim}^1 (\Ker(f),p) = \varprojlim (\Coker(f),p) = {\varprojlim}^1 (\Coker(f),p) = 0,
	\]
	and the map \(f\) induces isomorphisms
	\begin{equation} \label{eq:iso-lim-proj}
		\begin{split}
			\varprojlim (M,p) & \similarrightarrow  \varprojlim (N,p), \\
			{\varprojlim}^1 (M,p) & \similarrightarrow  {\varprojlim}^1 (N,p).
		\end{split}
	\end{equation}
	The isomorphisms~\eqref{eq:iso-lim-proj} implies that the module \(M\) is weakly \(p\)\=/complete if and only if the module \(N\) is weakly \(p\)\=/complete.
	Moreover, if either \(M\) or \(N\) is weakly \(p\)\=/complete, and either \(M_\torsion\) or \(N_\torsion\) has finite exponent, then both \(M\) and \(N\) are weakly \(p\)\=/complete and both \(M_\torsion\) or \(N_\torsion\) have finite exponent by the statements~\ref{proposition:p-adic-topology-torsion:1} and~\ref{proposition:p-adic-topology-torsion:3} we have just proved, and by Lemma~\ref{lemma:weakly-p-complete}, both \(M\) and \(N\) are \(p\)\=/adically complete and separated.

	We now assume that \(\Ker(f)\) is trivial.
	For each \(i \in \N\), the snake lemma applied to the commutative diagram
	\[
		\begin{tikzcd}
			0 \ar{r} & M \ar{r} \ar{d}{p^i} & N \ar{r} \ar{d}{p^i} & \Coker(f) \ar{d}{p^i} \ar{r} & 0 \\
			0 \ar{r} & M \ar{r} & N \ar{r} & \Coker(f) \ar{r} & 0
		\end{tikzcd}
	\]
	yields an exact sequence
	\[
		0 \rightarrow M[p^i] \rightarrow N[p^i] \rightarrow \Coker(f)[p^i].
	\]
	In particular, there is a short exact sequence of projective systems
	\begin{equation} \label{eq:ses-projective-systems-torsion}
		0 \rightarrow (M[p^i],p) \rightarrow (N[p^i],p) \rightarrow (N[p^i]/M[p^i],p) \rightarrow 0,
	\end{equation}
	where \(N[p^i]/M[p^i] \subseteq \Coker(f)\) is of finite exponent bounded independently of~\(i\).
	The bound on the finite exponents of \((N[p^i]/M[p^i])_{i \in \N}\) implies that
	\[
		\varprojlim N[p^i]/M[p^i] = {\varprojlim}^1 N[p^i]/M[p^i] =0,
	\]
	and the short exact sequence~\eqref{eq:ses-projective-systems-torsion} induces isomorphisms
	\begin{equation} \label{eq:iso-lim-proj-torsion}
		\begin{split}
			\varprojlim M[p^i] & \similarrightarrow  \varprojlim N[p^i], \\
			{\varprojlim}^1 M[p^i] & \similarrightarrow  {\varprojlim}^1 N[p^i].
		\end{split}
	\end{equation}

	The isomorphisms~\eqref{eq:iso-lim-proj} and~\eqref{eq:iso-lim-proj-torsion} and the five lemma imply that \(f\) induces an isomorphism of the long exact sequences~\eqref{eq:various-projective-system-sequence} associated with \(M\) and \(N\) respectively.
	In particular, the module \(M\) is \(p\)\=/adically complete and separated if and only if the module \(N\) is \(p\)\=/adically complete and separated by Proposition~\ref{proposition:completness-projective-system}, and the identity~\eqref{eq:divisible} implies that \(f\) induces an isomorphism between \(M_\divisible\) and \(N_\divisible\).
\end{proof}

Let \(G\) be a compact and separated topological group.
Let \(M\) be a \(\Zp\)\=/module endowed with the \(p\)\=/adic topology, and equipped with a continuous and \(\Zp\)\=/linear action of \(G\).

\begin{proposition} \label{proposition:cohomology-zp-modules}
	Let \(n \in \N\).
	\begin{enumerate}
		\item If \(M\) is torsion\-/free, then the topology of \(\HH^n(G,M)\) is the \(p\)\=/adic topology.
		\item If \(M\) is \(p\)\=/adically complete and separated, then \(\HH^n(G,M)\) is weakly \(p\)\=/complete.
	\end{enumerate}
\end{proposition}
\begin{proof}
	If \(M\) is torsion\-/free, then, for each \(i \in \N\), the short exact sequence
	\[
		0 \rightarrow M \xrightarrow{p^i} M \rightarrow M/p^i M \rightarrow 0
	\]
	implies that
	\[
		\Img\left(\HH^n(G,p^i M) \rightarrow \HH^n(G,M)\right) = p^i \HH^n(G,M),
	\]
	and thus, the topology of \(\HH^n(G,M)\) is the \(p\)\=/adic topology by Proposition~\ref{proposition:linear}.

	By the second point of Proposition~\ref{proposition:weakly-p-complete}, the module \(\varprojlim^1 \HH^{n-1}(G,M/p^i M)\) is weakly \(p\)\=/complete.
	By Remark~\ref{remark:complete-separated-weakly-complete} and Proposition~\ref{proposition:completness-projective-system}, the module \(\varprojlim \HH^n(G,M/p^i M)\) is \(p\)\=/adically complete and separated, and thus, weakly \(p\)\=/complete.
	The second statement then follows from the combination of Proposition~\ref{proposition:lim-cohomology} and the first point of Proposition~\ref{proposition:weakly-p-complete}.
\end{proof}

\begin{proposition} \label{proposition:cohomological-dimension}
	Let \(N \in \N\).
	Assume that \(G\) is profinite with \(p\)\=/cohomological dimension \(\leq N\).
	\begin{enumerate}
		\item If \(M\) is \(p\)\=/adically complete and separated, then, for each \(n > N\), the group \(\HH^n(G,M)\) is trivial.
		\item If \(M\) is divisible, then the group \(\HH^N(G,M)\) is divisible.
	\end{enumerate}
\end{proposition}
\begin{proof}
	By cohomological dimension, for each \(n > N\) and each \(i > 0\), the group \(\HH^n(G,M/p^i M)\) is trivial.
	Moreover, for each \(i \in \N\), the short exact sequence
	\[
		0 \rightarrow p^i M/p^{i+1} M \rightarrow M/p^{i+1} M \rightarrow M/p^i M \rightarrow 0
	\]
	induces an exact sequence
	\[
		\HH^N(G, M/p^{i+1} M) \rightarrow \HH^N(G,M/p^i M) \rightarrow \HH^{N+1}(G,p^i M/p^{i+1} M),
	\]
	and the group \(\HH^{N+1}(G,p^i M/p^{i+1} M)\) is trivial by cohomological dimension, which implies that the system \((\HH^N(G,M/p^i M))\) satisfies the Mittag-Leffler condition, and thus, the group \(\varprojlim^1  \HH^N(G,M/p^i M)\) is trivial.
	Therefore, by Proposition~\ref{proposition:lim-cohomology}, if \(M\) is \(p\)\=/adically complete and separated, then, for each \(n > N\), the group \(\HH^n(G,M)\) is trivial.

	If \(M\) is divisible, then the short exact sequence
	\[
		0 \rightarrow M[p] \rightarrow M \xrightarrow{p} M \rightarrow 0
	\]
	induces an exact sequence
	\[
		\HH^N(G,M) \xrightarrow{p} \HH^N(G,M) \rightarrow \HH^{N+1}(G,M[p]),
	\]
	and the group \(\HH^{N+1}(G,M[p])\) is trivial by cohomological dimension.
\end{proof}

\subsection{\texorpdfstring{\(p\)\=/}{p-}adic Banach representations}
\label{subsec:p-adic-banach-representations}

We review some results concerning topological \(\Qp\)\=/vector spaces (see for instance~\cite{Schneider2002}).

We will consider Banach spaces over \(\Qp\) up to norm equivalence, and thus, we set the following notation.
A \emph{\(p\)\=/adic Banach space} is a topological \(\Qp\)\=/vector space whose topology is the one of a Banach space over \(\Qp\).

Let \(U\) be a \(\Qp\)\=/vector space, and let \(\Uc\) be a \(\Zp\)\=/submodule of \(U\) which spans \(U\) over \(\Qp\), that is, it holds \(\Qp \otimes_{\Zp} \Uc \similarrightarrow U\).
If \(U\) is equipped with the locally convex topology defined by \(\Uc\), that is, the linear topology generated by the subgroups \(\{p^i \Uc\}_{i \in \N}\), then \(U\) is a \(p\)\=/adic Banach space if and only if \(\Uc\) is \(p\)\=/adically complete and separated.
Moreover, if \(\Uc^\prime\) is another \(\Zp\)\=/submodule of \(U\) which spans \(U\) over \(\Qp\), then \(\Uc\) and \(\Uc^\prime\) define the same locally convex topology on \(U\) if and only if \(\Uc\) and \(\Uc^\prime\) are commensurable, that is,  there exists \(s \in \N\) such that \(p^s \Uc \subseteq \Uc^\prime \subseteq p^{-s} \Uc\).

A \emph{lattice} in a \(p\)\=/adic Banach space \(X\) is an open \(\Zp\)\=/submodule \(\Xc\) in \(X\) which spans \(X\) over \(\Qp\) and which is separated and complete for the \(p\)\=/adic topology.
Note that the topology induced by a \(p\)\=/adic Banach space \(X\) on a lattice \(\Xc\) is the \(p\)\=/adic topology, and the topology of \(X\) coincides with the locally convex topology defined by \(\Xc\).

\begin{lemma} \label{lemma:banach-dense-subspace}
	Let \(X\) be a \(p\)\=/adic Banach space, and let \(\Xc\) be a lattice in \(X\).
	Let \(\pi: X \rightarrow X/\Xc\) be the quotient map.
	Let \(Y\) be a \(\Qp\)\=/vector subspace of \(X\).
	The subspace \(Y\) is dense in \(X\) if and only if \(\pi(Y) = X/\Xc\).
\end{lemma}
\begin{proof}
	The quotient map \(\pi\) is surjective and continuous, and the topological group \(X/\Xc\) is discrete.
	Therefore, if \(Y\) is dense in \(X\), then, by~\cite[Théorème~1 b)]{Bourbaki1971:TG:1-4}, it holds
	\[
		X/\Xc = \pi(X)=\pi(\overline{Y}) \subseteq \overline{\pi(Y)} = \pi(Y).
	\]

	Conversely, if \(\pi(Y) = X/\Xc\), then \(Y + \Xc = X\).
	Moreover, for each \(n \in \Z\), since \(p^n Y = Y\), it holds
	\begin{equation} \label{eq:y-lattice}
		Y + p^n \Xc = p^n Y + p^n \Xc = p^n(Y + \Xc) = p^n X = X.
	\end{equation}
	Finally, recall~\cite[I \S~7]{Schneider2002} that there are isomorphisms
	\begin{align}
		X & \simeq \varprojlim_n X/p^n\Xc, \label{eq:x-complete} \\
		\overline{Y} & \simeq \varprojlim_n (Y + p^n \Xc)/p^n\Xc. \label{eq:y-complete}
	\end{align}
	The combination of the equations~\eqref{eq:y-lattice},~\eqref{eq:x-complete} and~\eqref{eq:y-complete} yields
	\[
		\overline{Y} \simeq \varprojlim_n (Y + p^n \Xc)/p^n\Xc = \varprojlim_n X/p^n\Xc \simeq X.
	\]
\end{proof}

Let \(G\) be a compact and separated topological group.
A \emph{\(p\)\=/adic Banach representation of \(G\)} is a \(p\)\=/adic Banach space equipped with a continuous and \(\Qp\)\=/linear action of \(G\).
A \emph{\(G\)\=/stable lattice} in a \(p\)\=/adic Banach representation \(X\) of \(G\) is a lattice in \(X\) stable under the action of \(G\).

\begin{remark}
	Since \(G\) is compact, if \(X\) is a \(p\)\=/adic Banach representation of \(G\), then there exists a \(G\)\=/stable lattice in \(X\) (see~\cite[\S~6.5]{Emerton2017}).
\end{remark}

Let \(X\) be a \(p\)\=/adic Banach representation of \(G\), and let \(\Xc\) be a \(G\)\=/stable lattice in \(X\).
Let \(n \in \N\).
We study the topological abelian groups \(\HH^n(G,X)\) and \(\HH^n(G,\Xc)\).

Proposition~\ref{proposition:cohomology-zp-modules} yields the following.

\begin{proposition} \label{proposition:cohomology-lattice}
	The topology of \(\HH^n(G,\Xc)\) is the \(p\)\=/adic topology.
	Moreover, the \(\Zp\)\=/module \(\HH^n(G,\Xc)\) is weakly \(p\)\=/complete.
\end{proposition}

The strict exact sequence of topological \(G\)\=/modules
\[
	0 \rightarrow \Xc \rightarrow X \rightarrow X/\Xc \rightarrow 0,
\]
induces an exact sequence of topological groups
\begin{equation} \label{eq:torsion-divisible}
	\HH^n(G,\Xc) \rightarrow \HH^n(G,X) \rightarrow \HH^n(G,X/\Xc)  \rightarrow \HH^{n+1}(G,\Xc),
\end{equation}
and
\begin{itemize}
	\item the topological group \(\HH^n(G,X/\Xc)\) is discrete and torsion,
	\item the image of \(\HH^n(G,X/\Xc)\) in \(\HH^{n+1}(G,\Xc)\) is the torsion subgroup \(\HH^{n+1}(G,\Xc)_\torsion\) of \(\HH^{n+1}(G,\Xc)\),
	\item the group \(\HH^n(G,X)\) is a \(\Qp\)\=/vector space, and the image of \(\HH^n(G,\Xc)\) in \(\HH^n(G,X)\) spans \(\HH^n(G,X)\) over \(\Qp\),
	\item the image of \(\HH^n(G,X)\) in \(\HH^n(G,X/\Xc)\) is the divisible subgroup \(\HH^n(G,X/\Xc)_\divisible\) of \(\HH^n(G,X/\Xc)\).
	Indeed, on the one hand, since \(\HH^n(G,X)\) is divisible, its image in \(\HH^n(G,X/\Xc)\) is divisible.
	On the other hand, the module \(\HH^n(G,\Xc)\) is weakly \(p\)\=/complete by Proposition~\ref{proposition:cohomology-lattice}, which implies that \(\HH^n(G,\Xc)_\divisible\) is trivial by Lemma~\ref{lemma:weakly-p-complete}.
\end{itemize}

\begin{proposition} \label{proposition:locally-convex}
	The topological group \(\HH^n(G,X)\) is a topological \(\Qp\)\=/vector
space whose topology is the locally convex topology defined by the image of
\(\HH^n(G,\Xc)\) in \(\HH^n(G,X)\).
\end{proposition}
\begin{proof}
	For each \(i \in \N\), the short exact sequence
	\[
		0 \rightarrow \Xc \xrightarrow{p^i} X \rightarrow X/p^i \Xc \rightarrow 0
	\]
	implies that
	\[
		\Img\left(\HH^n(G,p^i \Xc) \rightarrow \HH^n(G,X)\right) = p^i \Img\left(\HH^n(G,\Xc) \rightarrow \HH^n(G,X)\right),
	\]
	and thus, the statement follows from Proposition~\ref{proposition:linear}.
\end{proof}

\begin{proposition} \label{proposition:cohomology-banach}
	If the torsion subgroup \(\HH^n(G,\Xc)_\torsion\) of \(\HH^n(G,\Xc)\) has finite exponent, then \(\HH^n(G,X)\) is a \(p\)\=/adic Banach space.
\end{proposition}
\begin{proof}
	Since \(\HH^n(G,\Xc)\) is weakly \(p\)\=/complete by Proposition~\ref{proposition:cohomology-lattice} and \(\HH^n(G,\Xc)_\torsion\) has finite exponent, the module \(\HH^n(G,\Xc)\) is \(p\)\=/adically complete and separated by Lemma~\ref{lemma:weakly-p-complete}.
	Therefore, by Proposition~\ref{proposition:p-adic-topology-torsion}, the quotient \(\HH^n(G,\Xc)/\HH^n(G,\Xc)_\torsion\), which is isomorphic to the image of \(\HH^n(G,\Xc)\) in \(\HH^n(G,X)\), is \(p\)\=/adically complete and separated, which implies, by Proposition~\ref{proposition:locally-convex}, that  \(\HH^n(G,X)\) is a \(p\)\=/adic Banach space.
\end{proof}

\begin{corollary} \label{corollary:h0-banach}
	The topological \(\Qp\)\=/vector space \(\HH^0(G,X)\) is a \(p\)\=/adic Banach space.
\end{corollary}
\begin{proof}
	The module \(\HH^0(G,\Xc) \similarrightarrow \Xc^G \subseteq \Xc\) is torsion\-/free, and thus, it satisfies the hypotheses of Proposition~\ref{proposition:cohomology-banach}.
\end{proof}

A \emph{\(p\)\=/adic representation of \(G\)} is a finite dimensional \(p\)\=/adic Banach representation of \(G\).

\begin{lemma} \label{lemma:finite-torsion}
	If \(X\) is a \(p\)\=/adic representation of \(G\), then \(\HH^1(G,\Xc)_\torsion\) is finite.
\end{lemma}
\begin{proof}
	The short exact sequence~\eqref{eq:torsion-divisible} induces an isomorphism of groups
	\begin{equation} \label{eq:tor-div}
		\HH^1(G,\Xc)_\torsion \similarrightarrow (X/\Xc)^{G}/(X/\Xc)^{G}_\divisible.
	\end{equation}
	We prove that the group \((X/\Xc)^{G}/(X/\Xc)^{G}_\divisible\) is finite.
	We denote by \(P = \Hom_{\Zp}((X/\Xc)^G,\Qp/\Zp)\) the Pontryagin dual of the discrete torsion group \((X/\Xc)^G\).
	The Pontryagin duality of the inclusion \((X/\Xc)^G \subseteq X/\Xc\) implies that \(P\) is a quotient of the free \(\Zp\)\=/module of finite rank \( \Hom_{\Zp}(X/\Xc,\Qp/\Zp) \simeq \Hom_{\Zp}(\Xc,\Zp)\), and thus, the \(\Zp\)\=/module \(P\) is finitely generated, and hence, the group \(P_\torsion\) is finite.
	Moreover, the Pontryagin dual of the quotient \((X/\Xc)^{G}/(X/\Xc)^{G}_\divisible\) is the torsion subgroup \(P_\torsion\) of \(P\).
	Therefore, by the isomorphism~\eqref{eq:tor-div}, the group \(\HH^1(G,\Xc)_\torsion\) is finite.
\end{proof}

\begin{corollary} \label{corollary:h1-representation-banach}
	If \(X\) is a \(p\)\=/adic representation of \(G\), then \(\HH^1(G,X)\) is a \(p\)\=/adic Banach space.
\end{corollary}
\begin{proof}
	The statement follows from Proposition~\ref{proposition:cohomology-banach} and Lemma~\ref{lemma:finite-torsion}.
\end{proof}

We will also need the following result.
Assume that \(G\) is procyclic, and let \(g\) be a topological generator of \(G\).
If \(M\) is a topological \(G\)\=/module, then there exists an exact sequence of groups
\begin{equation} \label{eq:cohomology-zp}
	0 \rightarrow \HH^0(G,M) \rightarrow M \xrightarrow{(g-1)} M \rightarrow  \HH^1(G,M) \rightarrow 0,
\end{equation}
where the map \(M \rightarrow \HH^1(G,M)\) is the morphism that sends an element \(m \in M\) to the class of the cocycle which maps the generator \(g\) to \(m\), which is surjective since a cocycle \(c \in \Zr^1(G,M)\) is determined by its value \(c(g)\) by continuity and by procyclicity of \(G\).

\begin{lemma} \label{lemma:cohomology-gamma}
	If the map
	\[
		\begin{split}
			(g - 1) : X & \rightarrow X \\
			x & \mapsto g(x) - x
		\end{split}
	\]
	is an automorphism of \(p\)\=/adic Banach spaces, then the group \(\HH^0(G,\Xc)\) is trivial, and the group \(\HH^1(G,\Xc)\) is torsion with finite exponent.
\end{lemma}
\begin{proof}
	Since \((g-1)\) is an automorphism of \(X\), the map \((g-1)\) is injective, and \((g -1)\Xc\) is a lattice in \(X\), and thus, the lattices \((g -1)\Xc\) and \(\Xc\) are commensurable (see~\cite[Lemma~6.5.1]{Emerton2017}).
	We conclude using the exact sequence~\eqref{eq:cohomology-zp}.
\end{proof}

\section{\texorpdfstring{\(p\)\=/}{p-}adic period rings and perfectoid fields}
\label{sec:p-adic-period-rings-and-perfectoid-fields}

\subsection{Perfectoid fields}
\label{subsec:perfectoid-fields}

We recall the definition of a perfectoid field introduced by Scholze~\cite[\S~3]{Scholze2012}.

A \emph{non\-/archimedean} field is a topological field whose topology is induced by a non\-/trivial valuation of rank \(1\).

A \emph{perfectoid} field is a complete non\-/archimedean field \(k\) of positive residue characteristic, say \(p\), whose associated rank\-/\(1\)\=/valuation is non\-/discrete, and such that the Frobenius map
\[
	\begin{split}
		\Oc_k/(p) & \rightarrow \Oc_k/(p) \\
		x & \mapsto x^p
	\end{split}
\]
is surjective.

\begin{remark} \label{remark:perfectoid-deeply-ramified}
	Coates and Greenberg~\cite{CoatesGreenberg1996} have introduced the notion of \emph{deeply ramified} fields, which is also used in the articles of Bondarko~\cite{Bondarko2003} and Iovita and Zaharescu~\cite{IovitaZaharescu1999:1}.
	Recall that a perfectoid field is deeply ramified, and a complete deeply ramified field whose valuation is of rank \(1\) is perfectoid (see~\cite[Remark~3.3]{Scholze2012} and~\cite[Proposition~6.6.6]{GabberRamero2003}).
	In particular, an algebraic extension \(L\) of \(K\) is deeply ramified if and only if the field \(\hat{L}\) is perfectoid.
\end{remark}

An extension of non\-/archimedean fields is a field extension \(k^\prime/k\) such that both \(k^\prime\) and \(k\) are non\-/archimedean fields and the valuation of \(k^\prime\) is an extension of the valuation of \(k\).
In view of Remark~\ref{remark:perfectoid-deeply-ramified}, recall the following~\cite[Corollary~6.6.16]{GabberRamero2003}.

\begin{proposition} \label{proposition:perfectoid}
	Let \(k^\prime/k\) be an algebraic extension of non\-/archimedean fields.
	\begin{enumerate}
		\item If \(\hat{k}\) is perfectoid, then \(\hat{k}^\prime\) is perfectoid.
		\item If \(k^\prime/k\) is finite and \(\hat{k}^\prime\) is perfectoid, then \(\hat{k}\) is perfectoid.
	\end{enumerate}
\end{proposition}

\begin{examples} \label{examples:perfectoid}
	We recall some examples of perfectoid fields.
	\begin{enumerate}
		\item A complete non\-/archimedean field of characteristic \(p\) is perfectoid if and only if it is perfect.
		\item The field \(\Cp\) is perfectoid.
		\item By a result of Sen~\cite{Sen1972} (see also~\cite[Theorem~2.13]{CoatesGreenberg1996}), if \(L\) is a Galois extension of \(K\) whose Galois group is a \(p\)\=/adic Lie group in which the inertia subgroup is infinite, then \(\hat{L}\) is perfectoid.
		In particular, if \(L\) is an infinitely ramified \(\Zp\)\=/extension of \(K\), then \(\hat{L}\) is perfectoid.
		Therefore, the completion of the \(p^\infty\)\=/cyclotomic extension \(K_\cyc\) generated over \(K\) by all the \(p\)\=/power roots of unity, and the completion of the \(\Zp\)\=/cyclotomic extension \(K_\infty\) of \(K\) contained in \(K_\cyc\) are both perfectoid.
		\item The completion of the maximal abelian extension \(K^\ab\) of \(K\) is perfectoid.
		\item Let \((p^{1/p^n})_{n \in \N}\) be a sequence of \(p\)\=/power roots of \(p\) in \(\Qpbar\) such that \((p^{1/p^{n+1}})^p = p^{1/p^n}\).
		For each \(n \in \N\), let \(K(p^{1/p^n})\) be the extension of \(K\) generated by \(p^{1/p^n}\), and let \(K(p^{1/p^{\infty}}) = \bigcup_{n \in \N} K(p^{1/p^n})\).
		The completion of the field \(K(p^{1/p^{\infty}})\) is perfectoid.
		Note that \(K(p^{1/p^{\infty}})\) is not Galois over \(K\).
		\end{enumerate}
		Note that the completion of the maximal unramified extension \(K^\ur\) of \(K\) is not perfectoid.
\end{examples}

\subsection{\texorpdfstring{\(p\)\=/}{p-}adic period rings}
\label{subsec:p-adic-period-rings}

We review the \(p\)\=/adic period rings introduced by Fontaine~\cite{Fontaine1994:II}.

The \emph{ring of \(p\)\=/adic periods} \(\BdR^+\) is a complete discrete valuation ring endowed with an action of \(G_K\), whose residue field is \(\Cp\), and which has a structure of \(K\)\=/algebra.
The \emph{field of \(p\)\=/adic periods} \(\BdR\) is the field of fractions of \(\BdR^+\).
There is a natural filtration \((\Fil^i \BdR)_{i \in \Z}\) on \(\BdR\), which is stable under the action of \(G_K\), by its fractional ideals, that is, if \(\xi\) is a uniformiser of \(\BdR^+\), then
\[
	\Fil^i \BdR = \BdR^+ \cdot \xi^i, i \in \Z,
\]
and we set \(\Gr^i \BdR = \Fil^i \BdR/\Fil^{i+1} \BdR\).

We recall the relation between uniformisers of \(\BdR^+\) and perfectoid fields (see~\cite[Theorem~5.2]{IovitaZaharescu1999:1} or~\cite[\S~3]{FarguesFontaine2018}).

\begin{proposition} \label{proposition:perfectoid-uniformiser}
	Let \(L\) be an algebraic extension of \(K\).
	The field \(\hat{L}\) is perfectoid if and only if there exists a uniformiser \(t_{\hat{L}}\) of \(\BdR^+\) such that \(t_{\hat{L}} \in (\BdR^+)^{G_L}\).
\end{proposition}

\begin{notation}
	Let \(\chi: G_{\Qp} \rightarrow \Zp^\ast\) be the cyclotomic character.
	Let \(i \in \Z\).
	We denote by \(\Zp(i)\) the free \(\Zp\)\=/module of rank \(1\) on which \(G_{\Qp}\) acts by \(\chi^i\).
	If \(L\) is an algebraic extension of \(\Qp\) and if \(M\) is a \(\Zp\)\=/module equipped with a \(\Zp\)\=/linear action of \(G_L\), then, the \emph{\(i\)\=/th Tate twist} \(M(i)\) of \(M\) is defined by
	\[
		M(i) = M \otimes_{\Zp} \Zp(i),
	\]
	on which \(G_L\) acts by \(g(m \otimes z) = g(m) \otimes g(z)= \chi^i(g) (g(m)\otimes z)\), for all \(g \in G_L\), \(m \in M\) and \(z \in \Zp(i)\).
\end{notation}

The choice of a compatible system of \(p\)\=/power roots of unity \((\zeta_{p^n})_{n \in \N}\), with \(\zeta_{p^{n+1}}^p = \zeta_{p^n}\) and \(\zeta_1 = 1\), allows to define a uniformiser of \(\BdR^+\) invariant by \(G_{K_\cyc}\), classically denoted simply by \(t\), on which \(G_K\) acts by \(g(t) = \chi(g)\cdot t\) for each \(g \in G_K\), and which, for each \(i \in \Z\), induces an isomorphism of \(G_K\)\=/modules
\begin{equation} \label{eq:iso-gr-cp}
		\Gr^i \BdR \similarrightarrow \Cp(i).
\end{equation}

The field \(\BdR\) is equipped with a topology, the so-called \emph{canonical} topology, which is coarser than the valuation topology from \(\BdR^+\).
The action of \(G_K\) on \(\BdR\) endowed with the canonical topology is continuous.
We will consider \(\BdR\) and its subquotients endowed with the canonical topology.
For each \(i \in \Z\) and \(j \in  \N\), the quotient \(\Fil^i \BdR/\Fil^{i+j} \BdR\) is a \(p\)\=/adic Banach representation of \(G_K\).
In particular, the isomorphisms~\eqref{eq:iso-gr-cp} are isomorphisms of \(p\)\=/adic Banach representations of \(G_K\).
For each \(i \in \Z\), the filtration \(\Fil^i \BdR\) is closed in \(\BdR\), and if \(\xi\) is a uniformiser of \(\BdR^+\), then, for each \(j \in \Z\), multiplication by \(\xi^j\) induces an isomorphism of topological modules \(\Fil^i \BdR \similarrightarrow \Fil^{i+j} \BdR\).
Moreover, there are isomorphisms of topological \(G_K\)\=/modules
\[
	\begin{split}
	\Fil^i \BdR & \similarrightarrow \varprojlim_j \Fil^i \BdR/\Fil^{i+j} \BdR, \\
	\BdR & \similarrightarrow \varinjlim_i \Fil^i \BdR,
	\end{split}
\]
and thus, the space \(\Fil^i \BdR\) is a Fréchet space and \(\BdR\) is an ind-Fréchet space~\cite[\S~2C]{Fontaine2020}.

The field \(\BdR\) contains \(K_0\)\=/subalgebras \(\Bcris \subset \Bst \subset \BdR\), the \emph{ring of crystalline periods} \(\Bcris\) and the \emph{ring of semistable periods} \(\Bst\).
The ring \(\Bcris\) contains \(t\).
Both \(\Bcris\) and \(\Bst\) are stable under the action of \(G_K\), and are equipped with a \(\sigma\)\=/semilinear injective endomorphism \(\varphi\) commuting with the action of \(G_K\) and satisfying \(\varphi(t) = p t\).
Moreover, there exists a \(K_0\)\=/linear surjective map \(N : \Bst \rightarrow \Bst\), which commutes with the action of \(G_K\) and satisfies the relation \(N \circ \varphi = p \varphi \circ N\), and whose kernel is \(\Bcris = \Bst^{N=0}\).
Furthermore, the natural maps
\begin{equation} \label{eq:Bcris_BdR}
	K \otimes_{K_0} \Bcris \rightarrow K \otimes_{K_0} \Bst \rightarrow \BdR
\end{equation}
are injective.

Let
\[
	\Be = \Bcris^{\varphi = 1} = \{ b \in \Bcris, \varphi(b)=b \}.
\]
There exists a strict exact sequence of topological \(G_K\)\=/modules, the so-called \emph{fundamental exact sequence},
\begin{equation} \label{eq:fundamental}
	0 \rightarrow \Qp \rightarrow \Be \rightarrow \BdR/\BdR^+ \rightarrow 0,
\end{equation}
where the map \(\Be \rightarrow \BdR/\BdR^+\) is the composition of the quotient map \(\BdR \rightarrow \BdR/\BdR^+\) with the inclusion of \(\Be\) in \(\BdR\) (see for instance~\cite[III \S~3]{Colmez1998}).

\subsection{Galois cohomology of \texorpdfstring{\(\Cp(i)\)}{Cp(i)}}
\label{subsec:galois-cohomology-of-cp-i}

We review the computation of the Galois cohomology of \(\Cp(i)\).
These results are standard and follow from Tate's method~\cite{Tate1967}.
We review part of their proofs, which we need to study the topology on these cohomology groups defined in the previous section~\ref{sec:continuous-group-cohomology}.

Let \(L\) be an algebraic extension of \(K\).

We first recall the Ax\--Sen\--Tate Theorem~\cite{Tate1967,Ax1970,Sen1969}.

\begin{theorem}[Ax--Sen--Tate] \label{theorem:ax-sen-tate}
	The natural inclusion \(\hat{L} \subseteq \Cp\) induces an isomorphism of topological \(\Zp\)\=/modules \(\HH^0(L,\Oc_{\Cp}) \similarrightarrow \Oc_{\hat{L}}\), and an isomorphism of \(p\)\=/adic Banach spaces \(\HH^0(L,\Cp) \similarrightarrow \hat{L}\).
\end{theorem}

\begin{theorem} \label{theorem:cohomology-cp-perfectoid}
	Let \(i \in \Z\).
	If \(\hat{L}\) is perfectoid, then the following holds.
	\begin{enumerate}
		\item There exist an isomorphism of topological \(\Zp\)\=/modules \(\HH^0(L,\Oc_{\Cp}(i)) \similarrightarrow \Oc_{\hat{L}}\), and an isomorphism of \(p\)\=/adic Banach spaces \(\HH^0(L,\Cp(i)) \similarrightarrow \hat{L}\).
		\item For each \(n > 0\), the group \(\HH^n(L,\Oc_{\Cp}(i))\) is torsion with finite exponent, and the group \(\HH^n(L,\Cp(i))\) is trivial.
	\end{enumerate}
\end{theorem}
\begin{proof}
	By Proposition~\ref{proposition:perfectoid-uniformiser}, there exists a uniformiser \(t_{\hat{L}}\) of \(\BdR^+\) invariant by \(G_L\).
	Therefore, multiplication by \(t_{\hat{L}}^{-i}\) in \(\BdR\) induces an isomorphism of \(p\)\=/adic Banach representation of \(G_L\)
	\[
		\Cp(i) \similarrightarrow \Gr^i \BdR \similarrightarrow \Gr^0 \BdR \similarrightarrow \Cp.
	\]
	The first statement then follows from the Ax\--Sen\--Tate Theorem~\ref{theorem:ax-sen-tate}, and the second statement follows from Tate's method~\cite{Tate1967}, or more generally from perfectoid techniques~\cite{Scholze2012}, which imply that, for each \(n > 0\), the \(\Oc_{\hat{L}}\)\=/module \(\HH^n(L,\Oc_{\Cp})\) is almost zero, that is, it holds
	\[
		\mf_{\hat{L}} \cdot \HH^n(L,\Oc_{\Cp}) = 0,
	\]
	and, in particular, the group \(\HH^n(L,\Oc_{\Cp})\) is \(p\)\=/torsion and
	\[
		\HH^n(L,\Cp) \similarrightarrow \Qp \otimes_{\Zp} \HH^n(L,\Oc_{\Cp})
	\]
	is trivial.
\end{proof}

The composition of the cyclotomic character \(\chi : G_{\Qp} \rightarrow \Zp^\ast\) with the \(p\)\=/adic logarithm \(\log_p : \Zp^\ast \rightarrow \Zp\) defines a continuous group homomorphism \(\log_p \circ \chi \in \Hom_\cont(G_{\Qp},\Zp)=\HH^1(\Qp,\Zp)\), and the restriction of \(\log_p \circ \chi\) to \(G_L\) is denoted by \((\log_p \circ \chi)_L \in \HH^1(L,\Zp)\).
The composition of \((\log_p \circ \chi)_L\) with the natural inclusion of \(\Zp\) in \(\Oc_{\Cp}\) defines an element \([\log_p \circ \chi]_L \in \HH^1(L,\Oc_{\Cp})\).

\begin{theorem} \label{theorem:cohomology-cp-non-perfectoid}
	If \(\hat{L}\) is not perfectoid, then the following holds.
	\begin{enumerate}
		\item Let \(i \neq 0\) be an integer.
		For each \(n \in \N\), the group \(\HH^n(L,\Oc_{\Cp}(i))\) is torsion with finite exponent, and the group \(\HH^n(L,\Cp(i))\) is trivial.
		\item The group \(\HH^1(L,\Oc_{\Cp})\) is \(p\)\=/adically complete and separated, and its torsion subgroup has finite exponent.
		Moreover, the element \([\log_p \circ \chi]_L \in \HH^1(L,\Oc_{\Cp})\) is non\-/torsion and induces an isomorphism of \(p\)\=/adic Banach spaces
		\[
			\begin{split}
				\HH^1(L,\Cp) & \similarrightarrow \hat{L} \\
				[\log_p \circ \chi]_L & \mapsto 1.
			\end{split}
		\]
		\item For each integer \(n >1\), the group \(\HH^n(L,\Oc_{\Cp})\) is torsion with finite exponent, and the group \(\HH^n(L,\Cp)\) is trivial.
	\end{enumerate}
\end{theorem}

\begin{remark}
	If the valuation of \(L\) is discrete, then Theorem~\ref{theorem:cohomology-cp-non-perfectoid} has been proved by several authors following Tate's method~\cite{Tate1967}, for instance~\cite[\S~3]{Fontaine2003} and~\cite[\S~4]{BarthelSchlankStapletonWeinstein2024}.
	Iovita and Zaharescu~\cite[\S~3]{IovitaZaharescu1999:1} have adapted Tate's method to compute \(\HH^n(L,\Cp(i))\) for a general algebraic extension \(L\) of \(\Qp\) .
\end{remark}

\begin{proof}
	We first establish the following two observations.

	Firstly, by Proposition~\ref{proposition:cohomology-zp-modules}, for each \(n \in \N\) and \(i \in \Z\), the group \(\HH^n(L,\Oc_{\Cp}(i))\) is weakly \(p\)\=/complete, and thus, \(p\)\=/adically complete by Lemma~\ref{lemma:weakly-p-complete}.
	In particular, concerning \(\HH^1(L,\Oc_{\Cp})\), it remains to prove that its torsion subgroup \(\HH^1(L,\Oc_{\Cp})_\torsion\) has finite exponent, which would implies that \(\HH^1(L,\Oc_{\Cp})\) is \(p\)\=/adically separated by Lemma~\ref{lemma:weakly-p-complete}, and to prove the property of the element \([\log_p \circ \chi]_L \).

	Secondly, we may prove the statement for a finite Galois extension \(L^\prime\) of \(L\).
	Indeed, by Proposition~\ref{proposition:perfectoid}, the field \(\hat{L}^\prime\) is not perfectoid.
	If the statements hold for \(L^\prime\), then, by Corollary~\ref{corollary:lim-cohomology-linear}, for each \(n \in \N\) and \(i \in \Z\), the system \((\HH^n(L^\prime,\Oc_{\Cp}(i)/p^j))_{j \in \N}\) satisfies the Mittag-Leffler condition and there are isomorphisms of topological groups
	\[
		\HH^n(L^\prime,\Oc_{\Cp}(i)) \similarrightarrow \varprojlim_j \HH^n(L^\prime,\Oc_{\Cp}(i))/p^j \similarrightarrow \varprojlim_j \HH^n(L^\prime,\Oc_{\Cp}(i)/p^j).
	\]
	Therefore, if \(\Delta = \Gal(L^\prime/L)\) denotes the Galois group of \(L^\prime/L\), then Proposition~\ref{proposition:spectral-sequence} applies and yields a convergent spectral sequence
	\begin{equation} \label{eq:spectral-sequence-delta}
		\HH^a(\Delta,\HH^b(L^\prime,\Oc_{\Cp}(i))) \Rightarrow \HH^{a+b}(L,\Oc_{\Cp}(i)).
	\end{equation}
	Since \(\Delta\) is finite, for each \(a \geq 1\), \(b \in \N\) and \(i \in \Z\) , the group \(\HH^a(\Delta,\HH^b(L^\prime,\Oc_{\Cp}(i)))\) is torsion with finite exponent (see~\cite[I \S~2.4 Proposition~9]{Serre1994}).
	We then deduce from the spectral sequence~\eqref{eq:spectral-sequence-delta} that the statement for \(L^\prime\) implies all the cases of the statement for \(L\) except the case concerning \(\HH^1(L,\Oc_{\Cp})\).
	About \(\HH^1(L,\Oc_{\Cp})\), the spectral sequence~\eqref{eq:spectral-sequence-delta} induces the inflation-restriction exact sequence
	\begin{equation} \label{eq:delta-inflation-restriction}
		0 \rightarrow \HH^1(\Delta,\Oc_{\hat{L}^\prime}) \rightarrow \HH^1(L,\Oc_{\Cp}) \rightarrow \HH^1(L^\prime,\Oc_{\Cp})^\Delta \rightarrow \HH^2(\Delta,\Oc_{\hat{L}^\prime}).
	\end{equation}
	Since the groups \(\HH^1(\Delta,\Oc_{\hat{L}^\prime})\) and \(\HH^2(\Delta,\Oc_{\hat{L}^\prime})\) are torsion with finite exponent, and the element \([\log_p \circ \chi]_L \in \HH^1(L,\Oc_{\Cp})\) is mapped to \([\log_p \circ \chi]_{L^\prime} \in \HH^1(L^\prime,\Oc_{\Cp})^\Delta\) by the restriction map, we deduce from the sequence~\eqref{eq:delta-inflation-restriction} and Proposition~\ref{proposition:p-adic-topology-torsion} that the statement for \(L^\prime\) implies the statement for \(\HH^1(L,\Oc_{\Cp})\) .

	We proceed to the proof of the statement.
	On the one hand, the field \(\widehat{L_\cyc}\) is perfectoid by Proposition~\ref{proposition:perfectoid} and the result by Sen from Examples~\ref{examples:perfectoid},
	and, on the other hand, the field \(\hat{L}\) is not perfectoid by hypothesis.
	Therefore, by another application of Proposition~\ref{proposition:perfectoid}, the extension \(L_\cyc/L\) is infinite.
	Considering a finite Galois extension of \(L\) in \(L_\cyc\) if necessary, we may assume that \(L_\cyc/L\) is a totally ramified \(\Zp\)\=/extension.
	We set \(\Gamma = \Gal(L_\cyc/L)\), and thus, the morphism \(\log_p \circ \chi\) induces an isomorphism of topological groups : \(\Gamma \similarrightarrow \Zp\), and we denote by \(\gamma\) a topological generator of \(\Gamma\).

	Iovita and Zaharescu~\cite[\S~3]{IovitaZaharescu1999:1}, adapting Tate's method, have proved the following statements.
	(Note that Iovita and Zaharescu assume that \(L\) contains \(\Qp^\ur\).
	However, their results hold without this assumption and with similar proofs, see also~\cite{Ohkubo2010}.)
	\begin{itemize}
		\item There exists a continous \(\hat{L}\)\=/linear surjective map of \(p\)\=/adic Banach representations of \(\Gamma\)
		\[
			  \Tr : \widehat{L_\cyc} \rightarrow \hat{L},
		\]
		the so\-/called \emph{Tate's normalised trace map}, which is a section of the inclusion \(\hat{L} \subseteq \widehat{L_\cyc}\) as \(p\)\=/adic Banach representations of \(\Gamma\), and thus, it induces an isomorphism of \(p\)\=/adic Banach representations of \(\Gamma\)
		\begin{equation} \label{eq:decomposition-trace}
			\widehat{L_\cyc} \similarrightarrow \hat{L} \oplus \Ker(\Tr).
		\end{equation}
		\item The map
		\[
			(\gamma -1): \widehat{L_\cyc} \rightarrow \widehat{L_\cyc}
		\]
		is a \(\hat{L}\)\=/linear morphism of \(p\)\=/adic Banach spaces, which is trivial on \(\hat{L}\) and an automorphism of \(p\)\=/adic Banach spaces on \(\Ker(\Tr)\).
		Moreover, for each integer \(i \neq 0\), the map
		\[
			(\gamma -1): \widehat{L_\cyc}(i) \rightarrow \widehat{L_\cyc}(i)
		\]
		is an automorphism of \(p\)\=/adic Banach spaces.
		(Iovita and Zaharescu prove that the map \((\gamma -1)\) on \(\Ker(\Tr)\) (respectively \(\widehat{L_\cyc}(i)\)) is bijective with continuous inverse, which implies that \((\gamma -1)\) is an automorphism of \(p\)\=/adic Banach spaces by the open mapping theorem~\cite[Proposition~I.1.3]{Colmez1998}.)
	\end{itemize}

	By Theorem~\ref{theorem:cohomology-cp-perfectoid} and Corollary~\ref{corollary:lim-cohomology-linear}, for each \(n \in \N\) and \(i \in \Z\), the system \((\HH^n(L_\cyc,\Oc_{\Cp}(i)/p^j))_{j \in \N}\) satisfies the Mittag-Leffler condition and there are isomorphisms of topological groups
	\begin{equation} \label{eq:isomorphism-lim-cyc}
		\HH^n(L_\cyc,\Oc_{\Cp}(i)) \similarrightarrow \varprojlim_j \HH^n(L_\cyc,\Oc_{\Cp}(i))/p^j \similarrightarrow \varprojlim_j \HH^n(L_\cyc,\Oc_{\Cp}(i)/p^j).
	\end{equation}
	Again by Theorem~\ref{theorem:cohomology-cp-perfectoid}, for each \(n > 0\) and each \(i \in \Z\), the group \(\HH^n(L_\cyc,\Oc_{\Cp}(i))\) is torsion with finite exponent.
	Moreover, there are isomorphism of topological \(\Gamma\)\=/modules
	\begin{equation} \label{eq:isomorphism-tate-twist}
		\HH^n(L_\cyc,\Oc_{\Cp}(i)) \simeq \HH^n(L_\cyc,\Oc_{\Cp})(i).
	\end{equation}

	Proposition~\ref{proposition:spectral-sequence} applies and yields a convergent spectral sequence
	\begin{equation} \label{eq:spectral-sequence-cyc}
		\HH^a(\Gamma,\HH^b(L_\cyc,\Oc_{\Cp}(i))) \Rightarrow \HH^{a+b}(L,\Oc_{\Cp}(i)).
	\end{equation}
	Since the \(p\)\=/cohomological dimension of \(\Gamma \simeq \Zp\) is \(1\), the combination of the isomorphisms~\eqref{eq:isomorphism-lim-cyc}, Proposition~\ref{proposition:lim-cohomology}, and Proposition~\ref{proposition:cohomological-dimension} imply that \(\HH^a(\Gamma,\HH^b(L_\cyc,\Oc_{\Cp}(i)))\) is trivial unless \(a \in \{0,1\}\).
	Therefore, the spectral sequence~\eqref{eq:spectral-sequence-cyc} induces short exact sequences
	\begin{equation} \label{eq:ses-spectral-sequence-cyc}
		0 \rightarrow \HH^1(\Gamma,\HH^{n-1}(L_\cyc,\Oc_{\Cp}(i))) \rightarrow \HH^n(L,\Oc_{\Cp}(i)) \rightarrow \HH^n(L_\cyc,\Oc_{\Cp}(i))^\Gamma \rightarrow 0.
	\end{equation}

	We then deduce from the exact sequence~\eqref{eq:ses-spectral-sequence-cyc} that for each \(n > 1\) and each \(i \in \Z\), the group \(\HH^n(L,\Oc_{\Cp}(i))\) is torsion with finite exponent.

	For \(i \neq 0\), since \((\gamma - 1)\) is an automorphism of \(\widehat{L_\cyc}(i)\), the combination of the isomorphism~\eqref{eq:isomorphism-tate-twist} and Lemma~\ref{lemma:cohomology-gamma} implies that the group \(\HH^0(\Gamma,\HH^0(L_\cyc,\Oc_{\Cp}(i)))\) is trivial and the group \(\HH^1(\Gamma,\HH^0(L_\cyc,\Oc_{\Cp}(i)))\) is torsion with finite exponent.
	We then deduce from the exact sequence~\eqref{eq:ses-spectral-sequence-cyc} that \(\HH^0(L,\Oc_{\Cp}(i))\) is trivial and \(\HH^1(L,\Oc_{\Cp}(i))\) is torsion with finite exponent.

	Finally, let \(\Lc\) be a \(\Gamma\)\=/stable lattice in \(\Ker(\Tr)\).
	Then, the direct sum \(\Oc_{\hat{L}} \oplus \Lc\) defines a \(\Gamma\)\=/stable lattice in \(\hat{L}\oplus\Ker(\Tr)\).
	By the isomorphism~\eqref{eq:decomposition-trace}, the image of \(\Oc_{\hat{L}} \oplus \Lc\) in \(\widehat{L_\cyc}\) is commensurable with the lattice \(\Oc_{\widehat{L_\cyc}}\), and, without loss of generality, we may assume that the image of \(\Oc_{\hat{L}} \oplus \Lc\) in \(\widehat{L_\cyc}\) is contained in \(\Oc_{\widehat{L_\cyc}}\).
	Therefore, for each \(n \in \N\), there exists a group homomorphism
	\begin{equation} \label{eq:lattices-splitting}
		\HH^n(\Gamma,\Oc_{\hat{L}})\oplus\HH^n(\Gamma,\Lc) \rightarrow \HH^n(\Gamma,\Oc_{\widehat{L_\cyc}})
	\end{equation}
	whose kernel and cokernel are torsion groups with finite exponent.
	On the one hand, since \((\gamma - 1)\) is an automorphism of \(\Ker(\Tr)\), by Lemma~\ref{lemma:cohomology-gamma}, the group \(\HH^0(\Gamma,\Lc)\) is trivial and the group \(\HH^1(\Gamma,\Lc)\) is torsion with finite exponent.
	On the other hand, the action of \((\gamma - 1)\) on \(\hat{L}\) is trivial, and hence, it holds
	\[
		\begin{split}
			\HH^0(\Gamma,\Oc_{\hat{L}}) & = \Oc_{\hat{L}}, \\
			\HH^1(\Gamma,\Oc_{\hat{L}}) & = \Hom_\cont(\Gamma,\Oc_{\hat{L}}),
		\end{split}
	\]
	and the \(\Oc_{\hat{L}}\)\=/module \(\Hom_\cont(\Gamma,\Oc_{\hat{L}})\) is generated by \((\log_p \circ \chi)_L\).
	We then deduce from the exact sequence~\eqref{eq:ses-spectral-sequence-cyc}, the maps~\eqref{eq:lattices-splitting}, and Proposition~\ref{proposition:p-adic-topology-torsion} the statement concerning \(\HH^1(L,\Oc_{\Cp})\).
\end{proof}

We gather the computation of the Galois cohomology of \(\Cp(i)\) in Table~\ref{table:galois-cohomology-cp} in which there exists an isomorphism of \(p\)\=/adic Banach spaces between \(\HH^n(L,\Cp(i))\) and each cell.

\begin{table}[H]
	\centering
\begin{tabular}{|c|c|c|}
	\hline
	\(\HH^n(L,\Cp(i))\) & \(i=0\) & \(i \neq 0 \) \\
	\hline
	\(n=0\) & \(\hat{L}\) &
	\(\begin{cases}
		\hat{L} & \text{if \(\hat{L}\) is perfectoid} \\
		0 & \text{otherwise}
	\end{cases}\)
	\\
	\hline
	\(n=1\) &
	\(\begin{cases}
		0 & \text{if \(\hat{L}\) is perfectoid} \\
		\hat{L} & \text{otherwise}
	\end{cases}\)
	& 0 \\
	\hline
	\(n > 1\) & 0 & 0 \\
	\hline
\end{tabular}
\caption{Galois cohomology of \(\Cp(i)\).}
\label{table:galois-cohomology-cp}
\end{table}

\begin{corollary} \label{corollary:gal-cp}
	Let \(E\) be an extension of \(\Qp\) contained in \(L\).
	If \(\hat{L}\) is not perfectoid, then \(\hat{E}\) is not perfectoid and the restriction map \(\HH^1(E,\Cp) \rightarrow \HH^1(L,\Cp)\) is injective.
\end{corollary}
\begin{proof}
	The image of \([\log_p \circ \chi]_E \in \HH^1(E,\Cp)\) in \(\HH^1(L,\Cp)\) by the restriction map is \([\log_p \circ \chi]_L\).
	Since \(\hat{L}\) is not perfectoid, the element \([\log_p \circ \chi]_L \in \HH^1(L,\Cp)\) is non\-/trivial by Theorem~\ref{theorem:cohomology-cp-non-perfectoid}.
	Therefore, the element \([\log_p \circ \chi]_E\) is non\-/trivial, which implies, by Theorem~\ref{theorem:cohomology-cp-non-perfectoid}, that \(\hat{E}\) is not perfectoid and that there is a commutative diagram of \(p\)\=/adic Banach spaces
	\[
		\begin{tikzcd}
			& \HH^1(E,\Cp) \ar{d}[sloped]{\sim} \ar{r} &  \HH^1(L,\Cp) \ar{d}[sloped]{\sim} \\
			0 \ar{r} & \hat{E} \ar{r} &  \phantom{.}\hat{L}.
		\end{tikzcd}
	\]
\end{proof}

\begin{remark}
	He~\cite{He2025} has recently computed the Galois cohomology of a Henselian valued field of rank \(1\) extension of \(\Qp\).
	In particular, He's results recover and generalise the computation of the groups \(\HH^n(L,\Oc_{\Cp}(i))\) from this subsection~\ref{subsec:galois-cohomology-of-cp-i}.
\end{remark}

\subsection{Galois cohomology of \texorpdfstring{\(p\)\=/}{p-}adic period rings}
\label{subsec:galois-cohomology-of-p-adic-period-rings}

The Galois cohomology of \(\BdR\) and its subquotients may be computed from the Galois cohomology of \(\Gr^i \BdR \similarrightarrow \Cp(i)\).

Let \(L\) be an algebraic extension of \(K\).

\begin{proposition} \label{proposition:quotient-fil-bdr}
	Let \(i \in \Z\) and \(j \in \N\).
	\begin{enumerate}
		\item For each \(n \in \N\), the space \(\HH^n(L,\Fil^i \BdR/\Fil^{i+j} \BdR)\) is a \(p\)\=/adic Banach space.
		\item If \(\hat{L}\) is perfectoid, then there exists an isomorphism of \(p\)\=/adic Banach spaces
		\[
			\HH^0(L,\Fil^i \BdR/\Fil^{i+j} \BdR) \similarrightarrow \hat{L}^{\oplus j}.
		\]
		Moreover, for each integer \(n > 0\), the group \(\HH^n(L,\Fil^i \BdR/\Fil^{i+j} \BdR)\) is trivial.
		\item If \(\hat{L}\) is not perfectoid, then, for \(n \in \{0,1\}\), there exists isomorphisms of \(p\)\=/adic Banach spaces
		\[
			\HH^n(L,\Fil^i \BdR/\Fil^{i+j} \BdR) \similarrightarrow
			\begin{cases}
				\hat{L} & \text{if \(i \leq 0 < i+j\)}, \\
				0 & \text{otherwise}.
			\end{cases}
		\]
		Moreover, for each integer \(n > 1\), the group \(\HH^n(L,\Fil^i \BdR/\Fil^{i+j} \BdR)\) is trivial.
	\end{enumerate}
\end{proposition}
\begin{proof}
	Let \(n \in \N\).
	We first deduce the structure of group of \(\HH^n(L,\Fil^i \BdR/\Fil^{i+j} \BdR)\) from the cohomology of \(\Gr^i \BdR \similarrightarrow \Cp(i)\) by induction on the length
	\[
		\length_{\BdR^+} \Fil^i \BdR/\Fil^{i+j} \BdR = j.
	\]
	If \(j=1\), then \(\Fil^i \BdR/\Fil^{i+1} \BdR = \Gr^i \BdR\) which has already been treated (Table~\ref{table:galois-cohomology-cp}).
	If \(j>0\), then then there exists a strict exact sequence of \(p\)\=/adic Banach representations of \(G_K\)
	\begin{equation} \label{eq:induction-banach}
		0 \rightarrow \Fil^{i-1}\BdR/\Fil^{i+j} \BdR \rightarrow \Fil^i\BdR/\Fil^{i+j} \BdR  \rightarrow \Gr^i \BdR \rightarrow 0,
	\end{equation}
	and we conclude by induction.

	We now prove the statements concerning the topology of the group \(\HH^n(L,\Fil^i \BdR/\Fil^{i+j} \BdR)\).

	If \(\hat{L}\) is perfectoid, then \(\HH^n(L,\Fil^i \BdR/\Fil^{i+j} \BdR)\) is trivial unless \(n=0\), and \(\HH^0(L,\Fil^i \BdR/\Fil^{i+j} \BdR)\) is a \(p\)\=/adic Banach space by Corollary~\ref{corollary:h0-banach}.
	Moreover, the existence of the isomorphism between \(\HH^0(L,\Fil^i \BdR/\Fil^{i+j} \BdR)\) and \(\hat{L}^{\oplus j}\) follows by induction on the length \(\length_{\BdR^+} \Fil^i \BdR/\Fil^{i+j} \BdR = j\).
	Indeed, if \(j=1\), then we use the cohomology of \(\Gr^i \BdR \simeq \Cp(i)\) (Table~\ref{table:galois-cohomology-cp}).
	If \(j>1\), then the exact sequence~\eqref{eq:induction-banach} induces a short exact sequence of \(p\)\=/adic Banach spaces
	\[
		0 \rightarrow \HH^0(L,\Fil^{i-1}\BdR/\Fil^{i+j} \BdR) \rightarrow \HH^0(L,\Fil^i\BdR/\Fil^{i+j} \BdR)  \rightarrow \HH^0(L,\Gr^i \BdR) \rightarrow 0,
	\]
	which is thus strict and split by the open mapping theorem~\cite[Proposition~I.1.5]{Colmez1998}, and we conclude by induction.

	Assume that \(\hat{L}\) is not perfectoid.
	Let \(\Lc\) be a \(G_K\)\=/stable lattice in \(\Fil^i \BdR/\Fil^{i+j} \BdR\).
	We prove by induction on the length \(\length_{\BdR^+} \Fil^i \BdR/\Fil^{i+j} \BdR = j\) that \(\HH^n(L,\Lc)\) satisfies the hypotheses of Proposition~\ref{proposition:cohomology-banach}, that is, the subgroup \(\HH^n(L,\Lc)_\torsion\) has finite exponent.

	If \(j=1\), then \(\Lc\) is a lattice in \(\Gr^i \BdR \similarrightarrow \Cp(i)\), and thus, it is commensurable with \(\Oc_{\Cp}(i)\).
	Without loss of generality, we may assume that \(\Lc \subseteq \Oc_{\Cp}(i)\), which induces a map \(\HH^n(L,\Lc) \rightarrow \HH^n(L,\Oc_{\Cp}(i))\) whose kernel and cokernel are torsion with finite exponent.
	We conclude by Proposition~\ref{proposition:p-adic-topology-torsion} and the computation of the cohomology of \(\Oc_{\Cp}(i)\) (Theorems~\ref{theorem:ax-sen-tate},~\ref{theorem:cohomology-cp-perfectoid}, and~\ref{theorem:cohomology-cp-non-perfectoid}).
	Assume \(j>0\), the lattice \(\Lc\) yields a strict exact sequence
	\[
		0 \rightarrow \Lc^\prime \rightarrow \Lc  \rightarrow \Lc^\dprime \rightarrow 0
	\]
	of \(G_K\)\=/stable lattices in the sequence~\eqref{eq:induction-banach}, which in turn induces an exact sequence
	\begin{equation} \label{eq:induction-lattice}
		\cdots \rightarrow \HH^n(L,\Lc^\prime) \rightarrow \HH^n(L,\Lc) \rightarrow \HH^n(L,\Lc^\dprime) \rightarrow \HH^{n+1}(L,\Lc^\prime) \rightarrow \cdots.
	\end{equation}
	By induction, for each \(m \in \N\), the groups \(\HH^m(L,\Lc^\prime)_\torsion\) and \( \HH^m(L,\Lc^\dprime)_\torsion\) have finite exponent.
	Moreover, by the group structure studied in the first part of the proof, either \(\{\HH^m(L,\Lc^\prime)\}_{m\in \N}\) or \(\{\HH^m(L,\Lc^\dprime)\}_{m\in \N}\) is a set of torsion groups.
	Therefore, Proposition~\ref{proposition:p-adic-topology-torsion} applied to the exact sequence~\eqref{eq:induction-lattice} implies that \(\HH^n(L,\Lc)\) is \(p\)\=/adically complete and separated and its torsion subgroup has finite exponent.
\end{proof}

\begin{proposition} \label{proposition:fil-bdr}
	Let \(i \in \Z\).
	\begin{enumerate}
		\item For each \(n \in \N\), the space \(\HH^n(L,\Fil^i \BdR)\) is separated.
		\item If \(\hat{L}\) is perfectoid, then the group \(\HH^0(L,\Fil^i \BdR)\) is non\-/trivial, Moreover, for each integer \(n > 0\), the group \(\HH^n(L,\Fil^i \BdR)\) is trivial.
		\item If \(\hat{L}\) is not perfectoid, then, for \(n \in \{0,1\}\), there exists continuous bijective morphisms of topological \(\Qp\)\=/vector spaces
		\[
			\HH^n(L,\Fil^i \BdR) \similarrightarrow
			\begin{cases}
				\hat{L} & \text{if \(i \leq 0\)}, \\
				0 & \text{otherwise}.
			\end{cases}
		\]
		Moreover, for each integer \(n > 1\), the group \(\HH^n(L,\Fil^i \BdR)\) is trivial.
	\end{enumerate}
\end{proposition}
\begin{proof}
	Let \(n\in\N\).
	We first deduce the structure of group of \(\HH^n(L,\Fil^i \BdR)\) from Proposition~\ref{proposition:quotient-fil-bdr}.
	For each \(j \in \N\), the quotient map \(\Fil^i \BdR/\Fil^{i+j+1} \BdR \rightarrow \Fil^i \BdR/\Fil^{i+j} \BdR\) is a surjective morphism of \(p\)\=/adic Banach spaces, and thus, it admits a continuous section by the open mapping theorem~\cite[Proposition~I.1.5~iii)]{Colmez1998}.
	Therefore, by Proposition~\ref{proposition:lim-cohomology}, the isomorphism of topological \(G_K\)\=/modules \(\Fil^i \BdR \similarrightarrow \varprojlim_j \Fil^i \BdR/\Fil^{i+j} \BdR\) induces short exact sequences of groups
	\begin{equation} \label{eq:lim-bdr}
	\begin{tikzcd}
		0 \ar{r} & \varprojlim_j^1 \HH^{n-1}(L,\Fil^i \BdR/\Fil^{i+j} \BdR) \ar[phantom, ""{coordinate, name=Z}]{d} \ar[rounded corners, to path={ -- ([xshift=2em]\tikztostart.east) |- (Z) [near end]\tikztonodes -| ([xshift=-2em]\tikztotarget.west) -- (\tikztotarget)}]{dl} & \\
		\HH^n(L,\Fil^i \BdR ) \ar{r} & \varprojlim_j \HH^n(L,\Fil^i \BdR/\Fil^{i+j} \BdR) \ar{r} & 0.
	\end{tikzcd}
	\end{equation}
	The combination of the exact sequence~\eqref{eq:lim-bdr} and Proposition~\ref{proposition:quotient-fil-bdr} yields the structure of group of \(\HH^n(L,\Fil^i \BdR)\).

	About the topology, note that the group \(\HH^n(L,\Fil^i \BdR)\) is trivial unless \(i \leq 0\) and \(n \in \{0,1\}\), in which case there is a continuous bijective morphisms of topological \(\Qp\)\=/vector spaces
	\[
		\HH^n(L,\Fil^i \BdR) \similarrightarrow \HH^n(L,\Fil^i \BdR/ \Fil^1 \BdR).
	\]
	By Proposition~\ref{proposition:quotient-fil-bdr}, the topological group \(\HH^n(\Fil^i \BdR/ \Fil^1 \BdR)\) is separated, and thus, the topological group \(\HH^n(L,\Fil^i \BdR)\) is separated.
\end{proof}

\begin{proposition} \label{proposition:bdr}
	The following properties hold.
	\begin{enumerate}
		\item If \(\hat{L}\) is perfectoid, then the group \(\HH^0(L,\BdR)\) is non\-/trivial.
		Moreover, for each integer \(n > 0\), the group \(\HH^n(L, \BdR)\) is trivial.
		\item If \(\hat{L}\) is not perfectoid, then, for \(n \in \{0,1\}\), there exists isomorphisms of \(\Qp\)\=/vector spaces
		\[
			\HH^n(L,\BdR) \similarrightarrow \hat{L}.
		\]
		Moreover, for each integer \(n > 1\), the group \(\HH^n(L,\BdR)\) is trivial.
	\end{enumerate}
\end{proposition}
\begin{proof}
	The isomorphism of topological \(G_K\)\=/modules \(\BdR \similarrightarrow \varinjlim_i \Fil^i \BdR\) induces an isomorphism of groups for each \(n \in \N\)
	\begin{equation} \label{eq:cohomology-bdr-limit}
		\varinjlim_i \Cr^n(G_L,\Fil^i \BdR) \similarrightarrow \Cr^n(G_L,\BdR).
	\end{equation}
	Indeed, the injectivity of the map~\eqref{eq:cohomology-bdr-limit} is obvious, and if \(f \in \Cr^n(G_L,\BdR)\) then, since \(G_L\) is compact, the image of \(f\) in \(\BdR\) is compact and thus bounded, which implies that the image of \(f\) is contained in some \(\Fil^i \BdR\) (see~\cite[Proposition~5.6]{Schneider2002}), and therefore, the element \(f\) is in the image of the map~\eqref{eq:cohomology-bdr-limit}.

	Since filtered colimits of abelian groups are exact, we note, considering the definition of the continuous cohomology groups, that the maps~\eqref{eq:cohomology-bdr-limit} induce an isomorphism of groups for each \(n \in \N\)
	\[
		\varinjlim_i \HH^n(L,\Fil^i \BdR) \similarrightarrow \HH^n(L,\BdR).
	\]
	The statement then follows from Proposition~\ref{proposition:fil-bdr}.
\end{proof}

\begin{corollary} \label{corollary:bdr-fil-bdr}
	Let \(i \in \Z\).
	If \(\hat{L}\) is not perfectoid, then there exists an isomorphism of \(p\)\=/adic Banach spaces
	\[
		\HH^0(L,\BdR/\Fil^i\BdR) \similarrightarrow
		\begin{cases}
			\hat{L} & \text{if \(i > 0\)},\\
			0 & \text{otherwise}.
		\end{cases}
	\]
\end{corollary}
\begin{proof}
	If \(i \leq 0\), then the strict exact sequence of topological \(G_K\)\=/modules
	\[
		0 \rightarrow \Fil^i \BdR \rightarrow \BdR \rightarrow \BdR/\Fil^i \BdR \rightarrow 0
	\]
	induces an exact sequence
	\[
		0 \rightarrow \HH^0(L,\Fil^i \BdR) \rightarrow \HH^0(L,\BdR) \rightarrow \HH^0(L,\BdR/\Fil^i \BdR) \rightarrow \HH^1(L,\BdR),
	\]
	which, by Proposition~\ref{proposition:fil-bdr} and Proposition~\ref{proposition:bdr}, implies that \(\HH^0(L,\BdR/\Fil^i\BdR)\) is trivial.

	If \(i > 0\), then the strict exact sequence of topological \(G_K\)\=/modules
	\[
		0 \rightarrow \BdR^+/\Fil^i \BdR \rightarrow \BdR/\Fil^i \BdR \rightarrow \BdR/\BdR^+ \rightarrow 0
	\]
	induces a continuous bijective map of topological groups
	\begin{equation} \label{eq:h-0-fil}
		\HH^0(L,\BdR^+/\Fil^i \BdR) \similarrightarrow \HH^0(L,\BdR/\Fil^i \BdR),
	\end{equation}
	since \(\HH^0(L,\BdR/\BdR^+)\) is trivial by the first part of the proof.
	Moreover, the topology of \(\HH^0(L,\BdR^+/\Fil^i \BdR)\) (respectively \(\HH^0(L,\BdR/\Fil^i \BdR)\)) is the subspace topology from \(\BdR^+/\Fil^i \BdR\) (respectively \(\BdR/\Fil^i \BdR\)), and since the topology of \(\BdR^+/\Fil^i \BdR\) is the subspace topology from \(\BdR/\Fil^i \BdR\), we conclude that the topology of \(\HH^0(L,\BdR^+/\Fil^i \BdR)\) is the subspace topology from \(\HH^0(L,\BdR/\Fil^i \BdR)\), and thus, the map~\eqref{eq:h-0-fil} is an isomorphism of topogical groups, which, by Proposition~\ref{proposition:quotient-fil-bdr}, are \(p\)\=/adic Banach spaces.
\end{proof}

The combination of Corollary~\ref{corollary:gal-cp} and Proposition~\ref{proposition:quotient-fil-bdr}, Proposition~\ref{proposition:fil-bdr}, Proposition~\ref{proposition:bdr} and Corollary~\ref{corollary:bdr-fil-bdr} yields the following.

\begin{corollary} \label{corollary:gal-bdr-restriction}
	Let
	\[
		M \in \{\BdR,\Fil^i \BdR,\BdR/\Fil^i \BdR,\Fil^i \BdR/\Fil^{i+j} \BdR\}_{i \in \Z, j \in \N}.
	\]
	Let \(E\) be an extension of \(\Qp\) contained in \(L\).
	If \(\hat{L}\) is not perfectoid, then the restriction map
	\[
		\HH^1(E,M) \rightarrow \HH^1(L,M)
	\]
	is injective.
\end{corollary}

Concerning the rings of crystalline and semistable periods, Fontaine~\cite{Fontaine1994:III} has proved that
\[
	\HH^0(K,\Bcris) = \HH^0(K,\Bst) = K_0.
\]

\subsection{\texorpdfstring{\(p\)\=/}{p-}adic Galois representations}
\label{subsec:p-adic-galois-representations}

We recall materials from the theory \(p\)\=/adic Galois representations developped by Fontaine~\cite{Fontaine1994:III}.
Let \(V\) be a \(p\)\=/adic representation of \(G_K\).

With \(V\) is associated the finite dimensional filtered \(K\)\=/vector space
\[
	\DdR(V) = (\BdR \otimes_{\Qp} V)^{G_K}
\]
equipped with the filtration by \(K\)\=/vector subspaces
\[
	\Fil^i \DdR(V) = (\Fil^i \BdR \otimes_{\Qp} V)^{G_K}.
\]
The representation \(V\) is \emph{de Rham} if \(\dim_K \DdR(V) = \dim_{\Qp} V\).
If \(V\) is de Rham, then the \emph{Hodge\--Tate weights} of \(V\) are the integer \(i\) such that the space \(\Fil^{-i} \DdR(V)/\Fil^{-i+1} \DdR(V)\) is non\-/trivial, and the multiplicity \(m_i(V)\) of \(i\) as a Hodge\--Tate weight of \(V\) is defined by
\[
	m_i(V) = \dim_K \Fil^{-i}\DdR(V)/\Fil^{-i+1} \DdR(V).
\]
If \(V\) is de Rham, then the inclusion \(\DdR(V) \subset \BdR \otimes_{\Qp} V\) induces an isomorphism of topological \(G_K\)\=/modules
\begin{equation} \label{eq:dr}
	\BdR \otimes_K \DdR(V) \similarrightarrow \BdR \otimes_{\Qp} V,
\end{equation}
and, for each \(i \in \Z\), an isomorphism
\begin{equation} \label{eq:fil-dr}
	\Fil^i \left(\BdR \otimes_K \DdR(V)\right) \similarrightarrow \Fil^i \BdR \otimes_{\Qp} V,
\end{equation}
where
\[
	\Fil^i\left(\BdR \otimes_K \DdR(V)\right) = \sum_{a + b = i} \Fil^a \BdR \otimes_K \Fil^b \DdR(V).
\]

With \(V\) is associated the finite dimensional filtered \((\varphi,N)\)\=/module over \(K\)
\[
	\Dst(V) = (\Bst \otimes_{\Qp} V)^{G_K},
\]
where the \(\sigma\)\=/semilinear map \(\varphi\) and the \(K_0\)\=/linear map \(N\) on \(\Dst(V)\) are respectively induced by the map \(\varphi\) and \(N\) on \(\Bst\), and the filtration by \(K\)\=/vector subspaces on the \(K\)\=/vector space \(K \otimes_{K_0}\Dst(V)\) is induced by the filtration on \(\DdR(V)\) via the map
\begin{equation} \label{eq:dcris-ddr}
	K \otimes_{K_0}\Dst(V) = (K \otimes_{K_0} \Bst \otimes_{\Qp} V)^{G_K} \rightarrow \DdR(V),
\end{equation}
defined by the inclusions~\eqref{eq:Bcris_BdR}.
Moreover, with \(V\) are associated the filtered \(\varphi\)\=/module over \(K\)
\[
	\Dcris(V) = \Dst(V)^{N=0} = (\Bcris \otimes_{\Qp} V)^{G_K},
\]
and the finite dimensional \(\Qp\)\=/vector space
\[
	\De(V) = \Dcris(V)^{\varphi = 1} = (\Be \otimes_{\Qp} V)^{G_K}.
\]

The representation \(V\) is \emph{semistable} (respectively \emph{crystalline}) if \(\dim_{K_0} \Dst(V) = \dim_{\Qp} V\) (respectively if \(\dim_{K_0} \Dcris(V) = \dim_{\Qp} V\)).
If \(V\) is semistable, then \(V\) is de Rham, and the map~\eqref{eq:dcris-ddr} defines an isomorphism of filtered \(K\)\=/vector spaces
\[
	K \otimes_{K_0} \Dcris(V) \similarrightarrow \DdR(V).
\]
If \(V\) is crystalline, then \(V\) is semistable, and \(\Dcris(V) = \Dst(V)\).

Let \(V^\ast(1) = \Hom_{\Qp}(V,\Qp(1))\) be the \emph{Tate dual representation} of \(V\) on which \(G_K\) acts by
\[
	g \cdot f(v) = g(f(g^{-1}v)) = \chi(g) \cdot f(g^{-1}v),
\]
for all \(g \in G_K\), \(f \in V^\ast(1)\), and \(v \in V\).
If \(V\) is crystalline (respectively semistable, de Rham), then \(V^\ast(1)\) is crystalline (respectively semistable, de Rham).

If \(K^\prime\) is a finite extension of \(K\), then we denote by \(V|_{K^\prime}\) the restiction of \(V\) to \(G_{K^\prime}\), and, for each \(\ast \in \{\e,\cris,\st,\dR\}\), we set \(\DDb_{\ast,K^\prime}(V) = {(\B_\ast \otimes_{\Qp} V)^{G_{K^\prime}}}\).
If \(V\) is crystalline (respectively semistable, de Rham), then the representation \(V|_{K^\prime}\) is crystalline and there is an isomorphism of filtered \(\varphi\)\=/modules \(K^\prime_0 \otimes_{K_0} \Dcris(V) \similarrightarrow \DDb_{\cris,K^\prime}(V)\) (respectively the representation \(V|_{K^\prime}\) is semistable and there is an isomorphism of filtered \((\varphi,N)\)\=/modules \(K^\prime_0 \otimes_{K_0} \Dst(V) \similarrightarrow \DDb_{\st,K^\prime}(V)\), the representation \(V|_{K^\prime}\) is de Rham and there is an isomorphism of filtered vector spaces \(K^\prime \otimes_{K} \DdR(V) \similarrightarrow \DDb_{\dR,K^\prime}(V)\)).

We will need the following~\cite[Proposition~6.3.1]{Ponsinet2024}.

\begin{proposition} \label{proposition:bounded-dim-de}
	There exists \(c \in \N\) such that, for each finite extension  \(K^\prime\) of \(K\), it holds \(\dim_{\Qp} \DDb_{\e,K^\prime}(V) \leq c\).
\end{proposition}

\begin{remark} \label{remark:p-adic-rep-results}
	Recall the following central results of the theory of \(p\)\=/adic Galois representations.
	\begin{enumerate}
		\item The \enquote{weakly admissible implies admissible} Theorem~\cite{ColmezFontaine2000,FarguesFontaine2018} states that the functor \(\Dst\) induces an equivalence between the category of semistable representations of \(G_K\) and the category of weakly admissible filtered \((\varphi,N)\)\=/modules over \(K\).
		\item The \enquote{\(p\)\=/adic monodromy} Theorem~\cite{Berger2002,FarguesFontaine2018} states that if \(V\) is de Rham, then there exists a finite extension \(K^\prime\) of \(K\) such that \(V|_{K^\prime}\) is semistable.
	\end{enumerate}
\end{remark}

It is convenient to introduce the following notation.
\begin{notation} \label{notation:tensor}
	We set the topological \(G_K\)\=/modules
	\[
		\begin{split}
			V_\dR & = \BdR \otimes_{\Qp} V, \\
			V_\dR^i & = \Fil^i \BdR \otimes_{\Qp} V, \text{ for each \(i \in \Z\)}, \\
			V_\dR^+ & = \BdR^+ \otimes_{\Qp} V, \\
			V_\st & = \Bst \otimes_{\Qp} V, \\
			V_\cris & = \Bcris \otimes_{\Qp} V, \\
			V_\e & = \Be \otimes_{\Qp} V.
		\end{split}
	\]
\end{notation}

The isomorphisms~\eqref{eq:dr} and~\eqref{eq:fil-dr} together with the results from the previous subsection~\ref{subsec:galois-cohomology-of-p-adic-period-rings} yield the following.

\begin{proposition} \label{proposition:rep-dr}
	Let
	\[
		M \in \{V_\dR,V_\dR^i,V_\dR/V_\dR^i,V_\dR^i/V_\dR^{i+j}\}_{i \in \Z, j \in \N \setminus \{0\}}.
	\]
	Let \(L\) be an algebraic extension of \(K\).
	If \(V\) is de Rham, then the following properties hold.
	\begin{enumerate}
		\item For each \(n \in \N\), the topological group \(\HH^n(L,V_\dR^i/V_\dR^{i+j})\) is a \(p\)\=/adic Banach space.
		\item For each \(n \in \N\), the topological group \(\HH^n(L,V_\dR^i)\) is separated.
		\item If \(\hat{L}\) is perfectoid, then the groups \(\HH^0(L,M)\) is non\-/trivial, and, for each integer \(n > 0\), the group \(\HH^n(L,M)\) is trivial.
		\item If \(\hat{L}\) is not perfectoid, then, for \(n \in \{0,1\}\), there exist isomorphisms of \(p\)\=/adic Banach spaces
		\[
			\HH^n(L,V_\dR^i/V_\dR^{i+j}) \similarrightarrow \hat{L} \otimes_K (\Fil^i \DdR(V)/ \Fil^{i+j} \DdR(V)),
		\]
		and
		\[
			\HH^0(L,V_\dR/V_\dR^i) \similarrightarrow \hat{L} \otimes_K (\DdR(V)/\Fil^i \DdR(V)),
		\]
		continuous bijective morphisms of topological \(\Qp\)\=/vector spaces
		\[
			\HH^n(L,V_\dR^i) \similarrightarrow \hat{L} \otimes_K \Fil^i \DdR(V),
		\]
		and isomorphisms of \(\Qp\)\=/vector spaces
		\[
			\HH^n(L,V_\dR) \similarrightarrow \hat{L} \otimes_K  \DdR(V),
		\]
		and
		\[
			\HH^1(L,V_\dR/V_\dR^i) \similarrightarrow \hat{L} \otimes_K (\DdR(V)/\Fil^i \DdR(V)).
		\]
		Moreover, for each integer \(n > 1\), the group \(\HH^n(L,M)\) is trivial.
		\item Let \(E\) be an extension of \(\Qp\) contained in \(L\).
		If \(\hat{L}\) is not perfectoid, then the restriction map
		\[
			\HH^1(E,M) \rightarrow \HH^1(L,M)
		\]
		is injective.
	\end{enumerate}
\end{proposition}

The \emph{tangent space} \(t_V\) of \(V\) is the finite dimensional \(K\)\=/vector space defined by
\[
	t_V = \HH^0(K,V_\dR/V_\dR^+).
\]
For each \(K\)\=/vector space \(Q\), we set \(t_V(Q) = Q \otimes_K t_V\).
If \(V\) is de Rham, then, by Proposition~\ref{proposition:rep-dr}, there is an isomorphism
\[
	\DdR(V)/\Fil^0 \DdR(V) \similarrightarrow t_V.
\]

The fundamental exact sequence~\eqref{eq:fundamental} induces a strict exact sequence of topological \(\Qp\)\=/vector spaces equipped with a continuous and \(\Qp\)\=/linear action by \(G_K\)
\begin{equation} \label{eq:fundamental-v}
	0 \rightarrow V \rightarrow V_\e \rightarrow V_\dR/V_\dR^+ \rightarrow 0.
\end{equation}

\section{Bloch\texorpdfstring{\--}{--}Kato groups and perfectoid fields}
\label{sec:bloch-kato-groups-and-perfectoid-fields}

\subsection{Bloch\texorpdfstring{\--}{--}Kato groups}
\label{subsec:bloch-kato-groups}

We recall the definition of the Bloch\--Kato groups~\cite[\S~3]{BlochKato1990}.

Let \(V\) be a \(p\)\=/adic representation of \(G_K\).
Let \(L\) be an algebraic extension of \(K\).
The \emph{exponential}, \emph{finite}, and \emph{geometric} \emph{Bloch\--Kato subgroups} of \(\HH^1(L,V)\) are respectively defined by
\[
	\begin{split}
		\HH^1_e(L,V) & = \Ker\left( \HH^1(L,V) \rightarrow \HH^1(L,V_\e) \right), \\
		\HH^1_f(L,V) & = \Ker\left( \HH^1(L,V) \rightarrow \HH^1(L,V_\cris) \right), \\
		\HH^1_g(L,V) & = \Ker\left( \HH^1(L,V) \rightarrow \HH^1(L,V_\dR) \right),
	\end{split}
\]
where the maps are induced by the natural inclusion of \(\Qp\) in the \(p\)\=/adic period rings.
The inclusions \(\Be \subset \Bcris \subset \BdR\) induce inclusions
\begin{equation} \label{eq:bk-inclusions}
	\HH^1_e(L,V) \subseteq \HH^1_f(L,V) \subseteq \HH^1_g(L,V) \subseteq \HH^1(L,V).
\end{equation}

\begin{lemma} \label{lemma:dr-bk-g}
	If \(V\) is de Rham, then
	\[
		\HH^1_g(L,V) = \Ker\left( \HH^1(L,V) \rightarrow \HH^1(L,V_\dR^+) \right),
	\]
	where the map on the right-hand side is induced by the natural inclusion \(\Qp \subseteq \BdR^+\).
\end{lemma}
\begin{proof}
	By Proposition~\ref{proposition:rep-dr}, the map \(\HH^1(L,V_\dR^+) \rightarrow \HH^1(L,V_\dR)\) is injective, which implies the statement.
\end{proof}

\begin{proposition} \label{proposition:bkg-closed}
	If \(V\) is de Rham, then the geometric Bloch\--Kato group \(\HH^1_g(L,V)\) is a \(p\)\=/adic Banach subspace of \(\HH^1(L,V)\).
\end{proposition}
\begin{proof}
	The natural inclusion \(\Qp \subseteq \BdR^+\) induces the morphism of topological \(\Qp\)\=/vector spaces \(\HH^1(L,V) \rightarrow \HH^1(L,V_\dR^+)\), whose kernel is \(\HH^1_g(L,V)\) by Lemma~\ref{lemma:dr-bk-g}.
	By Corollary~\ref{corollary:h1-representation-banach}, the space \(\HH^1(L,V)\) is a \(p\)\=/adic Banach space.
	Moreover, the space \(\HH^1(L,V_\dR^+) \) is separated by Proposition~\ref{proposition:rep-dr}.
	Therefore, the subspace \(\HH^1_g(L,V)\) is closed in \(\HH^1(L,V)\), and thus, it is a \(p\)\=/adic Banach subspace of \(\HH^1(L,V)\).
\end{proof}

\subsection{Comparison of cohomology of \texorpdfstring{\(p\)\=/}{p-}adic period rings}
\label{subsec:comparison-of-cohomology-of-p-adic-period-rings}

Let \(V\) be a \(p\)\=/adic representation of \(G_K\).
Let \(L\) be an algebraic extension of \(K\).
The inclusion of \(\Be\) in \(\BdR\) induces a morphism of topological \(\Qp\)\=/vector spaces
\[
	f_L : \HH^1(L,V_\e) \rightarrow \HH^1(L,V_\dR),
\]
which we now study.
We recall computations by Hyodo~\cite{Hyodo1991}.

\begin{proposition}[Hyodo] \label{proposition:hyodo}
	If \(V\) is de Rham, then there exists a commutative diagram of \(\Qp\)\=/vector spaces whose rows and columns are exact
	\[
		\begin{tikzcd}
			& & & 0 \ar{d} \\
			& & & \HH^1(L,V_\dR^+) \ar{d} \\
			0 \ar{r} & \HH^1_g(L,V)/\HH^1_e(L,V) \ar{r} \ar{d} & \HH^1(L,V_\e) \ar{r}{f_L} \ar[equal]{d} & \HH^1(L,V_\dR) \ar{d} \\
			0 \ar{r} & \HH^1(L,V)/\HH^1_e(L,V) \ar{r} & \HH^1(L,V_\e) \ar{r} & \HH^1(L,V_\dR/V_\dR^+) \ar{d} \\
			& & & \phantom{.}0.
		\end{tikzcd}
	\]
\end{proposition}
\begin{proof}
	By Proposition~\ref{proposition:rep-dr}, there exists a commutative diagram whose rows are exact
	\begin{equation} \label{eq:diag-hyodo}
		\begin{tikzcd}
			& & \HH^1(L,V_\e) \ar[equal]{r} \ar{d}{f_L} & \HH^1(L,V_\e) \ar{d} & \\
			0 \ar{r} & \HH^1(L,V_\dR^+) \ar{r} & \HH^1(L,V_\dR) \ar{r} & \HH^1(L,V_\dR/V_\dR^+) \ar{r} & 0.
		\end{tikzcd}
	\end{equation}
	By definition of the exponential Bloch\--Kato group, the cohomology of the fundamental exact sequence~\eqref{eq:fundamental-v} defines an isomorphism
	\begin{equation} \label{eq:fundamental-v-bk-e}
		\HH^1(L,V)/\HH^1_e(L,V) \similarrightarrow \Ker\left(\HH^1(L,V_\e) \rightarrow \HH^1(L,V_\dR/V_\dR^+)\right).
	\end{equation}
	The snake lemma applied to the diagram~\eqref{eq:diag-hyodo} combined with the isomorphism~\eqref{eq:fundamental-v-bk-e} yields
	\[
		\Ker(f_L) \similarrightarrow \Ker\left(\HH^1(L,V)/\HH^1_e(L,V) \rightarrow \HH^1(L,V_\dR^+)\right),
	\]
	and we conclude by Lemma~\ref{lemma:dr-bk-g}.
\end{proof}

We now study specifically the case where \(L=K^\prime\) is a finite extension of \(K\).
We first recall properties of the Bloch\--Kato groups over \(K^\prime\) established by Bloch and Kato~\cite[Proposition~3.8 and Corollary~3.8.4]{BlochKato1990}.

\begin{proposition}[Bloch\--Kato] \label{proposition:bk-dcris}
	If \(V\) is de Rham, then it holds
	\[
		\dim_{\Qp} \HH^1_f(K^\prime,V)/\HH^1_e(K^\prime,V) = \dim_{\Qp} \DDb_{\e,K^\prime}(V).
	\]
\end{proposition}

Recall~\cite[II \S~5.2 Théorème~2]{Serre1994} that local Tate duality defines a perfect pairing for each \(i \in \{0,1,2\}\)
\begin{equation} \label{eq:local-tate-duality}
	\HH^i(K^\prime,V) \times \HH^{2-i}(K^\prime,V^\ast(1)) \rightarrow \HH^2(K^\prime,\Qp(1)) \similarrightarrow \Qp.
\end{equation}

\begin{proposition}[Bloch\--Kato] \label{proposition:bk-dual}
	If \(V\) is de Rham, then under local Tate duality the orthogonal complement of \(\HH^1_e(K^\prime,V)\) (respectively \(\HH^1_f(K^\prime,V)\)) is \(\HH^1_g(K^\prime,V^\ast(1))\) (respectively \(\HH^1_f(K^\prime,V^\ast(1))\)).
\end{proposition}

\begin{corollary}[Hyodo] \label{corollary:hyodo}
	If \(V\) is de Rham, then it holds
	\[
		\dim_{\Qp} \Ker(f_{K^\prime}) = \dim_{\Qp} \DDb_{\e,K^\prime}(V) + \dim_{\Qp} \DDb_{\e,K^\prime}(V^\ast(1)).
	\]
\end{corollary}
\begin{proof}
	The combination of Propositions~\ref{proposition:hyodo}, \ref{proposition:bk-dcris} and~\ref{proposition:bk-dual} yields
	\[
		\begin{split}
			& \dim_{\Qp} \Ker(f_{K^\prime}) \\
			= & \dim_{\Qp} \HH^1_g(K^\prime,V)/\HH^1_e(K^\prime,V) \\
			= & \dim_{\Qp} \HH^1_f(K^\prime,V)/\HH^1_e(K^\prime,V) + \dim_{\Qp} \HH^1_g(K^\prime,V)/\HH^1_f(K^\prime,V) \\
			= & \dim_{\Qp} \HH^1_f(K^\prime,V)/\HH^1_e(K^\prime,V) + \dim_{\Qp} \HH^1_f(K^\prime,V^\ast(1))/\HH^1_e(K^\prime,V^\ast(1)) \\
			= & \dim_{\Qp} \DDb_{\e,K^\prime}(V) + \dim_{\Qp} \DDb_{\e,K^\prime}(V^\ast(1)).
		\end{split}
	\]
\end{proof}

\begin{corollary} \label{corollary:bk-exp-full}
	If \(V\) is de Rham such that the Hodge\--Tate weights of \(V\) are all \(> 0\), and the spaces \(\DDb_{\e,K^\prime}(V)\) and \(\DDb_{\e,K^\prime}(V^\ast(1))\) are trivial, then
	\[
		\HH^1_e(K^\prime,V) = \HH^1(K^\prime,V).
	\]
\end{corollary}
\begin{proof}
	By Proposition~\ref{proposition:rep-dr}, it holds
	\[
		\dim_{\Qp} \HH^1(K^\prime,V_\dR^+) = \dim_{\Qp} \Fil^0 \DDb_{\dR,K^\prime}(V) =0.
	\]
	Therefore, Lemma~\ref{lemma:dr-bk-g} implies that \(\HH^1_g(K^\prime,V) = \HH^1(K^\prime,V)\).
	Finally, Proposition~\ref{proposition:hyodo} and Corollary~\ref{corollary:hyodo} implies that \(\HH^1_e(K^\prime,V) = \HH^1_g(K^\prime,V)\).
\end{proof}

\begin{proposition} \label{proposition:cohomology-be}
	If \(V\) is de Rham, then the following properties hold.
	\begin{enumerate}
		\item For each \(n \in \N\), the \(\Qp\)\=/vector space \(\HH^n(K^\prime,V_\e)\) is finite dimensional.
		\item For each integer \(n > 2\), the space \(\HH^n(K^\prime,V_\e)\) is trivial.
		\item It holds
		\[
			\sum_{n=0}^2 (-1)^n \dim_{\Qp} \HH^n(K^\prime,V_\e) = -[K^\prime:\Qp]\dim_{\Qp} V.
		\]
	\end{enumerate}
\end{proposition}
\begin{proof}
	The fundamental exact sequence~\eqref{eq:fundamental-v} induces a long exact sequence of \(\Qp\)\=/vector spaces
	\begin{equation} \label{eq:les}
		\cdots \rightarrow \HH^n(K^\prime,V) \rightarrow \HH^n(K^\prime,V_\e) \rightarrow \HH^n(K^\prime,V_\dR/V_\dR^+) \rightarrow \cdots.
	\end{equation}
	On the one hand, by Tate~\cite[II \S~5.2 Proposition~14, \S~5.3 Proposition~15, and \S~5.7 Théorème~5]{Serre1994}, the \(\Qp\)\=/vector spaces \(\HH^n(K^\prime,V)\) are finite dimensional, moreover, they are trivial if \(n > 2\), and they satisfy the local Euler\--Poincaré characteristic
	\begin{equation} \label{eq:euler-poincare-characteristic}
		\sum_{n=0}^2 (-1)^n \dim_{\Qp} \HH^n(K^\prime,V) = -[K^\prime:\Qp] \dim_{\Qp} V.
	\end{equation}
	On the other hand, by Proposition~\ref{proposition:rep-dr} (also due to Tate in the case of finite extensions of \(\Qp\)), the groups \(\HH^n(K^\prime,V_\dR/V_\dR^+)\) are finite dimensional \(\Qp\)\=/vector spaces, moreover, they are trivial if \(n \geq 2\), and they satisfy
	\[
		\sum_{n=0}^1 (-1)^n \dim_{\Qp} \HH^n(K^\prime,V_\dR/V_\dR^+) = 0.
	\]
	Finally, the combination of the long exact sequence~\eqref{eq:les} and the properties of the vector spaces \(\HH^n(K^\prime,V)\) and \(\HH^n(K^\prime,V_\dR/V_\dR^+)\) implies the statement about the vector spaces \(\HH^n(K^\prime,V_\e)\).
\end{proof}

\begin{corollary} \label{corollary:isomorphism-fk}
	If the representation \(V\) is de Rham and the spaces \(\DDb_{\e,K^\prime}(V)\) and \(\DDb_{\e,K^\prime}(V^\ast(1))\) are trivial, then the map \(f_{K^\prime}\) defines an isomorphism
		\[
			\HH^1(K^\prime,V_\e) \similarrightarrow \HH^1(K^\prime,V_\dR).
		\]
\end{corollary}
\begin{proof}
	By Corollary~\ref{corollary:hyodo}, the map \(f_{K^\prime}\) is injective.
	Moreover, by Proposition~\ref{proposition:rep-dr} and Proposition~\ref{proposition:cohomology-be}, it holds
	\[
		\begin{split}
			\dim_{\Qp} \HH^1(K^\prime,V_\dR) & = \dim_{\Qp} \DDb_{\dR,K^\prime}(V) \\
			& = [K^\prime:\Qp] \dim_{\Qp} V \\
			& \leq \dim_{\Qp} \HH^1(K^\prime,V_\e).
		\end{split}
	\]
	Therefore, the morphism \(f_{K^\prime}\) is an isomorphism.
\end{proof}

\begin{corollary} \label{corollary:dim-bk-g}
	If \(V\) is a non\-/trivial de Rham whose Hodge\--Tate weights are not all \(\leq 0\), then
	\[
		\dim_{\Qp} \HH^1_g(K^\prime,V) > 0.
	\]
\end{corollary}
\begin{proof}
	On the one hand, by the Euler\--Poincaré characteristic~\eqref{eq:euler-poincare-characteristic}, it holds
	\[
		\dim_{\Qp} \HH^1(K^\prime,V) \geq [K^\prime:\Qp] \dim_{\Qp} V.
	\]
	On the other hand, by Proposition~\ref{proposition:rep-dr}, it holds
	\[
		\begin{split}
			\dim_{\Qp} \HH^1(K^\prime,V_\dR^+) & = \dim_{\Qp} \Fil^0 \DDb_{\dR,K^\prime}(V) \\
			& < \dim_{\Qp} \DDb_{\dR,K^\prime}(V) = [K^\prime:\Qp] \dim_{\Qp} V.
		\end{split}
	\]
	We conclude by Lemma~\ref{lemma:dr-bk-g}.
\end{proof}

\subsection{Universal norms}
\label{subsec:universal-norms}

We recall the definition of the modules of universal norms.

Let \(V\) be a \(p\)\=/adic representation of \(G_K\), and let \(T\) be a \(G_K\)\=/stable lattice in \(V\).
The strict exact sequence of topological \(G_K\)\=/modules
\[
	0 \rightarrow T \rightarrow V \rightarrow V/T \rightarrow 0
\]
induces an exact sequence for each algebraic extension \(E\) of \(K\)
\[
	\HH^1(E,T) \xrightarrow{\alpha_{E}} \HH^1(E,V) \xrightarrow{\beta_{E}} \HH^1(E,V/T).
\]

For each finite extension \(K^\prime\)  of \(K\), and for \(\ast \in \{e,f,g\}\), the Bloch\--Kato subgroups of \(\HH^1(K^\prime,T)\) and \(\HH^1(K^\prime,V/T)\) are respectively defined by
\[
	\begin{split}
		\HH^1_\ast(K^\prime,T) & = \alpha_{K^\prime}^{-1}(\HH^1_\ast(K^\prime,V)) \\
		\HH^1_\ast(K^\prime,V/T) & = \beta_{K^\prime}(\HH^1_\ast(K^\prime,V)). \\
	\end{split}
\]

\begin{remark} \label{remark:rank-bk-g}
	By Corollary~\ref{corollary:dim-bk-g} and since \(\alpha_{K^\prime}(\HH^1(K^\prime,T))\) spans \(\HH^1(K^\prime,V)\), if \(V\) is a non\-/trivial de Rham representation whose Hodge\--Tate weights are not all \(\leq 0\), then  the \(\Zp\)\=/rank of \(\HH^1_g(K^\prime,T)\) is \(>0\).
\end{remark}

Let \(L\) be an algebraic extension of \(K\).

For \(n \in \N\), the \(n\)\=/th \emph{Iwasawa cohomology} group is defined by
\[
	\HH^n_\Iw(K,L,T) = \varprojlim \HH^n(K^\prime,T),
\]
where \(K^\prime\) runs through the finite extensions of \(K\) contained in \(L\), and the transition morphisms are the corestriction maps.
Note that if \(L\) is a finite extension of \(K\), then \(\HH^n_\Iw(K,L,T) = \HH^n(L,T)\).

For each \(\ast \in \{e,f,g\}\), the Bloch\--Kato groups \(\HH^1_\ast(K^\prime,T)\) are compatible under the corestriction maps.
The module of \emph{\(\ast\)\=/universal norms} associated with \(T\) in the extension \(L/K\) is defined by
\[
		\HH^1_{\Iw,\ast}(K,L,T) = \varprojlim \HH^1_\ast(K^\prime, T),
\]
where \(K^\prime\) runs over all the finite extensions of \(K\) contained in \(L\), and the transition morphisms are the corestriction maps.

Recall that there exists a natural isomorphism
\[
	\HH^1(L,V/T) \similarrightarrow \varinjlim \HH^1(K^\prime,V/T),
\]
where \(K^\prime\) runs over all the finite extensions of \(K\) contained in \(L\), and the transition morphisms are the restriction maps~\cite[I \S~2.2 Proposition~8]{Serre1994}.
For each \(\ast \in \{e,f,g\}\), the groups \(\HH^1_\ast(K^\prime,V/T)\) are compatible under the restriction maps, and the Bloch\--Kato subgroups of \(\HH^1(L,V/T)\) are then defined by
\[
		\HH^1_\ast(L,V/T) = \varinjlim \HH^1_\ast(K^\prime, V/T),
\]
where \(K^\prime\) runs over all the finite extensions of \(K\) contained in \(L\), and the transition morphisms are the restriction maps.

\begin{remark}
	Let \(\ast \in \{e,f,g\}\).
	The group \(\HH^1_\ast(L,V/T)\) is not defined as the image of \(\HH^1_\ast(L,V)\) in \(\HH^1(L,V/T)\) by \(\beta_L\).
\end{remark}

Let \(V^\ast(1) = \Hom_{\Qp}(V,\Qp(1))\) be the Tate dual representation of \(V\), and let \(T^\ast(1) = \Hom_{\Zp}(T,\Zp(1))\) be the Tate dual of \(T\), which is a \(G_K\)\=/stable lattice in \(V^\ast(1)\).
Local Tate duality~\eqref{eq:local-tate-duality} yields a perfect pairing for each \(i \in \{0,1,2\}\)
\begin{equation} \label{eq:local-tate-duality-iw}
	\HH^i(L,V^\ast(1)/T^\ast(1)) \times \HH^{2-i}_{\Iw}(K,L,T) \rightarrow \Qp/\Zp.
\end{equation}
The duality properties of the Bloch\--Kato groups from Proposition~\ref{proposition:bk-dual} implies the following.

\begin{proposition}[Bloch\--Kato] \label{proposition:bk-dual-iw}
	If \(V\) is de Rham, then, under local Tate duality,
	\begin{enumerate}
		\item the orthogonal complement of \(\HH^1_e(L,V^\ast(1)/T^\ast(1))\) is \(\HH^1_{\Iw,g}(K,L,T)\),
		\item the orthogonal complement of \(\HH^1_f(L,V^\ast(1)/T^\ast(1))\) is \(\HH^1_{\Iw,f}(K,L,T)\),
		\item the orthogonal complement of \(\HH^1_g(L,V^\ast(1)/T^\ast(1))\) is \(\HH^1_{\Iw,e}(K,L,T)\).
	\end{enumerate}
\end{proposition}

\begin{corollary} \label{corollary:bk-infinite-extension}
	If \(V\) is a non\-/trivial de Rham representation whose Hodge\--Tate weights are not all \(>0\) and \(\HH^1_e(L,V/T) = \HH^1(L,V/T)\), then the extension \(L/K\) is infinite.
\end{corollary}
\begin{proof}
	If \(L\) is a finite extension of \(K\), then, by Remark~\ref{remark:rank-bk-g}, the group \(\HH^1_g(L,T^\ast(1))\) is non\-/trivial, which by Proposition~\ref{proposition:bk-dual-iw}, implies that \(\HH^1_e(L,V/T) \neq \HH^1(L,V/T)\).
\end{proof}

We recall an alternative description of the exponential Bloch\--Kato groups (see~\cite[Proposition~3.2.1]{Ponsinet2025}).

The fundamental exact sequence~\eqref{eq:fundamental-v} induces a strict exact sequence of topological \(G_K\)\=/modules
\begin{equation} \label{eq:fundamental-t}
	0 \rightarrow V/T \rightarrow V_\e/T \rightarrow V_\dR/V_\dR^+ \rightarrow 0,
\end{equation}
where we again denote by \(T \subseteq V_\e\) the image of \(T \subseteq V\) in \(V_\e\), and \(V_\e/T\) is endowed with the quotient topology from \(V_\e\).

\begin{lemma} \label{lemma:proj-bk-exp}
	Let \(E\) be an algebraic extension of \(K\).
	The strict exact sequence of topological \(G_K\)\=/modules~\eqref{eq:fundamental-t} induces a continuous map
	\[
		\HH^0(L,V_\dR/V_\dR^+) \rightarrow \HH^1(L,V/T)
	\]
	whose image is \(\beta_L(\HH^1_e(L,V))\).
\end{lemma}
\begin{proof}
	The commutative diagram of topological \(G_K\)\=/modules whose rows are exact
	\[
		\begin{tikzcd}
			0 \ar{r} & V \ar{r} \ar{d} & V_\e \ar{r} \ar{d} & V_\dR/V_\dR^+ \ar{r} \ar[equal]{d} & 0 \\
			0 \ar{r} & V/T \ar{r} & V_\e/T \ar{r} & V_\dR/V_\dR^+ \ar{r} & 0
		\end{tikzcd}
	\]
	induces a commutative diagram of topological abelian groups whose rows are exact
	\[
		\begin{tikzcd}
			\HH^0(L,V_\dR/V_\dR^+) \ar{r} \ar[equal]{d} &  \HH^1(L,V) \ar{d}{\beta_L} \ar{r} & \HH^1(L,V_\e) \\
			\HH^0(L,V_\dR/V_\dR^+) \ar{r} &  \HH^1(L,V/T) & ,
		\end{tikzcd}
	\]
	which implies the statement.
\end{proof}

If \(M\) is a topological \(G_K\)\=/module, then the associated \emph{maximal discrete \(G_K\)\=/module} is the discrete \(G_K\)\=/module \(M_\delta = \bigcup_{K^\prime} M^{G_{K^\prime}}\), where \(K^\prime\) runs through all the finite extensions of \(K\).

The maximal discrete \(G_K\)\=/modules associated with the strict exact sequence~\eqref{eq:fundamental-t} fit into a short exact sequence of discrete \(G_K\)\=/modules
\begin{equation} \label{eq:e-delta}
	0 \rightarrow V/T \rightarrow E_\delta(V/T) \rightarrow t_V(\Qpbar) \rightarrow 0.
\end{equation}

\begin{proposition} \label{proposition:e-delta-bk}
	Let \(E\) be an algebraic extension of \(K\).
	The short exact sequence of discrete \(G_K\)\=/modules~\eqref{eq:e-delta} induces a short exact sequence
	\[
		0 \rightarrow \HH^1_e(E,V/T) \rightarrow \HH^1(E,V/T) \rightarrow \HH^1(E,E_\delta(V/T)) \rightarrow 0.
	\]
\end{proposition}

\begin{corollary} \label{corollary:bk-finite-extension}
	Let \(L^\prime\) be a finite Galois extension of \(L\), and let
	\[
		\res: \HH^1(L,V/T) \rightarrow \HH^1(L^\prime,V/T)
	\]
	be the restriction map.
	Then, it holds
	\[
			\HH^1_e(L,V/T) = \res^{-1}(\HH^1_e(L^\prime,V/T))_\divisible.
	\]
\end{corollary}
\begin{proof}
	Let \(\Delta = \Gal(L^\prime/L)\) be the Galois group of the extension \(L^\prime/L\).
	Recall that the restriction map \(\res\) factorises through
	\[
		r: \HH^1(L,V/T) \rightarrow \HH^1(L^\prime,V/T)^\Delta \subseteq \HH^1(L^\prime,V/T).
	\]
	By Proposition~\ref{proposition:e-delta-bk}, there exists a commutative diagram whose rows and columns are exact
	\begin{equation} \label{eq:bk-finite-extension}
		\begin{tikzcd}
			0 \ar{r} & \HH^1_e(L,V/T) \ar{r} \ar{d} & \HH^1(L,V/T) \ar{r} \ar{d}{r} & \HH^1(L,E_\delta(V/T)) \ar{d}{r_\delta} \ar{r} & 0 \\
			0 \ar{r} & \HH^1_e(L^\prime,V/T)^\Delta \ar{r} & \HH^1(L^\prime,V/T)^\Delta \ar{r} & \HH^1(L^\prime,E_\delta(V/T))^\Delta & .
		\end{tikzcd}
	\end{equation}
	By the inflation\--restriction exact sequence, there exists isomorphisms
	\[
		\begin{split}
			\Ker(r) & \similarrightarrow \HH^1(\Delta,(V/T)^{G_{L^\prime}}) \\
			\Ker(r_\delta) & \similarrightarrow \HH^1(\Delta,E_\delta(V/T)^{G_{L^\prime}}),
		\end{split}
	\]
	and an injective morphism
	\[
		0 \rightarrow \Coker(r) \rightarrow \HH^2(\Delta,(V/T)^{G_{L^\prime}}).
	\]
	Therefore, the groups \(\Ker(r)\), \(\Coker(r)\) and \(\Ker(r_\delta)\) are all torsion with finite exponent, since the cohomology groups of the finite group \(\Delta\) are torsion with finite exponent (see~\cite[I \S~2.4 Proposition~9]{Serre1994}).
	Hence, it follows from the diagram~\eqref{eq:bk-finite-extension} that the group \(r^{-1}(\HH^1_e(L^\prime,V/T)^\Delta)/\HH^1_e(L,V/T)\) is torsion with finite exponent, which, by Proposition~\ref{proposition:p-adic-topology-torsion}, implies that
	\[
		r^{-1}( \HH^1_e(L^\prime,V/T)^\Delta)_\divisible = \HH^1_e(L,V/T)_\divisible.
	\]
	Moreover, by definition or by Proposition~\ref{proposition:e-delta-bk}, the group \( \HH^1_e(L,V/T)\) is divisible, and thus, it holds
	\[
		\res^{-1}(\HH^1_e(L^\prime,V/T))_\divisible = r^{-1}( \HH^1_e(L^\prime,V/T)^\Delta)_\divisible = \HH^1_e(L,V/T)_\divisible = \HH^1_e(L,V/T).
	\]
\end{proof}

\subsection{Bloch\texorpdfstring{\--}{--}Kato groups over non-perfectoid fields}
\label{subsec:bloch-kato-groups-over-non-perfectoid-fields}

Let \(V\) be a \(p\)\=/adic representation of \(G_K\), and let \(T\) be a \(G_K\)\=/stable lattice in \(V\).
Let \(L\) be an algebraic extension of \(K\).

\begin{proposition} \label{proposition:exp-bk-non-perfectoid}
	Assume that \(V\) is de Rham.
	If \(\hat{L}\) is not perfectoid, then \(\HH^1_e(L,V/T)\) coincides with the image of \(\HH^1_e(L,V)\) by the map \(\beta_L : \HH^1(L,V) \rightarrow \HH^1(L,V/T)\).
\end{proposition}
\begin{proof}
	On the one hand, by Proposition~\ref{proposition:rep-dr} and by Lemma~\ref{lemma:proj-bk-exp}, there exists a continuous map
	\begin{equation} \label{eq:exp}
		t_V(\hat{L}) \similarrightarrow \HH^0(L,V_\dR/V_\dR^+) \rightarrow \HH^1(L,V/T),
	\end{equation}
	whose image is \(\beta_L(\HH^1_e(L,V))\).
	On the other hand, by Proposition~\ref{proposition:e-delta-bk}, the image of \(t_V(L) \subseteq t_V(\hat{L})\) by the map~\eqref{eq:exp} is \(\HH^1_e(L,V/T)\).

	Since \(L\) is dense in \(\hat{L}\) and \(\HH^1(L,V/T)\) is discrete, the images of \(t_V(L)\) and \(t_V(\hat{L})\) in \(\HH^1(L,V/T)\) by the map~\eqref{eq:exp} coincides by continuity, and thus, it holds \(\HH^1_e(L,V/T) = \beta_L(\HH^1_e(L,V))\).
\end{proof}

\begin{corollary} \label{corollary:bk-non-perfectoid}
	Assume that \(V\) is de Rham.
	If \(\hat{L}\) is not perfectoid and \(\HH^1_e(L,V/T) = \HH^1(L,V/T)\), then
	\begin{enumerate}
		\item the subspace \(\HH^1_e(L,V)\) is dense in \(\HH^1(L,V)\),
		\item it holds \(\HH^1_g(L,V) = \HH^1(L,V)\).
	\end{enumerate}
\end{corollary}
\begin{proof}
	Recall that \(\HH^1(L,V)\) is equipped with a structure of \(p\)\=/adic Banach space for which the image of \(\HH^1(L,T)\) in \(\HH^1(L,V)\) is a lattice by Corollary~\ref{corollary:h1-representation-banach}.
	Moreover, by Proposition~\ref{proposition:exp-bk-non-perfectoid}, there exists a commutative diagram of groups
	\[
		\begin{tikzcd}
			\HH^1(L,T) \ar{r}{\alpha_{L}} & \HH^1(L,V) \ar{r}{\beta_{L}} & \HH^1(L,V/T) & \\
			 & \HH^1_e(L,V) \ar[hook]{u} \ar{r} & \HH^1_e(L,V/T) \ar[equal]{u} \ar{r} & 0,
		\end{tikzcd}
	\]
	which implies that the map \(\beta_{L}\) is surjective.
	We conclude, by Lemma~\ref{lemma:banach-dense-subspace} applied to \(X= \HH^1(L,V)\), \(\Xc = \Img(\alpha_L)\), and \(Y = \HH^1_e(L,V)\), that \(\HH^1_e(L,V)\) is dense in \(\HH^1(L,V)\).

	Concerning \(\HH^1_g(L,V)\), on the one hand, by Proposition~\ref{proposition:bkg-closed}, the subspace \(\HH^1_g(L,V)\) is closed in \(\HH^1(L,V)\).
	On the other hand, by the first statement, the subspace \(\HH^1_e(L,V)\) is dense in \(\HH^1(L,V)\).
	Therefore, by the inclusions~\eqref{eq:bk-inclusions}, the subspace \(\HH^1_g(L,V)\) is both closed and dense in \(\HH^1(L,V)\), and hence, it holds \(\HH^1_g(L,V) = \HH^1(L,V)\).
\end{proof}

\subsection{A characterisation of perfectoid fields}
\label{subsec:a-characterisation-of-perfectoid-fields}

Let \(V\) be a \(p\)\=/adic representation of \(G_K\), and let \(T\) be a \(G_K\)\=/stable lattice in \(V\).
Let \(L\) be an algebraic extension of \(K\).

\begin{theorem} \label{theorem:main}
	Assume that \(V\) is a non\-/trivial de Rham representation of \(G_K\) whose Hodge\--Tate weights are not all \(> 0\).
	If \(\HH^1_e(L,V/T) = \HH^1(L,V/T)\), then the field \(\hat{L}\) is perfectoid.
\end{theorem}

\begin{remark}
	By local Tate duality (Proposition~\ref{proposition:bk-dual-iw}), Theorem~\ref{theorem:main} is equivalent to Theorem~\ref{intro:theorem:main} stated in the introduction.
\end{remark}

\begin{proof}
	We proceed by contradiction.
	We thus assume that all the hypotheses of the statement hold and that the field \(\hat{L}\) is not perfectoid.

	We first note that the extension \(L/K\) is infinite by Corollary~\ref{corollary:bk-infinite-extension}.
	For each finite extension \(K^\prime\) of \(K\) contained in \(L\), by Proposition~\ref{proposition:rep-dr}, there is a commutative diagram of \(\Qp\)\=/vector spaces whose rows and columns are exact
	\begin{equation} \label{eq:dr-restriction-exact}
		\begin{tikzcd}
			0 \ar{r} & \HH^1(L,V_\dR^+) \ar{r} & \HH^1(L,V_\dR) \ar{r} & \HH^1(L,V_\dR/V_\dR^+) \ar{r} & 0 \\
			0 \ar{r} & \HH^1(K^\prime,V_\dR^+) \ar{r} \ar{u} & \HH^1(K^\prime,V_\dR) \ar{r} \ar{u} & \HH^1(K^\prime,V_\dR/V_\dR^+) \ar{r} \ar{u} & 0 \\
			& 0 \ar{u} & 0 \ar{u} & 0 \ar{u} & ,
		\end{tikzcd}
	\end{equation}
	where the vertical maps are the restriction maps.
	We will abusively consider the spaces \(\HH^1(L,V_\dR^+)\), \(\HH^1(K^\prime,V_\dR^+)\) and \(\HH^1(K^\prime,V_\dR)\) as subspaces of \(\HH^1(L,V_\dR)\)  via the diagram~\eqref{eq:dr-restriction-exact}.

	We consider the map
	\[
		f_L : \HH^1(L,V_\e) \rightarrow \HH^1(L,V_\dR).
	\]
	We first prove that
	\begin{equation} \label{eq:intersection-non-trivial}
		\Img(f_L) \cap \HH^1(L,V_\dR^+) \neq 0
	\end{equation}
	in \(\HH^1(L,V_\dR)\).
	For each finite extension \(K^\prime\) of \(K\) contained in \(L\), there exists a commutative diagram
	\begin{equation} \label{eq:restriction-intersection}
		\begin{tikzcd}
			\HH^1(L,V_\e) \ar{r}{f_L} & \HH^1(L,V_\dR) \\
			\HH^1(K^\prime,V_\e) \ar{r}{f_{K^\prime}} \ar{u} & \HH^1(K^\prime,V_\dR) \ar{u} \\
			& \phantom{,}0, \ar{u}
		\end{tikzcd}
	\end{equation}
	where the vertical maps are the restriction maps.
	By commutativity of the diagrams~\eqref{eq:dr-restriction-exact} and~\eqref{eq:restriction-intersection}, it holds
	\begin{equation} \label{eq:intersection-finite}
		\Img(f_{K^\prime}) \cap \HH^1(K^\prime,V_\dR^+) \subseteq \Img(f_L) \cap \HH^1(L,V_\dR^+)
	\end{equation}
	in \(\HH^1(L,V_\dR)\).
	Therefore, by the relation~\eqref{eq:intersection-finite}, in order to prove the relation~\eqref{eq:intersection-non-trivial}, it is enough to prove prove that there exists a finite extension \(K^\prime\) of \(K\) contained in \(L\) such that \(\Img(f_{K^\prime}) \cap \HH^1(K^\prime,V_\dR^+)\) is non\-/trivial, which we do by considering the dimensions of the finite dimensional \(\Qp\)\=/vector spaces involved.
	By Proposition~\ref{proposition:bounded-dim-de} and Corollary~\ref{corollary:hyodo}, there exists \(c \in \N\) such that, for each finite extension \(K^\prime\) of \(K\), it holds
	\begin{equation} \label{eq:dimension-ker}
		\dim_{\Qp} \Ker(f_{K^\prime}) \leq c.
	\end{equation}
	By Proposition~\ref{proposition:cohomology-be}, it holds
	\begin{equation} \label{eq:dimension-ve}
		\dim_{\Qp} \HH^1(K^\prime,V_\e) \geq [K^\prime:\Qp] \dim_{\Qp} V.
	\end{equation}
	Therefore, by the combination of the equation~\eqref{eq:dimension-ker} and~\eqref{eq:dimension-ve}, it holds
	\begin{equation} \label{eq:dimensions}
		\begin{split}
			[K^\prime:\Qp] \dim_{\Qp} V - c & \leq \dim_{\Qp} \Img(f_{K^\prime}) \\
			& \leq \dim_{\Qp} \HH^1(K^\prime,V_\dR) = [K^\prime:\Qp]\dim_{\Qp} V,
		\end{split}
	\end{equation}
	where the last equality follows from Proposition~\ref{proposition:rep-dr}.
	Since the representation \(V\) is non\-/trivial and the Hodge\--Tate weights of \(V\) are not all \(> 0\), the vector space \(\Fil^0 \DdR(V)\) is non\-/trivial, and hence by Proposition~\ref{proposition:rep-dr}, it holds
	\begin{equation} \label{eq:non-trivial}
		\dim_{\Qp} \HH^1(K^\prime,V_\dR^+) = [K^\prime:\Qp]\dim_{\Qp} \Fil^0 \DdR(V) \neq 0.
	\end{equation}
	In particular, by the equation~\eqref{eq:non-trivial} and since the extension \(L/K\) is infinite, we may find \(K^\prime\) such that
	\begin{equation} \label{eq:greater-than-constant}
		\dim_{\Qp} \HH^1(K^\prime,V_\dR^+) > c,
	\end{equation}
	in which case, the combination of the relations~\eqref{eq:dimensions} and~\eqref{eq:greater-than-constant} implies that
	\[
		\Img(f_{K^\prime}) \cap \HH^1(K^\prime,V_\dR^+) \neq 0.
	\]

	We now prove that
	\begin{equation} \label{eq:intersection-trivial}
		\Img(f_L) \cap \HH^1(L,V_\dR^+) = 0
	\end{equation}
	in \(\HH^1(L,V_\dR)\).
	By Proposition~\ref{proposition:hyodo} and Corollary~\ref{corollary:bk-non-perfectoid}, there exists a commutative diagram whose rows and columns are exact
	\begin{equation} \label{eq:bk-iso}
		\begin{tikzcd}
			& & & 0 \ar{d} \\
			& & & \HH^1(L,V_\dR^+) \ar{d} \\
			0 \ar{r} & \HH^1_g(L,V)/\HH^1_e(L,V) \ar{r} \ar{d}[sloped]{\sim} & \HH^1(L,V_\e) \ar{r}{f_L} \ar[equal]{d} & \HH^1(L,V_\dR) \ar{d} \\
			0 \ar{r} & \HH^1(L,V)/\HH^1_e(L,V) \ar{r} & \HH^1(L,V_\e) \ar{r} & \HH^1(L,V_\dR/V_\dR^+) \ar{d} \\
			& & & \phantom{.}0.
		\end{tikzcd}
	\end{equation}
	We deduce from the commutativity of the diagram~\eqref{eq:bk-iso} that the relation~\eqref{eq:intersection-trivial} holds.

	Finally, the combination of the relations~\eqref{eq:intersection-non-trivial} and~\eqref{eq:intersection-trivial} yields
	\[
		0 \neq \Img(f_L) \cap \HH^1(L,V_\dR^+) = 0.
	\]
	We thus reach a contradiction.
\end{proof}

The combination of Proposition~\ref{proposition:e-delta-bk} and Theorem~\ref{theorem:main} implies the following corollary.

\begin{corollary} \label{corollary:perfectoid-e-delta}
	Assume that \(V\) is a non\-/trivial de Rham representation of \(G_K\) whose Hodge\--Tate weights are not all \(> 0\).
	If \(\HH^1(L,E_\delta(V/T))\) is trivial, then the field \(\hat{L}\) is perfectoid.
\end{corollary}

\section{Examples}
\label{sec:examples}

\subsection{\texorpdfstring{\(p\)\=/}{p-}divisible groups}
\label{subsec:p-divisible-groups}

We recover results by Coates and Greenberg~\cite{CoatesGreenberg1996} for abelian varieties and by Bondarko~\cite{Bondarko2003} for \(p\)\=/divisible groups from Theorem~\ref{theorem:main}.
We first review results concerning \(p\)\=/divisible groups by Tate~\cite{Tate1967} and Fontaine~\cite[\S~6]{Fontaine1982}.

Let \(G\) be a \(p\)\=/divisible group over \(\Oc_K\) of height \(\height(G) \in \N\) and dimension \(\dim(G) \in \N\).
Let \(T_p (G) = \varprojlim_{\times p} G(\Oc_{\Cp})[p^n]\) be the \(p\)\=/adic Tate module of \(G\), and let \(V_p (G) = \Qp \otimes_{\Zp} T_p (G)\).
Note that \(V_p (G)/T_p (G) = G[p^\infty]\) is the \(p\)\=/primary torsion subgroup of \(G(\Oc_{\Cp})\).
The \(p\)\=/adic representation \(V_p (G)\) of \(G_K\) is crystalline with Hodge\--Tate weights all in \([0,1]\), of multiplicity
\begin{equation} \label{eq:ht-weights-p-div}
	\begin{split}
		m_0(V_p(G)) & = \height(G)-\dim(G), \\
		m_1(V_p(G)) & = \dim(G).
	\end{split}
\end{equation}
We denote by \(G^\vee\) the dual \(p\)\=/divisible group of \(G\).
It holds
\begin{equation} \label{eq:p-div-height-dim-dual}
	\height(G) = \height(G^\vee) = \dim(G) + \dim(G^\vee),
\end{equation}
and there exists an isomorphism of topological \(G_K\)\=/modules
\begin{equation} \label{eq:p-div-duality}
	T_p(G^\vee) \similarrightarrow T_p(G)^\ast(1).
\end{equation}

There exist short exact sequences of \(p\)\=/divisible groups
\begin{gather}
	\label{eq:p-divisible-connected-etale} 0 \rightarrow G^\circ \rightarrow G \rightarrow G^\et \rightarrow 0, \\
	\label{eq:p-divisible-multiplicative-coconnected} 0 \rightarrow G^\mult \rightarrow G^\circ \rightarrow G^{\circ \circ} \rightarrow 0,
\end{gather}
where
\begin{itemize}
	\item \(G^\circ\) is the connected component of \(G\),
	\item \(G^\et\) is the maximal étale quotient of \(G\),
	\item \(G^\mult\) is the multiplicative component of \(G\), that is, the maximal \(p\)\=/divisible subgroup of \(G\) whose dual is étale,
	\item \(G^{\circ\circ}\) is the maximal coconnected quotient of \(G^\circ\), that is, the maximal quotient of \(G^\circ\) whose dual is connected.
\end{itemize}

\begin{remark} \label{remark:p-div-dim-height}
	The dimension and height are additive in exact sequences of \(p\)\=/divisible groups.
	Moreover, by definition, it holds \(\dim(G) = \dim(G^\circ)\), and \(\dim(G) = 0\) if and only if \(G^\circ\) is trivial.
\end{remark}

The \(p\)\=/divisible groups associated with \(G\) from~\eqref{eq:p-divisible-connected-etale} and~\eqref{eq:p-divisible-multiplicative-coconnected} defines the following subquotient representations of \(T_p(G)\)
\begin{gather*}
	0 \rightarrow T_p(G^\circ) \rightarrow T_p(G) \rightarrow T_p(G^\et) \rightarrow 0, \\
	0 \rightarrow T_p(G^\mult) \rightarrow T_p(G^\circ) \rightarrow T_p(G^{\circ \circ}) \rightarrow 0.
\end{gather*}

\begin{proposition} \label{proposition:p-div-representations}
	The following properties hold.
	\begin{enumerate}
		\item The representation \(V_p G^\et\) is the maximal quotient of \(V_p G\) whose Hodge\--Tate weights are all \(0\).
		\item The representation \(V_p G^\circ\) is the maximal subrepresentation of \(V_p G\) which does not admit a non\-/trivial quotient whose Hodge\--Tate weights are all \(0\).
		\item The representation \(V_p G^\mult\) is the maximal subrepresentation of \(V_p G\) whose Hodge\--Tate weights all \(1\).
		\item The representation \(V_p G^{\circ\circ}\) is the maximal subquotient representation of \(V_p G\) which does not admit a non\-/trivial quotient whose Hodge\--Tate weights are all \(0\), and does not admit a subrepresentation whose Hodge\--Tate weights all \(1\).
	\end{enumerate}
\end{proposition}
\begin{proof}
	We prove the statements concerning \(V_p G^\et\) and \(V_p G^\circ\) simultaneously.
	The statements concerning \(V_p G^\mult\) and \(V_p G^{\circ\circ}\) follow from the statements concerning \(V_p G^\et\) and \(V_p G^\circ\) by duality~\eqref{eq:p-div-duality}.

	We proceed by contradiction.
	Assume that there exists a non\-/trivial quotient \(W\) of \(V_p G^\circ\) whose Hodge\--Tate weights are all \(0\).
	Then by Tate~\cite[Proposition~12]{Tate1967} and duality, there exists a \(p\)\=/divisible group \(H\) over \(\Oc_K\) such that \(W = V_p H\) and the quotient map \(V_p G^\circ \rightarrow V_p H\) is induced by a morphism of \(p\)\=/divisible groups \(G^\circ \rightarrow H\).
	Since the Hodge\--Tate weights of \(V_p H\) are all \(0\), the \(p\)\=/divisible group \(H\) is étale by the equation~\eqref{eq:ht-weights-p-div} and Remark~\ref{remark:p-div-dim-height}.
	We reach a contradiction since there are no non\-/trivial maps from a connected \(p\)\=/divisible group to an étale one over \(\Oc_K\).
\end{proof}

\begin{remark}
	If \(G\) is connected, then \(\De(V_p(G))\) is trivial (see for instance~\cite{Demazure1972}, or by the combination of Proposition~\ref{proposition:p-div-representations} and the \enquote{weakly admissible implies admissible} Theorem from Remark~\ref{remark:p-adic-rep-results}).
	In particular, if \(G\) is connected and coconnected, then by duality~\eqref{eq:p-div-duality}, the spaces \(\De(V_p(G))\) and \(\De(V_p(G)^\ast(1))\) are both trivial.
\end{remark}

If \(E\) is an algebraic extension of \(K\), then we set
\[
	G(\Oc_E) = \varinjlim G(\Oc_{K^\prime}),
\]
and where \(K^\prime\) runs through the finite extensions of \(K\) contained in \(E\), and the transition morphisms are the inclusion maps.

Let \(t_G\) be the tangent space of \(G\).
Recall that the short exact sequence of discrete \(G_K\)\=/modules (see~\cite[\S~2.4]{Tate1967})
\begin{equation} \label{eq:p-div-log-ses}
	0 \rightarrow G[p^\infty] \rightarrow G(\Oc_{\Qpbar}) \xrightarrow{\log_G} t_G(\Qpbar) \rightarrow 0,
\end{equation}
induces a short exact sequence
\[
	0 \rightarrow G(\Oc_L) \otimes_{\Zp} \Qp/\Zp \xrightarrow{\kappa_L} \HH^1(L,G[p^\infty]) \rightarrow \HH^1(L,G(\Oc_{\Qpbar})) \rightarrow 0.
\]
The map \(\kappa_L\) is the \emph{Kummer map}.
By Fontaine~\cite[\S~8.4]{Fontaine2003} and Bloch and Kato~\cite[Example~3.10.1]{BlochKato1990} (see also~\cite[Remark~3.2.2]{Ponsinet2025}), the short exact sequence~\eqref{eq:p-div-log-ses} is isomorphic to the short exact sequence~\eqref{eq:e-delta} associated with \(T_p(G)\).
In particular, it holds
\[
	\Img(\kappa_L) = \HH^1_e(L,G[p^\infty]).
\]
Moreover, from Corollary~\ref{corollary:perfectoid-e-delta}, we recover the following result due to Bondarko~\cite{Bondarko2003} for connected \(p\)\=/divisible groups over \(\Oc_K\), which is itself a generalisation of a result by Coates and Greenberg~\cite[Proposition~5.4]{CoatesGreenberg1996} for connected \(p\)\=/divisible groups arising from abelian varieties.

\begin{proposition}[Coates--Greenberg, Bondarko] \label{proposition:p-divisible-group}
	Assume that \(G\) is connected and \(\height(G) > \dim(G)\).
	If \(\HH^1(L,G(\Oc_{\Qpbar})) = 0\), then \(\hat{L}\) is perfectoid.
\end{proposition}

\begin{proof}
	The representation \(V_p(G)\) is non\-/trivial crystalline with Hodge\--Tate weights not all \(>0\) by the equation~\eqref{eq:ht-weights-p-div}.
	Therefore, Corollary~\ref{corollary:perfectoid-e-delta} applies and implies that the field \(\hat{L}\) is perfectoid.
\end{proof}

\begin{remark}
	Bondarko's result~\cite{Bondarko2003} concerns the more general setting of a finite dimensional commutative formal groups \(\Fc\) defined over the valuation ring \(\Oc_k\) of a complete discrete valuation field \(k\) with perfect residue field of positive characteristic, and it characterises the algebraic extensions \(l\) of \(k\) such that the completed field \(\hat{l}\) is perfectoid in terms of the Galois cohomology of \(l\) with coefficients in \(\Fc\).
	However, Bondarko's result requires the additional hypothesis that the valuation of the field \(l\) is non\-/discrete.
	Therefore, for connected \(p\)\=/divisible groups over \(\Oc_K\), Proposition~\ref{proposition:p-divisible-group} slightly improves upon Bondarko's result~\cite{Bondarko2003} by removing this hypothesis.
\end{remark}

\begin{example} \label{example:abelian-varieties}
	Let \(A\) be an abelian variety defined over \(K\).
	Recall that the representation \(V_p(A)\) is de Rham with Hodge\--Tate weights \(0\) and \(1\) both of multiplicity \(\dim(A)\) the dimension of \(A\).
	Moreover, the Weil pairing induces an isomorphism \(V_p(A)^\ast(1) \simeq V_p(A^\vee)\), where \(A^\vee\) is the dual abelian variety of \(A\), and the space \(\De(V_p(A))\) is trivial~\cite[Example~3.11]{BlochKato1990}.
	Further assume that \(A\) has semistable reduction.
	Let \(\hat{\Ac}\) be the formal group associated with a Néron model \(\Ac\) of \(A\) over \(\Oc_K\), and let \(\hat{\Ac}(p)\) be the connected \(p\)\=/divisible group associated with \(\hat{\Ac}\) over \(\Oc_K\).
	Then, the connected \(p\)\=/divisible group \(\hat{\Ac}(p)\) is of dimension \(\dim(\hat{\Ac}(p)) = \dim(A)\) and finite height \(\dim(A)\leq \height(\hat{\Ac}(p)) \leq 2\dim(A)\), the space \(\De(V_p(\hat{\Ac}(p)))\) is trivial, and \(V_p(A)/V_p(\hat{\Ac}(p))\) is the maximal quotient of \(V_p(A)\) whose Hodge\--Tate weigths are all \(0\) (see~\cite[Exp. IX]{SGAVII:1} and~\cite[Remark~3.3.6~2]{Ponsinet2025}).
	In particular, abelian varieties provide examples of connected \(p\)\=/divisible groups satisfying Proposition~\ref{proposition:p-divisible-group}.
	For instance, if \(A\) has good non\-/ordinary reduction, then \(\hat{\Ac}(p)\) is such a \(p\)\=/divisible group.
\end{example}

\subsection{Bloch\texorpdfstring{\--}{--}Kato groups over perfectoid fields}
\label{subsec:bloch-kato-groups-over-perfectoid-fields}

We review computation of the Bloch\--Kato groups over perfectoid fields by Berger~\cite{Berger2005}, and the author~\cite{Ponsinet2025,Ponsinet2024}.
These results provide examples of de Rham representations \(V\) satisfying the hypotheses of Theorem~\ref{theorem:main}, that is, non\-/trivial de Rham representations whose Hodge\--Tate weights are not all \(>0\), and such that, if \(\hat{L}\) is perfectoid, then \(\HH^1_e(L,V/T) = \HH^1(L,V/T)\).

Let \(V\) be a de Rham representation of \(G_K\).
Let \(V^{\leq 0}\) be the maximal quotient of \(V\) whose Hodge\--Tate weights are all \(\leq 0\), and let \(T^{\leq 0}\) be the image of \(T\) in \(V^{\leq 0}\).
The quotient map \(V/T \rightarrow V^{\leq 0}/T^{\leq 0}\) induces a morphism
\[
	\pi_0: \HH^1(L,V/T) \rightarrow \HH^1(L,V^{\leq 0}/T^{\leq 0}).
\]
Recall that it holds \(\HH^1_e(L,V/T) \subseteq \Ker(\pi_0)\) (by Lemma~\ref{lemma:proj-bk-exp}, see~\cite{Ponsinet2024} for details).
Moreover, if \(\hat{L}\) is perfectoid, then the \(p\)\=/cohomological dimension of \(G_L\) is \(\leq 1\) (see~\cite[Proposition~4.1.3]{Ponsinet2024}), which implies that the map \(\pi_0\) is surjective.
Thus, if \(\hat{L}\) is perfectoid, the map \(\pi_0\) induces a morphism
\[
	\HH^1(L,V/T)/\HH^1_e(L,V/T) \rightarrow \HH^1(L,V^{\leq 0}/T^{\leq 0}) \rightarrow 0.
\]

Recall the main result of~\cite{Ponsinet2024}.

\begin{theorem} \label{theorem:bk-perfectoid}
	Let \(n \geq 1\) be an integer.
	If \(\hat{L}\) is a perfectoid field such that \(L\) is dense in \((\BdR^+/\Fil^n \BdR)^{G_L}\), then the following holds.
\begin{enumerate}
	\item If the Hodge\--Tate weights of \(V\) are all \(\leq n\), then the map \(\pi_0\) induces an isomorphism
	\[
		\HH^1(L,V/T)/\HH^1_e(L,V/T) \similarrightarrow \HH^1(L,V^{\leq 0}/T^{\leq 0}).
	\]
	\item If there is no non\-/trivial quotient representation of \(V\) whose Hodge\--Tate weights are all in \(\Z \setminus [1,n]\), then \(\HH^1_e(L,V/T) = \HH^1(L,V/T)\).
\end{enumerate}
\end{theorem}

\begin{remark} \label{remark:case-n-1}
	Concerning Theorem~\ref{theorem:bk-perfectoid} in the case \(n=1\).
	\begin{enumerate}
		\item By the Ax\--Sen\--Tate theorem (Theorem~\ref{theorem:ax-sen-tate}), the field \(L\) is dense in \((\BdR^+/\Fil^1 \BdR)^{G_L}\similarrightarrow \Cp^{G_L} \similarrightarrow \hat{L}\).
		Therefore, in the case \(n=1\), Theorem~\ref{theorem:bk-perfectoid} applies to all perfectoid field \(\hat{L}\).
		\item In particular, in the case \(n=1\), Theorem~\ref{theorem:bk-perfectoid} applies to the \(p\)\=/adic Tate modules of abelian varieties and \(p\)\=/divisible groups and implies a result by Coates and Greenberg~\cite[Proposition~4.3]{CoatesGreenberg1996} (see~\cite[Remark~3.5.5]{Ponsinet2025}).
	\end{enumerate}
\end{remark}

By Remark~\ref{remark:case-n-1}, the combination of Theorem~\ref{theorem:main} and Theorem~\ref{theorem:bk-perfectoid} yields.

\begin{corollary} \label{corollary:equivalence}
	If \(V\) is a non\-/trivial de Rham representation such that
	\begin{enumerate}
		\item the Hodge\--Tate weights of \(V\) are not all \(>0\),
		\item there is no non\-/trivial quotient representation of \(V\) whose Hodge\--Tate weights are all in \(\Z \setminus \{1\}\),
	\end{enumerate}
	then the field \(\hat{L}\) is perfectoid if and only if \(\HH^1_e(L,V/T) = \HH^1(L,V/T)\).
\end{corollary}

\begin{example}
	Abelian varieties and \(p\)\=/divisible groups provide examples of de Rham representations satisfying Corollary~\ref{corollary:equivalence}.
	For instance, the \(p\)\=/adic Tate module \(T_p(G)\) of a non\-/trivial connected \(p\)\=/divisible group \(G\) defined over \(K\) is such a representation by Proposition~\ref{proposition:p-div-representations}.
	In particular, if \(A\) is a non\-/trivial abelian variety \(A\) defined over \(K\) with good supersingular reduction, then \(T_p(A) = T_p(\hat{\Ac}(p))\), and thus, the \(p\)\=/adic Tate module \(T_p(A)\) satisfies Corollary~\ref{corollary:equivalence} (see Example~\ref{example:abelian-varieties}).
\end{example}

\begin{example}
		The \enquote{weakly admissible implies admissible} Theorem from Remark~\ref{remark:p-adic-rep-results} allows to construct examples of representations satisfying simultaneously the hypotheses of Theorem~\ref{theorem:main}, and the hypotheses of either statements of Theorem~\ref{theorem:bk-perfectoid} for any integer \(n \geq 1\).
\end{example}

Assume that \(L=K_\cyc\).
Berger~\cite{Berger2005}, generalising results by Perrin-Riou~\cite{Perrin-Riou2000,Perrin-Riou2001}, has proved that the Pontryagin dual of the kernel of the map induced by \(\pi_0\)
\[
	\HH^1(K_\cyc,V/T)/\HH^1_e(K_\cyc,V/T) \rightarrow \HH^1(K_\cyc,V^{\leq 0}/T^{\leq 0}) \rightarrow 0
\]
is a finitely generated \(\Zp\)\=/module.
Moreover, if each subquotient representation \(W\) of \(V^\ast(1)\) satisfies \(\HH^0(K_\cyc,W)=0\), then the map \(\pi_0\) induces an isomorphism
\[
	\HH^1(K_\cyc,V/T)/\HH^1_e(K_\cyc,V/T) \similarrightarrow \HH^1(K_\cyc,V^{\leq 0}/T^{\leq 0}).
\]

\begin{remark}
	Berger~\cite{Berger2024} has proved that if \(L/K\) is an infinitely ramified \(\Zp\)\=/extension, then \(L\) is not dense in \((\BdR^+/\Fil^2 \BdR)^{G_L}\).
	In particular, the field \(K_\cyc\) is not dense in \((\BdR^+/\Fil^n\BdR)^{G_{K_\cyc}}\) for any \(n \geq 2\).
\end{remark}

\subsection{Hodge\texorpdfstring{\--}{--}Tate weights}
\label{subsec:hodge-tate-weights}

Let \(L\) be an algebraic extension of \(K\).
We provide examples of de Rham representations \(V\) of \(G_K\) violating the hypothesis on the Hodge\--Tate weights from Theorem~\ref{theorem:main}, and such that \(\HH^1_e(L,V/T) = \HH^1(L,V/T)\) while \(\hat{L}\) is not perfectoid.

\begin{proposition} \label{proposition:bk-ht}
	Assume that \(V\) is a de Rham representation of \(G_K\) such that
	\begin{enumerate}
		\item the Hodge\--Tate weights of \(V\) are all \(\geq 1\),
		\item the spaces \(\DDb_{\e,K^\prime}(V)\) and \(\DDb_{\e,K^\prime}(V^\ast(1))\) are trivial for each finite extension \(K^\prime\) of \(K\) contained in \(L\),
		\item the group
		\[
			\varprojlim \HH^0(K^\prime,V^\ast(1)/T^\ast(1))/(\HH^0(K^\prime,V^\ast(1))/\HH^0(K^\prime,T^\ast(1))),
		\]
		where \(K^\prime\) runs over all the finite extensions of \(K\) contained in \(L\), and the transition morphisms are the corestriction maps, is trivial.
	\end{enumerate}
	Then it holds \(\HH^1_e(L,V/T) = \HH^1(L,V/T)\).
\end{proposition}

\begin{remark}
	About the hypotheses of Proposition~\ref{proposition:bk-ht}.
	\begin{enumerate}
		\item If \(V\) is semistable with Hodge\--Tate weights all \(\geq 2\), then the first two hypotheses hold (see for instance~\cite[Lemme~6.7]{Berger2002}).
		\item The last hypothesis holds if the group \(\HH^0(L,V^\ast(1)/T^\ast(1))\) is trivial.
	\end{enumerate}
\end{remark}

\begin{proof}
	Firstly, by Corollary~\ref{corollary:bk-exp-full}, for each finite extension \(K^\prime\) of \(K\) contained in \(L\), it holds
	\begin{equation} \label{eq:exp-iso}
		\HH^1_e(K^\prime,V) = \HH^1(K^\prime,V).
	\end{equation}

	To conclude, we consider the map
	\begin{equation} \label{eq:V_V/T}
		\varinjlim \HH^1(K^\prime,V) \xrightarrow{(\beta_{K^\prime})_{K^\prime}} \varinjlim \HH^1(K^\prime,V/T) \similarrightarrow \HH^1(L,V/T),
	\end{equation}
	where \(K^\prime\) runs over all the finite extensions of \(K\) contained in \(L\), and the transition morphisms are the restriction maps.
	By the equalities~\eqref{eq:exp-iso} and the definition of the exponential Bloch\--Kato group \(\HH^1_e(L,V/T)\), if the map~\eqref{eq:V_V/T} is surjective, then \(\HH^1_e(L,V/T) = \HH^1(L,V/T)\).
	Recall that the cokernel of the map \(\beta_{K^\prime}\) is isomorphic to the torsion subgroup \(\HH^2(K^\prime,T)_\torsion\) of \(\HH^2(K^\prime,T)\), and thus, the surjectivity of the map~\eqref{eq:V_V/T} is equivalent to the vanishing of \(\varinjlim \HH^2(K^\prime,T)_\torsion\), which in turn, by local Tate duality~\eqref{eq:local-tate-duality-iw}, is equivalent to the vanishing of the group
	\[
		\begin{split}
			& \varprojlim \HH^0(K^\prime,V^\ast(1)/T^\ast(1))/\HH^0(K^\prime,V^\ast(1)/T^\ast(1))_\divisible \\
			= & \varprojlim \HH^0(K^\prime,V^\ast(1)/T^\ast(1))/(\HH^0(K^\prime,V^\ast(1))/\HH^0(K^\prime,T^\ast(1))),
		\end{split}
	\]
	where the transition maps are the corestriction maps.
\end{proof}

\begin{example}
	Let \(i \in \Z\).
	The structure of \(\Zp\)\=/module of \(\HH^n(K,\Zp(i))\) is known (see for instance~\cite[Proposition~7.3.10]{NeukirchSchmidtWingberg2008}).
	In particular, if \(p \neq 2\), \(K=\Qp\) and \(i \neq 0\), then \(\HH^0(\Qp,\Qp(i)/\Zp(i))\) is trivial.
	Moreover, if \(i \neq 0\), then the space \(\De(\Qp(i))\) is trivial.
	Therefore, by Proposition~\ref{proposition:bk-ht}, if  \(p \neq 2\) and \(i \geq 2\), then it holds
	\[
		\HH^1_e(\Qp,\Qp(i)/\Zp(i)) = \HH^1(\Qp,\Qp(i)/\Zp(i)).
	\]
\end{example}

A \(p\)\=/adic representation \(V\) of \(G_K\) is \emph{unramified} if the inertia subgroup \(I_K\) of \(G_K\) acts trivially on \(V\).

\begin{remark} \label{remark:sen-ht-0}
	Let \(V\) be a de Rham representation of \(G_K\).
	By a result of Sen~\cite{Sen1973}, there exists a finite extension \(K^\prime\) of \(K\) such that \(V|_{K^\prime}\) is unramified if and only if the Hodge\--Tate weights of \(V\) are all \(0\).
\end{remark}

The extension \(L/K\) is \emph{infinitely wildly ramified} if the \(p\)\=/adic valuation of the ramification index \(e(K^\prime/K)\) is unbounded as \(K^\prime\) runs through the finite extensions of \(K\) contained in \(L\).

\begin{lemma} \label{lemma:lim-torsion-0}
	Let \(V\) be an unramified \(p\)\=/adic representation of \(G_K\).
	If the extension \(L/K\) is infinitely wildly ramified, then the group
	\[
		\varprojlim \HH^0(K^\prime,V/T)/(\HH^0(K^\prime,V)/\HH^0(K^\prime,T)),
	\]
	where \(K^\prime\) runs over all the finite extensions of \(K\) contained in \(L\), and the transition morphisms are the corestriction maps, is trivial.
\end{lemma}
\begin{proof}
	We set \((K_n)_{n \in \N}\) a sequence of finite extensions of \(K\) contained in \(L\) such that \(K_0 = K\), \(K_n \subseteq K_{n+1}\), and \(L = \bigcup_{n\in\N} K_n\).
	For each \(n \in \N\) and \(m \geq n\), we denote by  \(K_n^m\) be the maximal unramified extension of \(K_n\) in \(K_m\), we set
	\[
		\begin{split}
			R_n = \HH^0(K_n,V/T)/(\HH^0(K_n,V)/\HH^0(K_n,T)), \\
			R_n^m = \HH^0(K_n^m,V/T)/(\HH^0(K_n^m,V)/\HH^0(K_n^m,T)),
		\end{split}
	\]
	and, we denote the corestriction maps by
	\[
		\begin{tikzcd}
			R_m \ar{rr}{c_{m,n}} \ar{rd}[swap]{c^m_n} & & R_n \\
			& R_n^m \ar{ur}[swap]{c_{m,n}^\ur} & .
		\end{tikzcd}
	\]

	Since \(V\) is unramified, there is an isomorphism
	\begin{equation} \label{eq:iso-unramified}
		\HH^0(K_n^m,V/T) \similarrightarrow \HH^0(K_m,V/T).
	\end{equation}
	Since \(V\) is finite dimensional, there exists \(n_0 \in \N\) such that for all \(n \geq n_0\),
	it holds
	\begin{equation} \label{eq:iso-invariants}
		\HH^0(K_{n_0},V) = \HH^0(K_n,V).
	\end{equation}
	For each \(n \geq n_0\) and \(m \geq n\), the isomorphisms~\eqref{eq:iso-unramified} and~\eqref{eq:iso-invariants} induce an isomorphism
	\[
		\varepsilon: R_n^m \similarrightarrow R_m.
	\]
	Therefore, it holds
	\begin{equation} \label{eq:corestriction}
		\begin{split}
			c_{m,n}(R_m) & = c_{m,n}^\ur \circ c^m_n(R_m) = c_{m,n}^\ur ( c^m_n \circ \varepsilon (R_n^m)) \\
			& = c_{m,n}^\ur ([K_m:K_n^m]\cdot R_n^m) =  c_{m,n}^\ur( e(K_m/K_n) \cdot R_n^m) \\
			& \subseteq e(K_m/K_n) \cdot R_n.
		\end{split}
	\end{equation}

	By Lemma~\ref{lemma:finite-torsion}, the groups \(R_n\) are finite of order say \(p^{a_n}\).
	Finally, since \(L/K\) is infinitely wildly ramified, there exists \(m\) sufficiently large such that \(e(K_m/K_n) \geq p^{a_n}\), and hence, the equation~\eqref{eq:corestriction} yields \(c_{m,n}(R_m) = 0\), and thus, the limit \(\varprojlim R_n\) is trivial.
\end{proof}

\begin{corollary} \label{corollary:infinitely-wildly-ramified}
	Assume that \(V\) is a de Rham representation of \(G_K\) such that
	\begin{enumerate}
		\item the Hodge\--Tate weights of \(V\) are all \(1\),
		\item the spaces \(\DDb_{\e,K^\prime}(V)\) and \(\DDb_{\e,K^\prime}(V^\ast(1))\) are trivial for each finite extension \(K^\prime\) of \(K\).
	\end{enumerate}
	If the extension \(L/K\) is infinitely wildly ramified, then it holds \(\HH^1_e(L,V/T) = \HH^1(L,V/T)\).
\end{corollary}
\begin{proof}
	The Hodge\--Tate weights of \(V^\ast(1)\) are all \(0\), and hence, by Sen (Remark~\ref{remark:sen-ht-0}), there exists a finite extension \(K^\prime\) of \(K\) such that \(V^\ast(1)|_{K^\prime}\) is unramified, and we may assume that \(K^\prime/K\) is Galois, taking the Galois closure of \(K^\prime\) if necessary.
	We set \(L^\prime = LK^\prime\) the compositum of \(L\) and \(K^\prime\), then \(L^\prime/L\) is a finite Galois extension, and \(L^\prime/K^\prime\) is infinitely wildly ramified.
	The combination of Proposition~\ref{proposition:bk-ht} and Lemma~\ref{lemma:lim-torsion-0} applies to  \(V^\ast(1)|_{K^\prime}\) and \(L^\prime/K^\prime\) and yields \(\HH^1_e(L^\prime,V/T) = \HH^1(L^\prime,V/T)\).

	The \(p\)\=/cohomological dimension of an infinitely wildly ramified extension is \(\leq 1\) (see~\cite[II \S~2.3 Proposition~4 and II \S~5.6 Lemme~3]{Serre1994}).
	Therefore, by Proposition~\ref{proposition:cohomological-dimension}, the group \(\HH^1(L,V/T)\) is divisible.
	We conclude by Corollary~\ref{corollary:bk-finite-extension}
	\[
		\begin{split}
			\HH^1_e(L,V/T) & = \res^{-1}(\HH^1_e(L^\prime,V/T))_\divisible \\
			& = \res^{-1}(\HH^1(L^\prime,V/T))_\divisible \\
			& = \HH^1(L,V/T)_\divisible \\
			& = \HH^1(L,V/T).
		\end{split}
	\]
\end{proof}

\begin{example}
	Abelian varieties provide examples of representations satisfying the hypotheses of Corollary~\ref{corollary:infinitely-wildly-ramified}.
	For instance, if \(A\) is an abelian variety defined over \(K\) with good ordinary reduction, then \(V_p(\hat{\Ac}(p))\) is such a representation (see Example~\ref{example:abelian-varieties}).
\end{example}

\begin{remark}
	Corollary~\ref{corollary:infinitely-wildly-ramified} generalises a result by Coates and Greenberg for abelian varieties~\cite[Proposition~4.7]{CoatesGreenberg1996}.
\end{remark}

\begin{remark}
	Coates and Greenberg~\cite[Lemma~2.12]{CoatesGreenberg1996} have proved that if the field \(\hat{L}\) is perfectoid, then the extension \(L/K\) is infinitely wildly ramified.
	Therefore, for the representations \(V\) satisfying the hypotheses of Corollary~\ref{corollary:infinitely-wildly-ramified}, there exists a larger class of fields than perfectoid fields, namely the infinitely wildly ramified extensions \(L\) of \(K\), over which \(\HH^1_e(L,V/T) = \HH^1(L,V/T)\) holds.
\end{remark}

\printbibliography[heading=bibintoc]

@Article{Artusa2024,
	author		= {Marco Artusa},
	title		= {Duality for condensed cohomology of the Weil
                           group of a \(p\)\=/adic field},
	journaltitle	= {Documenta Mathematica},
	date		= {2024},
	volume		= {29},
	number		= {6},
	pages		= {1381--1434},
	doi		= {10.4171/DM/977},
}

@Article{Ax1970,
	author		= {James Ax},
	title		= {Zeros of polynomials over local fields -- the
                           Galois action},
	journaltitle	= {Journal of Algebra},
	date		= {1970},
	volume		= {15},
	number		= {3},
	pages		= {417--428},
	doi		= {10.1016/0021-8693(70)90069-4},
}

@Unpublished{BarthelSchlankStapletonWeinstein2024,
	author		= {Tobias Barthel and Tomer Moshe Schlank and
                           Nathaniel Stapleton and Jared Weinstein},
	title		= {On the rationalization of the
                           \(\operatorname{K}(n)\)\=/local sphere},
	date		= {2024},
	url		= {https://arxiv.org/abs/2402.00960},
}

@Article{Berger2002,
	author		= {Laurent Berger},
	title		= {Représentations \(p\)\=/adiques et équations
                           différentielles},
	journaltitle	= {Inventiones mathematicae},
	date		= {2002},
	volume		= {148},
	number		= {2},
	pages		= {219--284},
	doi		= {10.1007/s002220100202},
}

@Article{Berger2005,
	author		= {Laurent Berger},
	title		= {Représentations de de Rham et normes
                           universelles},
	journaltitle	= {Bulletin de la Société Mathématique de
                           France},
	date		= {2005},
	volume		= {133},
	number		= {4},
	pages		= {601--618},
	doi		= {10.24033/bsmf.2498},
}

@Article{Berger2024,
	author		= {Laurent Berger},
	title		= {Kähler differentials and
                           \(\Zp\)\=/extensions},
	journaltitle	= {Journal de Théorie des Nombres de Bordeaux},
	date		= {2024},
	volume		= {36},
	number		= {3},
	pages		= {1077--1084},
	doi		= {10.5802/jtnb.1308},
}

@InCollection{BlochKato1990,
	author		= {Spencer Bloch and Kazuya Kato},
	title		= {\(L\)\=/functions and Tamagawa numbers of
                           motives},
	date		= {1990},
	booktitle	= {The Grothendieck Festschrift, Volume I},
	booksubtitle	= {A collection of articles written in honor
                           of the 60th birthday of Alexander
                           Grothendieck},
	series		= {Progress in Mathematics},
	number		= {86},
	publisher	= {Birkhäuser},
	pages		= {333--400},
	doi		= {10.1007/978-0-8176-4574-8_9},
}

@Article{Bondarko2003,
	author		= {Mikhail Vladimirovich Bondarko},
	title		= {Cohomology of formal group moduli and
                           deeply ramified extensions},
	journaltitle	= {Mathematical Proceedings of the Cambridge
                           Philosophical Society},
	date		= {2003},
	volume		= {135},
	number		= {1},
	pages		= {19--24},
	doi		= {10.1017/S0305004102006485},
}

@Book{Bourbaki1971:TG:1-4,
	author		= {Nicolas Bourbaki},
	title		= {Topologie générale},
	date		= {1971},
	subtitle	= {Chapitres 1 à 4},
	series		= {Éléments de Mathématique},
	publisher	= {Springer},
	pagetotal	= {xv+357},
	doi		= {10.1007/978-3-540-33982-3},
}

@Book{Bourbaki1974:TG:5-10,
	author		= {Nicolas Bourbaki},
	title		= {Topologie générale},
	date		= {1974},
	subtitle	= {Chapitres 5 à 10},
	series		= {Éléments de Mathématique},
	publisher	= {Springer},
	pagetotal	= {vii+317},
	doi		= {10.1007/978-3-540-34486-5},
}

@Article{CoatesGreenberg1996,
	author		= {John Coates and Ralph Greenberg},
	title		= {Kummer theory for abelian varieties over
                           local fields},
	journaltitle	= {Inventiones mathematicae},
	date		= {1996},
	volume		= {124},
	number		= {1},
	pages		= {129--174},
	doi		= {10.1007/s002220050048},
}

@Article{Colmez1998,
	author		= {Pierre Colmez},
	title		= {Théorie d'Iwasawa des représentations de de
                           Rham d'un corps local},
	journaltitle	= {Annals of Mathematics},
	date		= {1998},
	volume		= {148},
	number		= {2},
	pages		= {485--571},
	doi		= {10.2307/121003},
}

@Article{ColmezFontaine2000,
	author		= {Pierre Colmez and Jean-Marc Fontaine},
	title		= {Construction des représentations \(p\)\=/adiques
                           semi-stables},
	journaltitle	= {Inventiones mathematicae},
	date		= {2000},
	volume		= {140},
	number		= {1},
	pages		= {1--43},
	doi		= {10.1007/s002220000042},
}

@Book{Demazure1972,
	author		= {Michel Demazure},
	title		= {Lectures on \(p\)\=/divisible groups},
	date		= {1972},
	series		= {Lecture Notes in Mathematics},
	publisher	= {Springer},
	pagetotal	= {viii+100},
	doi		= {10.1007/BFb0060741},
}

@Book{Emerton2017,
	author		= {Matthew Emerton},
	title		= {Locally analytic vectors in representations
                           of locally \(p\)\=/adic analytic groups},
	date		= {2017},
	volume		= {248},
	series		= {Memoirs of the American Mathematical
                           Society},
	number		= {1175},
	publisher	= {Amercian Mathematical Society},
	doi		= {10.1090/memo/1175},
}

@Book{FarguesFontaine2018,
	author		= {Laurent Fargues and Jean-Marc Fontaine},
	title		= {Courbes et fibrés vectoriels en théorie de
                           Hodge \(p\)\=/adique},
	date		= {2018},
	series		= {Astérisque},
	number		= {406},
	publisher	= {Société mathématique de France},
	pagetotal	= {xiii+382},
	doi		= {10.24033/ast.1056},
}

@Article{Fontaine1982,
	author		= {Jean-Marc Fontaine},
	title		= {Sur certains types de représentations \(p\)\=/adiques
                           du groupe de Galois d'un corps local; construction
                           d'un anneau de Barsotti-Tate},
	journaltitle	= {Annals of Mathematics},
	date		= {1982},
	volume		= {115},
	number		= {3},
	pages		= {529--577},
	doi		= {10.2307/2007012},
}

@InCollection{Fontaine1992,
	author		= {Jean-Marc Fontaine},
	title		= {Valeurs spéciales des fonctions \(L\) des motifs},
	date		= {1992},
	booktitle	= {Séminaire Bourbaki},
	booksubtitle	= {Volume 1991/92, Exposés 745-759},
	series		= {Astérisque},
	number		= {206},
	note		= {Exposé 751},
	publisher	= {Société mathématique de France},
	pages		= {205--249},
	doi		= {10.24033/ast.136},
}

@InCollection{Fontaine1994:II,
	author		= {Jean-Marc Fontaine},
	title		= {Le corps des périodes \(p\)\=/adiques},
	date		= {1994},
	booktitle	= {Périodes \(p\)\=/adiques},
	booksubtitle	= {Séminaire de Bures, 1988},
	series		= {Astérisque},
	number		= {223},
	publisher	= {Société mathématique de France},
	pages		= {59--101},
	doi		= {10.24033/ast.1122},
}

@InCollection{Fontaine1994:III,
	author		= {Jean-Marc Fontaine},
	title		= {Représentations \(p\)\=/adic semi-stables},
	date		= {1994},
	booktitle	= {Périodes \(p\)\=/adiques},
	booksubtitle	= {Séminaire de Bures, 1988},
	series		= {Astérisque},
	number		= {223},
	publisher	= {Société mathématique de France},
	pages		= {113--184},
	doi		= {10.24033/ast.1122},
}

@InCollection{Fontaine2003,
	author		= {Jean-Marc Fontaine},
	title		= {Presque \(\Cp\)\=/représentations},
	date		= {2003},
	booktitle	= {A collection of manuscripts written in
                           honour of Kazuya Kato on the occasion of
                           his fiftieth birthday},
	series		= {Documenta Mathematica Series},
	pages		= {285--385},
	doi		= {10.4171/dms/3/10},
}

@Article{Fontaine2020,
	author		= {Jean-Marc Fontaine},
	title		= {Almost \(\Cp\) Galois representations and
                           vector bundles},
	journaltitle	= {Tunisian Journal of Mathematics},
	date		= {2020},
	volume		= {2},
	number		= {3},
	pages		= {667--732},
	doi		= {10.2140/tunis.2020.2.667},
}

@InProceedings{FukayaKato2006,
	author		= {Takako Fukaya and Kazuya Kato},
	title		= {A formulation of conjectures on
                           \(p\)\=/adic zeta functions in
                           noncommutative Iwasawa theory},
	date		= {2006},
	booktitle	= {Proceedings of the St. Petersburg
                           Mathematical Society, Volume XII},
	series		= {American Mathematical Society Translation:
                           Series 2},
	number		= {219},
	publisher	= {American Mathematical Society},
	pages		= {1--85},
	doi		= {10.1090/trans2/219/01},
}

@Book{GabberRamero2003,
	author		= {Ofer Gabber and Lorenzo Ramero},
	title		= {Almost Ring Theory},
	date		= {2003},
	edition		= {1},
	series		= {Lecture Notes in Mathematics},
	publisher	= {Springer},
	pagetotal	= {vi+318},
	doi		= {10.1007/b10047},
}

@Unpublished{He2025,
	author		= {Tongmu He},
	title		= {The \(p\)\=/adic Galois cohomology of
                           valuation fields},
	date		= {2025},
	url		= {https://arxiv.org/abs/2502.14858},
}

@InProceedings{Hyodo1991,
	author		= {Osamu Hyodo},
	title		= {\(\operatorname{H}^1_g(K,V) = \operatorname{H}^1_{st}(K,V)\)},
	date		= {1991},
	booktitle	= {Proceedings of a Symposium on Arithmetic Geometry},
	editor		= {Kazuya Kato and Masato Kurihara and Takeshi Saito},
	publisher	= {University of Tokyo},
	pages		= {136--142},
}

@Article{IovitaZaharescu1999:1,
	author		= {Adrian Iovita and Alexandru Zaharescu},
	title		= {Galois theory of \(\mathbf{B}_\mathrm{dR}^+\)},
	journaltitle	= {Compositio Mathematica},
	date		= {1999},
	volume		= {117},
	number		= {1},
	pages		= {1--33},
	doi		= {10.1023/A:1000642625728},
	ids		= {IovitaZaharescu1999},
}

@Article{Jannsen1988,
	author		= {Uwe Jannsen},
	title		= {Continuous étale cohomology},
	journaltitle	= {Mathematische Annalen},
	date		= {1988},
	volume		= {280},
	number		= {2},
	pages		= {207--245},
	doi		= {10.1007/BF01456052},
}

@Article{Mazur1972,
	author		= {Barry Mazur},
	title		= {Rational points of abelian varieties with
                           values in towers of number fields},
	journaltitle	= {Inventiones mathematicae},
	date		= {1972},
	volume		= {18},
	number		= {3},
	pages		= {183--266},
	doi		= {10.1007/BF01389815},
}

@Book{NeukirchSchmidtWingberg2008,
	author		= {Jürgen Neukirch and Alexander Schmidt and Kay
                           Wingberg},
	title		= {Cohomology of Number Fields},
	date		= {2008},
	edition		= {2},
	series		= {Grundlehren der mathematischen
                           Wissenschaften},
	publisher	= {Springer},
	pagetotal	= {xv+826},
	doi		= {10.1007/978-3-540-37889-1},
}

@Article{Ohkubo2010,
	author		= {Shun Ohkubo},
	title		= {Galois theory of \(\BdR^+\) in the imperfect
                           residue field case},
	journaltitle	= {Journal of Number Theory},
	date		= {2010},
	volume		= {130},
	number		= {7},
	pages		= {1609--1641},
	doi		= {10.1016/j.jnt.2010.03.005},
}

@Article{Perrin-Riou2000,
	author		= {Bernadette Perrin-Riou},
	title		= {Représentations \(p\)\=/adiques et normes
                           universelles I. Le cas cristallin},
	journaltitle	= {Journal of the American Mathematical
                           Society},
	date		= {2000},
	volume		= {13},
	number		= {3},
	pages		= {533--551},
	doi		= {10.1090/S0894-0347-00-00329-5},
}

@Book{Perrin-Riou2001,
	author		= {Bernadette Perrin-Riou},
	title		= {Théorie d’Iwasawa des représentations
                           \(p\)\=/adiques semi-stables},
	date		= {2001},
	series		= {Mémoires de la Société Mathématique de
                           France},
	number		= {84},
	publisher	= {Société mathématique de France},
	pagetotal	= {vi+111},
	doi		= {10.24033/msmf.397},
}

@Unpublished{Ponsinet2024,
	author		= {Gautier Ponsinet},
	title		= {Bloch\--Kato groups over perfectoid fields
                           and Galois theory of \(p\)\=/adic periods},
	date		= {2024},
	url		= {http://arxiv.org/abs/2412.00227},
}

@Article{Ponsinet2025,
	author		= {Gautier Ponsinet},
	title		= {Universal norms and the Fargues\--Fontaine
                           curve},
	journaltitle	= {Mathematische Annalen},
	date		= {2025},
	volume		= {392},
	pages		= {2853--2912},
	doi		= {10.1007/s00208-025-03131-8},
}

@Collection{SGAVII:1,
	editor		= {Alexander Grothendieck and Michèle Raynaud
                           and Dock Sang Rim},
	title		= {Groupes de monodromie en géométrie
                           algébrique},
	date		= {1972},
	subtitle	= {Séminaire de géométrie algébrique du
                           Bois-Marie 1967--1969 (SGA 7 I)},
	volume		= {1},
	edition		= {1},
	series		= {Lecture Notes in Mathematics},
	publisher	= {Springer},
	doi		= {10.1007/BFb0068688},
}

@Book{Schneider2002,
	author		= {Peter Schneider},
	title		= {Nonarchimedean functional analysis},
	date		= {2002},
	edition		= {1},
	series		= {Springer Monographs in Mathematics},
	publisher	= {Springer},
	pages		= {vii+156},
	doi		= {10.1007/978-3-662-04728-6},
}

@Article{Scholze2012,
	author		= {Peter Scholze},
	title		= {Perfectoid spaces},
	journaltitle	= {Publications mathématiques de l'IHÉS},
	date		= {2012},
	volume		= {116},
	number		= {1},
	pages		= {245--313},
	doi		= {10.1007/s10240-012-0042-x},
}

@Article{Sen1969,
	author		= {Shankar Sen},
	title		= {On automorphisms of local fields},
	journaltitle	= {Annals of Mathematics},
	date		= {1969},
	volume		= {90},
	number		= {1},
	pages		= {33--46},
	doi		= {10.2307/1970680},
}

@Article{Sen1972,
	author		= {Shankar Sen},
	title		= {Ramification in \(p\)\=/adic Lie extensions},
	journaltitle	= {Inventiones mathematicae},
	date		= {1972},
	volume		= {17},
	number		= {1},
	pages		= {44--50},
	doi		= {10.1007/BF01390022},
}

@Article{Sen1973,
	author		= {Shankar Sen},
	title		= {Lie algebras of Galois groups arising from
                           Hodge\--Tate modules},
	journaltitle	= {Annals of Mathematics},
	date		= {1973},
	volume		= {97},
	number		= {1},
	pages		= {160--170},
	doi		= {10.2307/1970879},
}

@Book{Serre1994,
	author		= {Jean-Pierre Serre},
	title		= {Cohomologie galoisienne},
	date		= {1994},
	edition		= {5},
	series		= {Lecture Notes in Mathematics},
	number		= {5},
	publisher	= {Springer},
	pagetotal	= {ix+181},
	doi		= {10.1007/BFb0108758},
}

@InProceedings{Tate1967,
	author		= {John Tate},
	title		= {\(p\)\=/divisible groups},
	date		= {1967},
	booktitle	= {Proceedings of a Conference on Local
                           Fields},
	booksubtitle	= {NUFFIC Summer School held at Driebergen
                           (The Netherlands) in 1966},
	publisher	= {Springer},
	pages		= {158--183},
	doi		= {10.1007/978-3-642-87942-5_12},
}

@Article{Tate1976,
	author		= {John Tate},
	title		= {Relations between \(\operatorname{K}_2\) and Galois
                           cohomology},
	journaltitle	= {Inventiones mathematicae},
	date		= {1976},
	volume		= {36},
	number		= {1},
	pages		= {257--274},
	doi		= {10.1007/BF01390012},
}
\end{document}